\documentclass[12pt]{amsart}
\usepackage{amsfonts, amssymb, latexsym, amscd}
\usepackage{amsmath, amsthm}
\usepackage[mathscr]{eucal}
\usepackage{color}
\usepackage{graphicx}
\usepackage{enumerate}
\usepackage{esint}
\usepackage{verbatim}
\usepackage{MnSymbol}

\usepackage{mathtools}
\usepackage[colorlinks=true, pdfstartview=FitV, linkcolor=blue, citecolor=blue, urlcolor=blue,pagebackref=false]{hyperref}
\usepackage{microtype}

\definecolor{darkgreen}{rgb}{0,0.5,0}
\definecolor{darkblue}{rgb}{0.1,0.1,0.9}
\definecolor{darkred}{rgb}{0.9,0.1,0.1}

\theoremstyle{plain}
\newtheorem{thm}{Theorem}[section]
\newtheorem{cor}[thm]{Corollary}
\newtheorem{lem}[thm]{Lemma}
\newtheorem{prop}[thm]{Proposition}

\newtheoremstyle{named}{}{}{\itshape}{}{\bfseries}{.}{.5em}{\thmnote{#3's }#1}
\theoremstyle{named}

\newcounter{constnum}
\makeatletter
\def\const@nt#1{c_{#1}}
\def\c{\@ifnextchar[{\@with}{\@without}}
\def\@with[#1]{%
	\ifcsname constnum@#1\endcsname
	\else
		\stepcounter{constnum}%
		\expandafter\xdef\csname constnum@#1\endcsname{\theconstnum}%
	\fi
	\const@nt{\csname constnum@#1\endcsname}}
\def\@without{\stepcounter{constnum}c_{\theconstnum}}
\makeatother

\theoremstyle{definition}
\newtheorem{definition}[thm]{Definition}

\newtheorem{example}[thm]{Example}
\newtheorem{conjecture}[thm]{Conjecture}

\newcommand{\mscr}[1]{\mathscr{#1}}

\newcommand{\PP}{\mathbb{P}}

\newcommand{\RR}{\mathbb{R}}

\newcommand{\ve}{\varepsilon}

\DeclareMathOperator{\limsups}{limsup*}
\DeclareMathOperator{\liminfs}{liminf*}

\newcommand{\ga}{\gamma}
\newcommand{\de}{\delta}

\newcommand{\la}{\lambda}

\newcommand{\om}{\omega}

\newcommand{\thet}{\theta}
\newcommand{\aint}{\fint}

\newcommand{\Ga}{\Gamma}
\newcommand{\De}{\Delta}

\newcommand{\La}{\Lambda}

\newcommand{\Om}{\Omega}

\newcommand{\tr}{\mbox{tr}} 
\newcommand{\rst}[1]{\ensuremath{{\mathbin\upharpoonright}%
\raise-.5ex\hbox{$#1$}}} 

\newcommand{\calS}{\mathcal{S}}
\newcommand{\bbR}{\mathbb{R}}

\newcounter{gscan}
\newcounter{bwscan}
\newcounter{cscan}
\newcounter{hscan}
\newcounter{fscan}
\newcounter{pscan}
\newcounter{sscan}
\newcounter{iscan}
\newcounter{rscan}
\newcounter{rrscan}
\newcounter{fpscan}

\renewcommand{\therscan}{R\arabic{rscan}}

\renewcommand{\tilde}{\widetilde}

\numberwithin{equation}{section}

\begin{document}
\title[Stochastic Homogenization]{Stochastic Homogenization for Reaction-Diffusion Equations} 
\author{Jessica Lin}
\address{McGill University\\
Department of Mathematics\\
Montreal, QC H3A 0B9, Canada}
\email[Jessica Lin]{jessica.lin@mcgill.ca}

\author{Andrej Zlato\v{s}}
\address{Department of Mathematics\\
University of California San Diego\\
La Jolla, CA 92130, USA}
\email[Andrej Zlatos]{zlatos@ucsd.edu}


\begin{abstract}
In the present paper we study stochastic homogenization for reaction-diffusion equations with stationary ergodic reactions, although some of our results are new even for periodic reactions.  We first show that under suitable hypotheses, initially localized solutions to the PDE asymptotically become approximate characteristic functions of a ballistically expanding Wulff shape.  The next crucial component is the proper definition of relevant front speeds and subsequent establishment of their existence.  We achieve the latter by finding a new relation between the front speeds  and 
the Wulff shape, provided the Wulff shape does not have corners.  Once front speeds are proved to exist in all directions, by the above means or otherwise, we are able to obtain general stochastic homogenization results, showing that large space-time evolution of solutions to the PDE is governed by a simple deterministic Hamilton-Jacobi equation whose Hamiltonian is given by these front speeds.  We primarily consider the case of non-negative reactions but we also extend our results to the more general PDE $u_{t}= F(D^2 u,\nabla u,u,x,\om)$ as long as its solutions satisfy some basic hypotheses including positive lower and upper bounds on spreading speeds in all directions and a sub-ballistic bound on the width of the transition zone between the two equilibria of the PDE.
\end{abstract}

\subjclass[2010]{35K57, 35B27, 35B40, 60H25}

\keywords{stochastic homogenization, reaction-diffusion equations, front speeds, spreading speeds, front propagation, Wulff shape}

\maketitle

\section{Introduction and Main Results}

\subsection{Background and informal discussion of the main results}
The primary motivation for this work is the understanding of long-time behavior of solutions to heterogeneous reaction-diffusion equations in random media. Specifically, we are interested in equations of the form 
\begin{equation}\label{genhomeq}
u_{t}=\De u+f(x,u,\om) \qquad\text{on } (0, \infty)\times \RR^{d},
\end{equation}
with a given initial condition $u(0, \cdot)$ (or, more generally, $u(0, \cdot,\om)$).  The argument $\om$ is an element of some probability space $(\Om, \mathcal{F}, \PP)$ which models the random environment via a random (stationary ergodic in the spatial variable $x$) nonlinear {\it reaction function} $f(\cdot, \cdot, \om)$.  We are primarily concerned with \emph{ignition} and {\it monostable} reactions here (see Subsection \ref{ss.hypmr} for the precise hypotheses), which model phenomena such as combustion, chemical kinetics, or population dynamics, with $u$ representing  temperature, concentration of a reactant, or density of a species.  Nevertheless, several of our results can be generalized to other types of reactions (including {\it bistable} and mixed types), as well as to other ``phase transition'' processes modeled by the more general PDE 
\begin{equation} \label{1.10}
u_{t}= F(D^2 u,\nabla u,u,x,\om)  \qquad\text{on  $(0,\infty)\times\bbR^d$},
\end{equation}
(see Subsection \ref{S.1.3} below).
 This includes, for instance, some (viscous) Hamilton-Jacobi equations with possibly non-convex or non-coercive Hamiltonians, or the porous medium equation.

These equations are frequently used in modeling invasions of one equilibrium state of a physical process  by another, such as forest fires (burned area invades unburned regions) or spreading of invasive species.   It is standard to let these equilibria  be $u^-\equiv 0$ and $u^+\equiv 1$, in which case one considers $f(\cdot,0,\cdot)\equiv f(\cdot,1,\cdot)\equiv 0$ and solutions $0\le u\le 1$ (the latter being guaranteed by $0\le u(0,\cdot)\le 1$). However, our results can again be extended to more general situations, including with $\om$-dependent equilibria $u^-(\cdot,\om)<u^+(\cdot,\om)$.

When studying long-time (and thus also large-scale) evolution of solutions to \eqref{genhomeq}, it is natural to consider the rescaled functions 
\begin{equation}\label{eq:rescale}
u^{\ve}(t,x, \om):=u\left(\frac{t}{\ve}, \frac{x}{\ve}, \om\right),
\end{equation}
with $\ve>0$ small.
This turns \eqref{genhomeq} into
\begin{equation}\label{eq:homeq}
u^{\ve}_{t}=\ve\De u^{\ve}+\frac{1}{\ve}f\left(\frac{x}{\ve}, u^{\ve}, \om\right) \qquad \text{on $(0, \infty)\times \RR^{d}$},
\end{equation}
making the scale of the heterogeneities in $f$ microscopic.  
One might therefore hope that in the limit $\ve\to 0$,  we can observe an effective homogeneous deterministic behavior of solutions. Specifically, if one solves \eqref{eq:homeq} with an initial condition that is independent of $\ve,\om$ (or at least converges to an $\om$-independent limit as $\ve\to 0$), then the solutions $u^\ve$ also converge to some $\om$-independent function $\overline u$ that solves a PDE whose coefficients do not depend on $x$ or $\om$.  This is the principal goal of the theory of homogenization of \eqref{genhomeq} and other PDE.

The first question to answer here is how do the limiting solutions $\overline u$ look like, and which PDE (if any) do they solve.  In many models, the limiting solutions have the same (or even better) regularity as the original ones, and the homogenized PDE are of the same type as the original PDE. This is not the case for  \eqref{genhomeq}.  As one can notice from observation of the physical processes modeled by the PDE, the width of the transition zone between the two equilibria (i.e., the distance of points with $u(t,x)=\eta$ from those with $u(t,x)=1-\eta$, for small $\eta>0$), is frequently uniformly bounded in time.  Indeed, as fires spread through forests, the scale of the burned region grows roughly linearly in time but the actively burning areas are typically confined to  neighborhoods of time-dependent curves whose widths are bounded uniformly in time.  If the solutions to \eqref{genhomeq} exhibit the same {\it bounded width} behavior (see Theorem \ref{thm:bw} below), then the scaling \eqref{eq:rescale} necessarily requires the limiting solutions (if they exist) to only take values 0 and 1, and thus be characteristic functions of time-dependent subsets of $\bbR^d$.

Of course, if general solutions become characteristic functions of sets in the $\ve\to 0$ limit, it is natural to consider initial conditions that are also (approximate) characteristic functions.  Homogenization for \eqref{genhomeq} should therefore involve initial conditions satisfying
\[
\lim_{\ve\to 0}  u^\ve(0,\cdot,\om)= \kappa \chi_{A}
\]
in some sense,  with $\chi_{A}$ the characteristic function of some given initial set $A\subseteq \bbR^d$ and $\kappa\le 1$ close to 1, and the corresponding solutions to \eqref{eq:homeq} should  then  for almost all $\om$ have the limit
\begin{equation} \label{1.20}
\lim_{\ve\to 0}  u^\ve(\cdot,\cdot,\om)=  \chi_{\Theta^A} \qquad (=\overline u)
\end{equation}
in some sense, with $\Theta^A\subseteq(0,\infty)\times\bbR^d$  some $\om$-independent set.  Of course, in that case $\overline u$ cannot solve a second-order PDE like \eqref{genhomeq}.  Indeed, as our main results show, the homogenized solutions will instead solve a first-order Hamilton-Jacobi PDE (specifically, \eqref{eq:hjeq} below) in the viscosity sense, with the set $\Theta^A_t:=\{x\in\bbR^d\,\mid\,(t,x)\in \Theta^A\}$ expanding at any point of its boundary with normal velocity that depends on the normal vector at that point but not at the point itself (with appropriate modifications when  $\partial\Theta^A_t$ does not have normal vectors everywhere).  We note that while sometimes the homogenized solutions also satisfy a (possibly non-isotropic) Huygens principle (see Theorem \ref{T.1.4}(iv) below), this is not always the case.

The reason for the propagation velocity being only dependent on the normal vector is that if $\partial\Theta^A_t$ has a tangent with outer normal $e\in\mathbb S^{d-1}$ at some point, then after rescaling from \eqref{eq:homeq} back to \eqref{genhomeq}, $\Theta^A_t$ will be close to a half-space with outer normal $e$ on a ball of size $O(\frac 1\ve)$ centered at the ``rescaled'' point.  Its propagation speed should therefore be determined by the speed of propagation of solutions starting from (approximate) characteristic functions of half-spaces with normal vector $e$ (i.e., {\it front-like initial data} oriented in direction $e$).  If such (deterministic) {\it front speed} indeed exists, it must be independent of which half-space with normal $e$ we consider due to stationarity and ergodicity of $f$.

Existence of the front speeds for general stationary ergodic reactions in several dimensions $d\ge 2$ is, however, a non-trivial question, and it is the main reason for a dearth of homogenization results in this setting.  (For the case $d=1$, $\Theta^A_t$ is typically an interval and hence its boundary has a trivial geometry, see \cite{nolryz, andrejhom, freidgart, freidbook, nadinrandom, bernadin1D} and the references therein.)  In part due to this, the only such result in dimensions $d\geq 2$ prior to the present paper appears to be homogenization for stationary ergodic   {\it KPP reactions} by Lions and Souganidis \cite[Theorem 9.3]{plltakisvisc}.   
We note that KPP reactions are a subclass of monostable reactions, satisfying $f(x,u,\om)\le f_u(x,0,\om)u$ for all $(x,u,\om)\in \bbR^d\times[0,1]\times\Om$, and were first studied in the one-dimensional homogeneous setting by Kolmogorov, Petrovskii, and Piskunov \cite{KPP} as well as by Fisher \cite{Fisher}.  Crucially, their properties allow one to study them via their linearization at $u=0$, as the two dynamics typically agree in the leading order. This was at the core of the Lions-Souganidis approach, who perform the Hopf-Cole transformation $v^\ve:=\ln u^\ve$ to convert the problem of stochastic homogenization for reaction-diffusion equations with KPP reactions into the problem of stochastic homogenization for viscous Hamilton-Jacobi equations.  

In the case of periodic (in $x\in\bbR^d$) reactions, it is known from the works of Xin \cite{xinflame} and Berestycki and Hamel \cite{berhamel}  that front speeds indeed do exist for very general ignition and monostable reactions.  These are obtained after finding the corresponding {\it pulsating front} solutions to \eqref{genhomeq} in direction $e$, which are of the form $u(t,x)=U(x\cdot e -ct,x)$, with $U$ periodic in the second argument and satisfying the boundary conditions $\lim_{s\to -\infty} U(s,x)=1$ and $\lim_{s\to \infty} U(s,x)=0$ uniformly in $x$.  Here both the front profile $U$ and front speed $c$ are unknown, and the speed of propagation of typical solutions whose initial data vanish on a half-space with inner normal $e$ is the unique (for ignition $f$) or minimal (for monostable $f$) pulsating front speed in direction $e$.  

As these results are more than 15 years old, one might think that homogenization has already been proved for general periodic reactions.  Nevertheless, the step from existence of pulsating fronts to homogenization is far from trivial.  In fact, other than the general stationary ergodic KPP result in \cite{plltakisvisc}, homogenization has previously only been proved for monostable periodic reactions and convex initial sets $A$ with smooth boundaries, in a recent work of  Alfaro and Giletti \cite{alfarogiletti}.  Hence, our treatment of general stationary ergodic reactions also establishes new results for periodic reactions, and we will  obtain as a byproduct not only the result of Alfaro-Giletti (under slightly weaker hypotheses on $f$ and without the requirement of smoothness of $\partial A$) but also  full homogenization for periodic ignition reactions (i.e., for any $A$).  In fact, we prove homogenization whenever \eqref{genhomeq} has a {\it deterministic exclusive front speed} (see Definition \ref{D.1.5} below) in each direction $e$, which is the case for periodic ignition reactions in any dimension.

Homogenization results have, however, been obtained in the case of \eqref{genhomeq} with homogeneous reactions $f(x,u)\equiv f(u)$ and  {\it periodic linear terms}.  In \cite{barlestakis}, Barles and Souganidis develop their theory of generalized front propagation,
which essentially handles crystal-growth-like models where normal growth speeds are given but the boundary of the crystal may be quite rough.  One of its applications is the proof of homogenization in the case of homogeneous  bistable reactions 
with $\int_0^1 f(u)du>0$ and spatially periodic linear terms, under the hypothesis of existence of pulsating fronts in all directions for this model.  ({\it Bistable reactions} have $f(u)<0$  for all small $u>0$, and one of us has in fact shown that pulsating fronts need not exist and homogenization need not hold for heterogeneous bistable reactions, even periodic ones in one dimension $d=1$  \cite{ZlaBist}.)  An important advantage of bistable reactions is that  solutions with small enough initial data converge to 0, which means that {\it if} front speeds do exist, they are automatically the exclusive front speeds from Definition \ref{D.1.5}.  Since  this  convergence does not hold for non-negative reactions, which we  consider in the present paper (and in particular, Hypothesis (H4) from \cite{barlestakis} is not satisfied for them), we 
need to introduce the new concept of exclusive front speeds here.

Homogenization for the $\ve\to 0$ limit of \eqref{eq:homeq} with a homogeneous KPP reaction $f(u)$ and also time-dependent advection $V(t, x, \ve^{-\alpha}t, \ve^{-\alpha}x)$, periodic in the last two arguments and with $\alpha\in(0,1]$, was studied by Majda and Souganidis \cite{majtakis}.  They proved that the limit is 0 and 1 on the sets $\{Z<0\}$ and ${\rm int}\{Z=0\}$, respectively, where the function $Z\le 0$ solves some Hamilton-Jacobi equation on the set $\{Z<0\}$.
The advection field becomes $V(\ve t, \ve x, \ve^{1-\alpha}t, \ve^{1-\alpha}x)$ in the scaling of \eqref{genhomeq}, so if it is constant in the first two arguments and $\alpha=1$, then one obtains \eqref{genhomeq} with a homogeneous KPP reaction  and an $\ve$-independent space-time periodic advection $V(t,x)$.  That is, the homogenization limit is then also the large space-time limit for \eqref{genhomeq} that we are studying here for general stationary ergodic reactions.  
We note that in this ``$\ve$-independent''  case  \cite{majtakis} also identifies a Hamilton-Jacobi equation like  \eqref{eq:hjeq} below that governs the evolution of the homogenization limit, provided $V$ is divergence-free.   Although the relationship to the relevant pulsating fronts  (whose existence had not yet been established at that point) is not investigated in \cite{majtakis}, one can conclude existence of pulsating front speeds for periodic incompressible advections and homogeneous KPP reactions from this result.

While the periodic results for non-KPP reactions are crucially dependent on the existence of pulsating fronts, it is not clear whether some analogous solutions exist for general stationary ergodic reactions.  (Also, since it follows from the results of one of us that no reasonable definition exists that would yield such solutions for general heterogeneous reactions in dimensions $d\ge 2$ \cite{andrejbd}, one can in general only hope for their existence for  almost all $\om\in\Om$ at best.)  In this work, we are therefore left with the task of defining and identifying front speeds {\it without} the existence of some special front solutions.

We achieve this goal  by first defining the front speeds $c^*(e)$ via tracking solutions evolving from (approximate) characteristic functions of relevant half-spaces (see Definition \ref{D.1.3}). Then, under appropriate hypotheses, we identify these front speeds by relating them to another family of speeds that is relevant to the question at hand ---  the  {\it spreading speeds} $w(e)$ (see Remark 1 after Definition \ref{D.1.2}).  
These are the asymptotic speeds of propagation in different directions $e\in\mathbb S^{d-1}$ of solutions starting from compactly supported initial data (with large enough supports so that propagation happens, i.e.,  $\lim_{t\to\infty} u(t,x)=1$ locally uniformly on $\bbR^d$).  In the case of periodic reactions, the spreading speeds are known to exist and 
can be found from the (unique/minimal) pulsating front speeds (which coincide  with our front speeds $c^*(e)$ in the periodic case) via the {\it Freidlin-G\" artner formula}
\begin{equation}\label{eq:fg}
w(e)=\inf_{e'\in\mathbb S^{d-1} \,\&\, e'\cdot e>0} \frac{c^{*}(e')}{e'\cdot e}.
\end{equation}
This was obtained by Freidlin and G\" artner in \cite{freidgart} (see also \cite{morefreid}) for periodic KPP reactions, and was extended to periodic monostable and ignition reactions by Weinberger \cite{wein} and Rossi \cite{rossi}.  Of course, existence of spreading speeds in all directions and the comparison principle show that after scaling down by $t$, general solutions with (large enough) compactly supported initial data converge to the characteristic function of the  {\it Wulff shape}
\begin{equation}\label{eq:wulffdef}
\mathcal{S}:=\left\{se \,|\, e\in \mathbb{S}^{d-1} \text{ and }  s\in\left[0,w(e)\right) \right\} \qquad (\subseteq\bbR^d)
\end{equation}
as $t\to\infty$.

In the case of stationary ergodic reactions, however, we  reverse this process ---  we start by identifying the spreading speeds rather than the front speeds.  This should be an easier task in non-periodic media as it involves solutions evolving from compactly supported data rather than from characteristic functions of half-spaces, contrasting with periodic media, where the pulsating front ansatz turns \eqref{genhomeq} into a degenerate elliptic PDE on the quasi-one-dimensional domain $\bbR\times\mathbb T^d$. We use here the subadditive ergodic theorem (Theorem \ref{thm:set}) together with results guaranteeing that solutions have bounded widths (Theorem~\ref{thm:bw}) to obtain existence of the spreading speeds and the (convex) Wulff shape $\mathcal S$ under appropriate hypotheses (see Theorem \ref{T.1.4}(i) and Theorem \ref{T.1.8}(i)).

Once this is achieved, consider a solution  $u$ of \eqref{genhomeq} evolving from compactly supported initial data and let $y\in\partial\mathcal S$ be any point at which $\partial\mathcal S$ has a tangent hyperplane with some unit outer normal vector $e$.  Then for all large $t$,  the solution $u$ is $o(t)$ close to the characteristic function of the half-space $\{x\cdot e<ty\cdot e\}$ in an $o(t)$ neighborhood of the point $ty\in\partial (t\mathcal S)$.  If we can obtain good enough bounds on the difference of $u$ and the solution $u'$ starting from the (approximate)  characteristic function of this half-space, valid for a time $t'$ during which $u$ expands from $\sim \chi_{t\mathcal S}$ to $\sim \chi_{(t+t')\mathcal S}$, then we would show that near the point $ty$ and on the time interval $[t,t']$, the solution $u'$ is close to the characteristic function of the above half-space moving with speed $w(\frac y{|y|}) \frac y{|y|} \cdot e$ in its normal direction $e$.  This and convexity of $\mathcal S$ would then yield the {\it inverse Freidlin-G\" artner formula}
\begin{equation}\label{eq:fsdef}  
c^{*}(e)=\sup_{e'\in \mathbb{S}^{d-1}} w(e')e'\cdot e.
\end{equation}
In reality, the above argument has to contend with (otherwise uncontrollable) $o(t)$ errors, which means that the $o(t)$ above as well as $t'$ would instead have to be of size $O(t)$, thus causing additional difficulties.   Nevertheless, after carefully calibrating their mutual proportionality constants, we will be able to execute this approach in a  rigorous fashion and obtain existence of front speeds in all directions $e$ that are outer normals to $\mathcal S$ (see Theorem \ref{T.1.4}(ii)).  If $\mathcal S$ has {\it no corners}, then this includes all $e\in\mathbb S^{d-1}$.

We note that Theorem \ref{thm:bw} is restricted to dimensions $d\le 3$, and this limitation is sharp. As a result,  our proof of existence of the Wulff shape for stationary ergodic reactions only applies in this setting.  Nevertheless, if one can prove existence of a Wulff shape in another setting by other means, our method provides existence of front speeds in all its normal directions.  Similarly, as we discuss below, existence of (exclusive) front speeds in all directions is itself also sufficient for our main homogenization results to hold.  This, in particular, is the reason why we are able to establish homogenization for periodic reactions in any dimension.

It is  remarkable that despite the long history of the subject, formula \eqref{eq:fsdef} for normal vectors to $\partial\mathcal S$ appears to be new even in the periodic setting. In fact, our search of the literature for such a formula while writing this paper has only yielded the works of Soravia \cite{soravia} and of Osher and Merriman \cite{oshermerriman}, in which they primarily study the growth of crystals with an {\it a priori given} growth speed $c^*(e)$ in each normal direction $e$ to the crystal's surface (so no reaction-diffusion equations). They show emergence of a Wulff shape $\mathcal S$ from \eqref{eq:wulffdef} for this growth, satisfying \eqref{eq:fg}, and also find that if the function $c^*(\frac y{|y|}) |y|$ is convex, then \eqref{eq:fsdef} holds as well.   Additionally, Osher and Merriman  show that any initial crystal $A\subseteq\bbR^d$ grows in time $t$ into $A+t\mathcal S$ (for each $t>0$) when $c^*(\frac y{|y|}) |y|$ is convex and, conversely,  they observe that if  any initial crystal $A$  grows as $A+t\mathcal S$ for some convex set $\mathcal S$, then the normal speed $c^*$ of this growth satisfies \eqref{eq:fsdef} and $c^*(\frac y{|y|}) |y|$ is convex.  Of course, since for each unit vector $e$ one can choose an initial crystal $A$ whose boundary contains (an open subset of) the hyperplane $\{x\cdot e<0\}$,  \eqref{eq:fsdef} is immediate from the growth being $A+t\mathcal S$.  In contrast, here we only use emergence of the Wulff shape for solutions with compactly supported initial data, which in the large space-time scaling limit corresponds to the Osher-Merriman growth rule $\Theta^A_t=A+t\mathcal S$ for only the set $A=\{0\}$. 

We also note that the Osher-Merriman growth rule (which is the non-isotropic Huygens principle) is the model currently used by Canadian Forest Fire Prediction System. The relevant model, in which $\mathcal S$ is an ellipse whose parameters are determined from environmental factors such as wind speeds, is called Richards equation \cite{richards}.  Our main results for reaction-diffusion equations, namely Theorems \ref{T.1.4}(ii--iv) and \ref{T.1.6}(ii), justify this approach for stationary ergodic media when the Wulff shape is indeed an ellipse (or, more generally, when it has no corners).

It is important to stress here that  we prove \eqref{eq:fsdef} for \eqref{genhomeq} 
only for vectors  that are unit outer normals to $\mathcal S$.  More generally and similarly to \cite{oshermerriman}, \eqref{eq:fsdef} holds for  all $e\in\mathbb S^{d-1}$ precisely when the function $c^*(\frac y{|y|})|y|$ is defined everywhere (we let it be 0 at $y=0$) and is convex, in which case we also recover the  Osher-Merriman growth rule for \eqref{genhomeq} in the asymptotic limit.  This is in fact the case for general KPP reactions \cite{LinZla2}. However, it follows from the work of Caffarelli, Lee, and Mellet \cite[Theorems 2.6 and A.2]{cafleemel} that there exist periodic  ignition reactions in two dimensions for which $c^*(\frac y{|y|}) |y|$ is not convex  --- in which case our results show that the corresponding Wulff shapes must have corners. On the other hand \eqref{eq:fg} always holds (see Theorem \ref{T.1.9}). 

Nevertheless, as long as existence of (exclusive) front speeds in all directions is known --- even if \eqref{eq:fsdef} does not hold for all $e$ and hence $c^*(\frac y{|y|})|y|$ is not convex --- we are still able to obtain homogenization results for \eqref{genhomeq}. 
To achieve this, we  will show that the lower and upper limits of $ u^\ve$ as $\ve\to 0$ are deterministic viscosity  super- and sub-solutions, respectively, to the  Hamilton-Jacobi PDE
\begin{equation}\label{eq:hjeq}
\overline{u}_{t}-c^{*}\left(-\frac{\nabla \overline{u}}{|\nabla \overline{u}|}\right)|\nabla \overline{u}|=0 \qquad\text{on } (0, \infty)\times \RR^{d}.
\end{equation}
(We note that our method of showing this
shares some elements with 
that employed by Barles and Souganidis in their theory of generalized front propagation \cite{barlestakis}.)
The results of Soravia \cite{soravia} and Barles, Soner, and Souganidis \cite{bss} on uniqueness of viscosity solutions to such equations can then be used to show that the super- and sub-solution in fact coincide and hence the limit $\overline u:= \lim_{\ve\to 0} u^\ve$ exists and solves \eqref{eq:hjeq}.  This eventually yields our main homogenization results for \eqref{genhomeq}, Theorems \ref{T.1.4}(iii) and \ref{T.1.6}(ii).  On the other hand, if \eqref{eq:fsdef} holds (hence we have the Osher-Merriman growth rule) and the initial set $A$ is convex, then we can obtain homogenization by a simpler method in Theorem \ref{T.1.4}(iv) without having to resort to the viscosity solutions theory for \eqref{eq:hjeq}.

We should also mention here that existence of a Wulff shape for \eqref{genhomeq} with a homogeneous KPP reaction $f(u)$ and space-time stationary ergodic divergence-free advection $V(t,x, \om)$ satisfying a finite moment condition has been proved by Nolen and Xin \cite{nolenxin}.  
They did not study homogenization or front speeds for that model --- 
indeed, what they call front speeds are actually our spreading speeds $w(e)$.

Finally, let us note that almost everything here applies to the more general PDE \eqref{1.10} under some basic hypotheses, and we collect the corresponding main results in Theorem \ref{T.1.8} in Subsection \ref{S.1.3} below.  The results for \eqref{genhomeq} are contained in Subsection \ref{ss.hypmr}.

\subsection{Relation to homogenization results for Hamilton-Jacobi equations}\label{ss.hj}

There is a vast literature on homogenization for Hamilton-Jacobi and viscous Hamilton-Jacobi equations, such as
\[
u_t + H(\nabla u,x,\om) = \tr(A(x,\om)D^2u) 
\]
with a coercive Hamiltonian $H$ and a positive semi-definite matrix $A$, and we refer to \cite{arms-soug, armstrongtran} and references therein for an overview of the subject.  While we do not study such equations here, let us review the similarities and differences between these results and ours.  

A typical Hamilton-Jacobi homogenization result considers continuous initial data $u^\ve(0,\cdot,\om)=g(x)$ in the rescaled equation 
\begin{equation} \label{1.346}
u^{\ve}_{t}+H\left(\nabla u^\ve,\frac x\ve,\om \right) = \ve\tr \left(A\left(\frac x\ve, \om \right)D^2u^\ve\right)
\end{equation}
for
\[
u^{\ve}(t,x, \om):= \ve u\left(\frac{t}{\ve}, \frac{x}{\ve}, \om\right).
\]
This scaling differs from the natural scaling \eqref{eq:rescale} in the reaction-diffusion case by a factor of $\ve$.  So while in the Hamilton-Jacobi case any interval of values for the unscaled equation is compressed to a single value in the $\ve\to 0$ limit, this is not so in the reaction-diffusion case, where one needs to also show that  the width of the transition zone between the regions where $u^\ve\approx 0$ and $u^\ve\approx 1$ becomes infinitesimally small in the $\ve\to 0$ limit, at least almost surely.

This is not just a technical issue as was mentioned above: due to it, homogenization need not happen for bistable reactions (even periodic ones in one dimension  \cite{ZlaBist}, although we do prove homogenization for general periodic ignition and monostable reactions in any dimension in Theorem \ref{T.1.6a}), and there exist stationary ergodic ignition reactions in dimensions $d\ge 4$ (even i.i.d.~ones) such that the width of the transition zone between $u^\ve\approx 0$ and $u^\ve\approx 1$ is almost surely unbounded in time   \cite{andrejbd}. This is also  why we need to define and establish/assume existence of exclusive front speeds (rather than just of front speeds) in the strongest versions of our homogenization results.

A second important difference is due to the relationship of the respective original and homogenized PDE.
In the $\ve\to 0$ limit, the second order term in \eqref{1.346} disappears and $H$ is replaced by another (homogenized) Hamiltonian  $\overline H(\nabla u)$ (under appropriate hypotheses).  
So the limiting equation is again a (non-viscous) Hamilton-Jacobi equation, and the almost surely deterministic limit $\lim_{\ve\to 0} u^\ve$ remains continuous if $g$ is.  
In the case of reaction-diffusion equations, the homogenized solutions are instead discontinuous viscosity solutions to the  Hamilton-Jacobi equation \eqref{eq:hjeq} (again under appropriate hypotheses).  Moreover, unlike in the Hamilton-Jacobi case, the main term in \eqref{eq:hjeq} does not have a counterpart in the original PDE.
This, in particular, makes it difficult to obtain counterparts of various results in the Hamilton-Jacobi case, where the assumption of convexity of $H$   --- or at least convexity of its sub-level sets \cite{arms-soug}  --- in the first argument has been central (of course, then $\overline H$ has the same property). Indeed, we do not know of a comparable assumption on the reaction $f$ that would simplify the task of proving homogenization for \eqref{genhomeq}, except in the KPP case (see below).  In fact, the abovementioned result from \cite{cafleemel} shows that even for the simplest periodic ignition reactions  in two dimension, of the form $f(x_1,x_2)=f(x_1)$, while existence of front speeds in all directions is known, the Hamiltonian $c^*(-\frac p{|p|})|p|$ in \eqref{eq:hjeq} may have non-convex sub-level sets (and hence be non-convex, too).   
We note that relatively few positive homogenization results have been obtained for Hamilton-Jacobi  and viscous Hamilton-Jacobi equations with non-convex Hamiltonians, see \cite{armstrongtranyu2, kosdavini, gao1, armstrongtranyu1, gao2, card-soug, scottpierre}.  Majority of these require fairly restrictive structural assumptions on the Hamiltonian (e.g., $H(p,x, \om)=H(p)+V(x, \om)$) and/or hold only in 1-dimension, an exception being \cite{scottpierre}, where the authors consider $\alpha$-homogeneous Hamiltonians with $\alpha\ge 1$ that are also i.i.d.~in space.  In fact, counterexamples to stochastic homogenization of Hamilton-Jacobi equations with non-convex Hamiltonians have recently been obtained in \cite{bruno,willtakis}.

As mentioned above, our approach to this problem starts with the proof of existence of the Wulff shape for \eqref{genhomeq} in Section \ref{sec:sspeed}.  The method to achieve this goes back to the study of first passage percolation and similar ideas have also been recently employed in the study of homogenization for Hamilton-Jacobi equations \cite{davini, arms-soug, armstrongtran}, although the reaction-diffusion case is somewhat more involved on account of the need for appropriate bounds on the width of the transition zone discussed above.  The analog of the Wulff shape in the Hamilton-Jacobi case are the asymptotics of the solutions to the so-called metric problem.  These solutions can be shown to be approximate super-correctors for the PDE, and if the Hamiltonian in \eqref{1.346} is convex --- or at least has convex sub-level sets --- in the first argument, then their negatives will also be approximate sub-correctors.  This and appropriate comparison arguments can be used to show that the deterministic limit
\[
\overline H(p):=-\lim_{\delta\to 0} \delta v^\delta (0,\om;p)
\]
exists almost surely for the unique solutions to the macroscopic problem
\[
\delta v^\delta + H(p+\nabla v^\delta,x,\om)=0
\]
(for any fixed $p\in\bbR^d$) \cite{arms-soug, armstrongtran}. Existence of this limit is similar in spirit to existence of our (exclusive) front speeds.  It ultimately yields 
homogenization 
in a standard way via the perturbed test function method introduced by Evans \cite{evanshom}.  We note that in the reaction-diffusion case, this last step is again more involved, needing both existence of exclusive front speeds and the use of the theory of discontinuous viscosity solutions to \eqref{eq:hjeq}.
We perform it in Section~\ref{s:fstohom}.


The above approach fails for general non-convex Hamiltonians, and likely does not have an analog for reaction-diffusion equations.  Instead, we show in Section \ref{s:fs} that the Wulff shape itself becomes a front-like solution (with some asymptotic speed $c^{*}(e)$) in direction $e$ near the point $ty$ (asymptotically as $t\to\infty$) whenever $y$ is a non-corner boundary point of the Wulff shape with  outer normal $e$.  
We note that the spreading speeds, which define the Wulff shape, are essentially the convex dual to the front speeds if the latter exist. Therefore, $y$ above is not a corner of the Wulff shape precisely when positive multiples of $e$ are extreme points of the level sets of $c^*(-\frac p{|p|})|p|$.  This relates to  \cite{card-soug}, where it was proved that if homogenization holds
in probability 
for a fairly general (viscous) Hamilton-Jacobi equation, then correctors exist almost surely for any extreme point of any sub-level set of the effective Hamiltonian.  However, this approach {\it needs to  assume homogenization} in probability (except in the case of isotropic media, under some additional structural assumptions), while our approach via the Wulff shape does not.

This idea can in fact also be applied to Hamilton-Jacobi equations with non-convex Hamiltonians, and one can show that appropriate limits of solutions to a version of the metric problem will almost surely be the desired correctors, provided the level sets of these solutions do not have asymptotic corners in the relevant direction \cite{ZlaHJ}.  Therefore, if the Wulff shape or asymptotic level sets of the solutions to the metric problem have no corners (in which case we also find that $c^*(-\frac p{|p|})|p|$ resp.~$\overline H(p)$ have convex sub-level sets), then full homogenization follows.  The no-corner question seems not an easy one to answer in general, but the answer is always affirmative for isotropic media, when those shapes are just spheres (see Corollary \ref{C.1.7} below and \cite{ZlaHJ}). 

The one exception where neither of the above two difficulties applies are KPP reactions.  Indeed, in this case the dynamics of solutions is determined to the leading order by values arbitrarily close to 0, so the non-compression of values as $\ve\to 0$ does not cause a significant hurdle.  As mentioned above, this also allows one to use the Hopf-Cole transformation $v^\ve:=\ln u^\ve$ to essentially convert the reaction-diffusion PDE into a viscous Hamilton-Jacobi PDE with a convex Hamiltonian, and use results for such equations to obtain homogenization \cite{plltakisvisc}  (a more direct proof of homogenization in the KPP setting will be provided in \cite{LinZla2}).  This is the reason why stochastic homogenization for \eqref{genhomeq} in several dimensions had previously been proved only for KPP reactions.
 
Moreover, for ignition reactions in several dimensions, not even periodic homogenization had been known to hold prior to our work.
Nevertheless, 
we hope that one should be able to overcome the issue of potential corners of the Wulff shape and prove a general homogenization result at least for ignition reactions in dimensions $d\le 3$ that are i.i.d.~in space.  Such a result was proved for Hamilton-Jacobi equations in any dimension in \cite{scottpierre}, with Hamiltonians that are $\alpha$-homogeneous in the first argument with $\alpha\ge 1$ and i.i.d.~in space.  For general non-KPP reaction-diffusion equations, however, this  remains an open question.

\subsection{Hypotheses and main results for \eqref{genhomeq}}\label{ss.hypmr}
%
Let  $(\Om, \mathcal{F}, \PP)$ be a probability space that  is endowed with a group of measure-preserving transformations $\left\{\mscr{T}_{y} :\Om\rightarrow \Om \right\}_{y\in \RR^{d}}$  such that
 \begin{equation*}
 \mscr{T}_{y}\circ \mscr{T}_{z} =\mscr{T}_{y+z}
 \end{equation*}
  for all $y, z\in \RR^{d}$.
Our reaction function $f:\bbR^d\times[0,1]\times\Om\to[0,\infty)$ will then satisfy certain uniform bounds and be stationary ergodic:

\begin{list}{ (\therscan)}
{
\usecounter{rscan}
\setlength{\topsep}{1.5ex plus 0.2ex minus 0.2ex}
\setlength{\labelwidth}{1.2cm}
\setlength{\leftmargin}{1.5cm}
\setlength{\labelsep}{0.3cm}
\setlength{\rightmargin}{0.5cm}
\setlength{\parsep}{0.5ex plus 0.2ex minus 0.1ex}
\setlength{\itemsep}{0ex plus 0.2ex}
}

\item\label{h:stat0} Uniform bounds: $f$ is Lipschitz with constant $M\ge 1$ and
\[
f(x,0,\om)=f(x,1,\om)=0
\]
for each $(x,\om)\in\bbR^d\times\Om$.  There is also $\theta_0\in (0,1)$ and a Lipschitz function $f_0:[0,1]\to[0,\infty)$  with $f_0(u)=0$ for $u\in[0,\theta_0]\cup\{1\}$ and $f_0(u)>0$ for $u\in(\theta_0,1)$  such that 
\[
f(x,u,\om)\ge f_0(u) 
\]
for each $(x,u,\om)\in \bbR^d\times [0,1]\times\Om$. 

If there is also $\theta>0$ such that $f(\cdot,u,\cdot)\equiv 0$ for $u\in[0,\theta]$ and $f$ is non-increasing in $u$ 
on $[1-\theta,1]$ (for each $(x,\om)\in\bbR^d\times\Om$), then we say that $f$ is an {\it ignition reaction}.
\item\label{h:stat}  Stationarity:  for each $(x,y,u,\om)\in \RR^{2d}\times[0,1]\times\Om$ we have
\begin{equation*}
f(x+y, u, \om)=f(x, u, \mscr{T}_{y}\om). 
\end{equation*}
\item\label{h:erg}  Ergodicity: if 
$\mscr{T}_{y}E=E$ for some $E\in \mathcal{F}$ and each  $y\in \RR^{d}$,
then $\PP[E]\in\{0,1\}$.  
\end{list}

The hypotheses on $f_0$ and the definition of ignition reactions in \eqref{h:stat0} obviously imply that ignition reactions vanish for $u$ near 0, that is, the {\it ignition temperature} $\inf\{ u>0\,|\, f(\cdot,u,\cdot)\not\equiv 0\}$ is positive.  On the other hand, 
reactions with $f(x,u,\om)>0$ whenever $u\in(0,1)$ are usually called {\it monostable}.  Some of our results will apply to general reactions satisfying \eqref{h:stat0}--\eqref{h:erg}, while others will only apply to ignition reactions.

It follows from \eqref{h:stat0} that for each $\om\in\Om$, the functions $u\equiv 0,1$ are equilibrium solutions of \eqref{genhomeq}.  The maximum principle then shows that if $0\leq u(0,\cdot, \om)\leq 1$, then $0\leq u(t,\cdot,\om)\leq 1$ for all $t>0$ (we will only consider such solutions here).  However, our results immediately extend to the case when \eqref{genhomeq} has $\om$-dependent equilibria $u^-(\cdot,\om)<u^+(\cdot,\om)$ and more general $f$
(including of bistable and mixed types).  In that case we would also need to assume certain hypotheses analogous to Definition \ref{D.1.1a} below in Theorems \ref{T.1.4}(i) and \ref{T.1.6}(i).  Rather than stating these in detail, we refer the reader to the hypotheses of Theorem 2.7 in \cite{andrejbd} (and to Remark 2 after it), which is the  result that replaces Theorem~\ref{thm:bw} below in the proof of those results.

It follows from the work of Aronson and Weinberger \cite{aronwein} that the equilibrium $u\equiv 1$ is ``more stable'' than $u\equiv 0$ when \eqref{h:stat0} holds.  Specifically, solutions to \eqref{genhomeq} with initial data greater than $\theta_0$ on large enough balls converge to 1 locally uniformly as $t\to\infty$ (see Lemma \ref{lem:slowspread} below).  In fact, this spreading occurs  (asymptotically) at speeds no less than $c_0>0$, the {\it asymptotic spreading speed} and {\it traveling front speed} for the homogeneous ignition reaction $f_0$.  This speed is the one from the unique (up to translation in $s$) solution $(c_0,U_0)$ to $U_0''(s)+c_0U_0'(s)+f_0(U_0)=0$ with boundary conditions $\lim_{s\to -\infty} U_0(s)=1$ and $\lim_{s\to \infty} U_0(s)=0$ (see \cite{aronwein}), which means that for any $e\in\mathbb S^{d-1}$, the function $u(t,x):=U_0(x\cdot e-c_0t)$ is a traveling front solution for  
 \begin{equation} \label{1.3}
 u_t=\Delta u+f_0(u)
 \end{equation}
  moving in direction $e$.
 
To obtain the existence of a deterministic Wulff shape for \eqref{genhomeq}, we will need one more hypothesis on $f$, which is relatively mild for ignition reactions but introduces more stringent limits on the behavior of monostable reactions at small values of $u$.    Loosely speaking, we will require that the solutions to \eqref{genhomeq} are {\it pushed} (as opposed to {\it pulled}), meaning that their dynamics are determined by the values of $f$ at ``intermediate'' $u$ (rather than at $u$ near 0).  We note that this is not the case for {\it KPP reactions}, whose solutions are pulled in the above sense.  

We will therefore consider reactions that are not too strong at small $u$, with the strength of the reaction at some $u$ being $\frac{f(x,u,\om)}u$ (which is the exponential rate of growth of solutions to the ODE $u_t=f(x,u,\om)$).  This obviously does not affect ignition reactions at all, but we will also need to assume that once the reaction does become strong as $u$ increases, it cannot become arbitrarily weak until $u\approx 1$.  This essentially prevents the decoupling of the propagation of intermediate values of $u$ from the propagation of values near 1 (see \eqref{def:bw} below).

To satisfy both these requirements, we will assume in Theorems \ref{T.1.4}(i) and \ref{T.1.6}(i) below that $f\in \mathcal{F}(f_0,M,\zeta,\xi)$, for some $\zeta<c_0^2/4$ and $\xi>0$, with the class of reactions $\mathcal{F}(f_0,M,\zeta,\xi)$ defined below. (This hypothesis excludes KPP reactions, as any KPP $f\ge f_0$ is known to satisfy $\inf_{(x,\om)\in\bbR^d\times\Om} f_u(x,0,\om)\ge c_0^2/4$.)  We use the convention $\inf\emptyset=\infty$.

\begin{definition} \label{D.1.1a}
For $f_0,M$ from \eqref{h:stat0} and $\zeta, \xi>0$, 
 let $\mathcal{F}(f_0,M,\zeta,\xi)$ be the class of all $f$ from \eqref{h:stat0} such that 
\begin{equation}  \label{1.4b}
 \inf_{(x,\om)\in\bbR^d\times\Om} \, \inf_{u\in[\gamma_f(x,\om;\zeta),\theta_0]} f(x,u,\om)  \ge \xi,
\end{equation}
where
\begin{equation} \label{1.4a}
 \gamma_f(x,\om;\zeta):= \inf \{ u\ge 0 \,|\, f(x,u,\om)> \zeta u \}.
\end{equation}
\end{definition}
\smallskip

This
and the results from \cite{andrejbd} on {\it bounded width} of solutions to \eqref{genhomeq} (see Theorem \ref{thm:bw} below), which guarantee that the relevant solutions are pushed, will enable us to prove existence of a deterministic Wulff shape for \eqref{genhomeq} in dimensions $d\le 3$.  However, if one can show by some other means that appropriate solutions are pushed in a very weak sense (see Theorem \ref{T.1.8}(i) below), then the hypotheses $f\in \mathcal{F}(f_0,M,\zeta,\xi)$ and $d\le 3$ are not needed to obtain a deterministic Wulff shape.  Moreover, our homogenization results also do not specifically require these hypotheses. 

Next we state our definition of a (deterministic) Wulff shape for \eqref{genhomeq}. 

\begin{definition} \label{D.1.2}
%
Assume \eqref{h:stat0}, and let $R_0>0$ be large enough so that the solution to \eqref{1.3} with initial data $u(0,\cdot)=\frac {1+\theta_0}2\chi_{B_{R_0}(0)}$ converges locally uniformly to 1 as $t\to\infty$ (see Lemma \ref{lem:slowspread} below).
For any fixed $\om\in\Om$, 
let $u_\om$ 
solve \eqref{genhomeq} with initial data 
\begin{equation} \label{1.21}
u_\om(0,\cdot) = \frac {1+\theta_0}2\chi_{B_{R_0}(0)}. 
\end{equation}
If there is a continuous function $w:\mathbb{S}^{d-1}\to(0,\infty)$ such that with the (open bounded) set $\mathcal S$ from \eqref{eq:wulffdef}
we have 
\begin{align}
\lim_{t\rightarrow\infty}  
\inf_{x\in (1-\de) \mathcal{S}t} u_\om(t,x) & =1, \label{eq:lbsxi'}
\\ \lim_{t\rightarrow\infty} 
\sup_{x\notin (1+\de) \mathcal{S}t} u_\om(t,x) & =0 \label{eq:ubsxi'}
\end{align}
for  each $\de>0$, 
then we say that $\mathcal S$ is a {\it Wulff shape} for \eqref{genhomeq} with this $\om\in\Om$. If there is $\Om_{0}\subseteq \Om$ with $\PP[\Om_{0}]=1$ such that \eqref{genhomeq} with each $\om\in\Om_0$ has the same Wulff shape $\mathcal S$, then we say that $\mathcal S$ is a {\it deterministic Wulff shape} for \eqref{genhomeq}.
\end{definition}

{\it Remarks.}  1.  Of course, then $w(e)$ is the {\it (deterministic) spreading speed} in direction $e$ for \eqref{genhomeq}.
\smallskip

2. One may wonder about the choice of initial data for $u_\om$, but the comparison principle shows that the validity of \eqref{eq:lbsxi'} and \eqref{eq:ubsxi'} is independent of this choice as long as the initial data are non-negative, compactly supported, have supremum less than one, and the resulting solutions to \eqref{1.3}  converge to 1 locally uniformly as $t\to\infty$ (see the start of Section \ref{sec:sspeed}).   In particular, the definition is independent of the choice of $R_0$ as long as it is large enough.
\smallskip


In addition to long time evolution of solutions starting from compactly supported initial data, we will also need to consider the case of front-like initial data.  As explained in the introduction, while it is not clear whether existence of traveling or pulsating front solutions in periodic media  extends to existence of some type of front-like solutions in random media, we will show that the analogous question for front speeds can be answered in the affirmative in some cases.  We will use the following definition for the latter.

\begin{definition} \label{D.1.3}
Assume \eqref{h:stat0}, and for any $(\om,e)\in\Om\times\mathbb S^{d-1}$,
let $u_{\om,e}$ 
solve \eqref{genhomeq} with initial data 
\[
u_{\om,e}(0,\cdot) = \frac {1+\theta_0}2\chi_{\{x\cdot e< 0\}}. 
\]
Fix any $(\om,e)\in\Om\times\mathbb S^{d-1}$. If there is $c^*(e)\in\bbR$ such that for  each  compact set $K\subseteq  \{x\cdot e>0\}\subseteq\bbR^d$ we have 
\begin{align*}
\lim_{t\to \infty}   \inf_{x\in (c^*(e)e-K)t} u_{\om,e}(t,x) &= 1,  
\\ \lim_{t\to \infty}  \sup_{x\in (c^*(e)e+K)t} u_{\om,e}(t,x) &= 0,
\end{align*}
then we say that $c^*(e)$ is a {\it front speed} in direction $e$ for \eqref{genhomeq} with the fixed $\om\in\Om$.
If there is $\Om_{0}\subseteq \Om$ with $\PP[\Om_{0}]=1$ such that \eqref{genhomeq} with each $\om\in\Om_0$ has the same front speed $c^*(e)$  in direction $e$, then we say that $c^*(e)$ is a {\it deterministic front speed} in direction $e$ for \eqref{genhomeq}.
\end{definition}

{\it Remarks.} 1. One can again consider instead any initial data
\[
(\theta_0+\alpha)\chi_{\{x\cdot e< -\alpha^{-1}\}} \le u_{\om,e}(0,\cdot)  \le (1-\alpha)\chi_{\{x\cdot e\le \alpha^{-1}\}},
\]
with any $\alpha>0$.  
\smallskip

2. In both these definitions we do require vanishing of the initial data for large $|x|$ resp.~$x\cdot e$, rather than just convergence to 0.  The latter might result in faster spreading speeds and front speeds depending on the decay rates and on $f$ (cf.~Definition \ref{D.1.5} below).  The speeds we define here could therefore be also called ``minimal front speeds'', but we do not use this terminology here.
\smallskip

3. Definition 1.3 appears to be the first definition of front speeds that does not rely on the existence of special solutions (such as traveling and pulsating fronts in the homogeneous and periodic settings).  However, it is conceptually related to the (time-independent) planar metric problem introduced by Armstrong and Cardaliaguet in \cite{scottpierre} in their study of stochastic homogenization for quasilinear viscous Hamilton-Jacobi equations. In their work, for each $e\in \mathbb{S}^{d-1}$, the solution of this problem at a point $z$ can be interpreted as a ``distance'' (relative to the Hamiltonian in question) from $z$ to the half-space $\{x\cdot e< 0\}$.  
They show that under appropriate hypotheses, such solutions converge almost surely and locally uniformly to  $c_ez\cdot e$ as $\ve\to 0$, with $c_e$  some constant.
\smallskip

The reason for only considering compact sets $K$ in Definition \ref{D.1.3} is unboundedness of the hyperplane $\{x\cdot e=0\}$ when $d\ge 2$.  In general, one may therefore expect to find arbitrarily large exceptional regions in its neighborhood, on which we may observe propagation with speeds different from $c^*(e)$ on arbitrarily long time scales.  But since $f$ is stationary ergodic, distance of such regions from the origin grows very quickly with their size for almost all $\om\in\Om$, and they will therefore not pose a threat for compact $K$.  We also note that lower and upper bounds on the speed of propagation in \eqref{1.1} (see Lemmas \ref{lem:slowspread} and \ref{L.4.2} below) allow one to replace compact sets $K\subseteq  \{x\cdot e>0\}$ in Definition \ref{D.1.3} by cones $C_{\delta}:=\{x\cdot e\ge \delta+\delta|x-(x\cdot e)e|\}$ for all $\delta>0$. 

With these definitions we can now state our first main result for \eqref{genhomeq} with general stationary ergodic reactions (including both monostable and ignition ones), whose homogenization parts (iii) and (iv) apply to convex initial sets $A$. As mentioned in the introduction, if deterministic front speeds exist in each direction and the inverse Freidlin-G\"artner formula \eqref{eq:fsdef} holds, we in fact obtain the Osher-Merriman growth rule here.  Specifically, the homogenized solution is $\chi_{\Theta^{A, \mathcal{S}}}$, where $\mathcal S$ is the Wulff shape and 
\[
\Theta^{A, \mathcal{S}}:=\{(t,x)\in(0,\infty)\times\bbR^d \,|\, x\in A+t\mathcal{S} \}.
\]
It is not difficult to see (and follows from the results in Section \ref{s:fstohom} below) that in this case, 
$\chi_{\Theta^{A, \mathcal{S}}}$
 is a viscosity solution to the first order Hamilton-Jacobi PDE \eqref{eq:hjeq}.
We actually then also have $\Theta^{A, \mathcal{S}}=\Theta^{A, c^{*}}$, where the open set
\begin{equation} \label{1.40'}
\Theta^{A, c^{*}}:=\{(t,x)\in(0,\infty)\times\RR^d \,|\, \overline{v}(t,x)>0\}
\end{equation}
is obtained by taking any uniformly continuous function $\overline{v}_{0}$ on $\bbR^d$  satisfying $\overline v_0>0$ on $A$ and $\overline{v}_0<0$ on $\bbR^d\setminus  \overline A$, and letting $\overline v$ be the unique viscosity solution  to \eqref{eq:hjeq} with initial data $\overline{v}(0,\cdot)=\overline{v}_{0}$ (then also $\Theta^{A, c^{*}}$ is independent of the specific choice of $\overline{v}_0$ and $\partial \Theta^{A, c^{*}}$ has zero measure, see Section \ref{s:fstohom}).  Our homogenization result here continues to hold even if \eqref{eq:fsdef} does not hold for all directions $e$, with existence of deterministic front speeds being the only requirement and $\chi_{\Theta^{A, c^{*}}}$ the limiting function. 

Finally, for the sake of generality, we allow for $O(1)$ shifts  and $o(1)$ errors in initial data as $\ve\to 0$ in \eqref{eq:homeq}.  For $A\subseteq\bbR^d$ and $r> 0$, we therefore let $B_{r}(A)=\bigcup_{x\in A}B_{r}(x)=A+B_r(0)$ and $A_r^0:=A\setminus \overline{B_r(\partial A)}$.  
For the sake of completeness, let us also denote $B_0(A):=\overline{A}$ and $A^0:=A^0_0:={\rm int}(A)$.

\begin{thm} \label{T.1.4}
Assume \eqref{h:stat0}--\eqref{h:erg} and that $A\subseteq\RR^d$ is open.

(i) If 
$d\le 3$ and $f\in \mathcal{F}(f_0,M,\zeta,\xi)$ for some $\zeta<c_0^2/4$ and $\xi>0$, then \eqref{genhomeq} has a  deterministic Wulff shape.

(ii)  If  \eqref{genhomeq} has a deterministic Wulff shape $\mathcal S$, then $\mathcal S$ is convex.  Also, if $e\in \mathbb{S}^{d-1}$ is a unit outer normal of $\mathcal{S}$ and $w$ is given by \eqref{eq:wulffdef}, then \eqref{genhomeq} has a deterministic front speed in direction $e$, given by \eqref{eq:fsdef}.

(iii) If \eqref{genhomeq} has  a deterministic front speed $c^{*}(e)$ in each direction $e\in \mathbb{S}^{d-1}$,  then for almost all $\om\in\Om$ the following holds.  If $A$ is convex, $\alpha>0$,  and $u^\ve(\cdot,\cdot,\om)$ solves \eqref{eq:homeq} and
\begin{equation} \label{4.4a}
(\theta_0+\alpha)\chi_{A_{\psi(\ve)}^0+y_\ve} \le u^\ve(0,\cdot,\om)\le (1-\alpha)\chi_{B_{\psi(\ve)}(A)+y_\ve}
\end{equation}
for each  $\ve>0$, with some $y_\ve\in B_{1/\alpha}(0)$ and $\lim_{\ve\to 0}\psi(\ve)=0$, then 
\begin{equation} \label {1.30}
\lim_{\ve\to 0}  u^\ve(t,x+y_\ve,\om)= \chi_{\Theta^{A, c^{*}}}(t,x)
\end{equation}
locally uniformly on $ ([0,\infty)\times\RR^d) \setminus \partial \Theta^{A, c^{*}}$ (and $\partial \Theta^{A, c^{*}}$ has zero measure).

(iv) If \eqref{genhomeq} has  a deterministic front speed $c^{*}(e)$ in each direction $e\in \mathbb{S}^{d-1}$, then it
has a deterministic Wulff shape $\mathcal S$ with $w$ from \eqref{eq:wulffdef} satisfying \eqref{eq:fg} for each $e\in \mathbb{S}^{d-1}$.  Moreover, if \eqref{eq:fsdef} holds for each $e\in \mathbb{S}^{d-1}$ 
(i.e., $c^*(\frac y{|y|}) |y|$ is convex),
then $\Theta^{A, c^{*}}=\Theta^{A, \mathcal{S}}$.
\end{thm}

{\it Remarks.}  1.   (ii) shows that Wulff shapes with tangent hyperplanes can give rise to front speeds in stationary ergodic media.  Here, a hyperplane $H\subseteq\mathbb R^d$ is tangent to $\mathcal S$ at $y\in\partial\mathcal S\cap H$ when  $\partial \mathcal S\cap B_\delta(y)\subseteq B_{\phi(\delta)}(H)$ for each $\delta>0$, with $\lim_{\delta\to 0} \frac 1\delta\phi(\delta)=0$.  (If $\mathcal S$ is also convex, this is equivalent to $H$ being the unique supporting hyper-plane for $\mathcal S$ at $y$.)
If $e\in\mathbb S^{d-1}$ is the unit normal vector to this $H$ such that $y+se\notin \mathcal S$ for all small $s>0$, then we say that  $\mathcal S$ has unit outer normal $e$ (at $y$).  If each $e\in\mathbb S^{d-1}$ is a unit normal of $\mathcal S$, then we say that $\mathcal S$ has no corners.
\smallskip


2.  Homogenization results are typically stated with $\psi\equiv 0$ and $y_\ve=0$ above.  We use the present form of (iii) for the sake of generality.
\smallskip

3.  (iii) obviously extends to initial conditions that can take the value $\theta_0$ on $\partial A+y_\ve$ (in the limits $x\to\partial A+y_\ve$ and $\ve\to 0$) because $\alpha>0$ is arbitrary and the convergences are locally uniform.  We could even consider initial data with some values in $(0,\theta_0]$, but then there would be a transient initial time interval during which the limiting solution $\overline u$ would also have values in $(0,\theta_0]$.  The  region $\{\overline u(t,\cdot)=1\}$ would then invade the region $\{\overline u(t,\cdot)\in(0,\theta_0]\}$ (both time-dependent) at speeds that would depend on the unit outer normal vector to the former region at each point $x$ of its boundary as well as on the value $\overline u (t,x)$ at that point.  We do not pursue this generalization here.
\smallskip

4.  Note that if the claim in (iii) holds, then it follows for any $e\in\mathbb S^{d-1}$ that $c^*(e)$ is the deterministic front speed in direction $e$.
\smallskip



Under a stronger hypothesis concerning propagation speed of front-like solutions to \eqref{genhomeq}, we are able to obtain homogenization for general initial sets $A$ and more general initial conditions than \eqref{4.4a}.  Let us start with the following definition.

\begin{definition} \label{D.1.5}
Assume \eqref{h:stat0}, and for $(\om,e,\alpha)\in\Om\times\mathbb S^{d-1}\times[0,1]$,
let $u_{\om,e,\alpha}$ 
solve \eqref{genhomeq} with initial data 
\[
u_{\om,e,\alpha}(0,\cdot) = \chi_{\{x\cdot e< 0\}} + \alpha \chi_{\{x\cdot e\ge 0\}}. 
\]
Fix any $(\om,e)\in\Om\times\mathbb S^{d-1}$. If \eqref{genhomeq} with the fixed $\om\in\Om$ has front speed $c^*(e)$ in direction $e$, and for each  compact set $K\subseteq  \{x\cdot e>0\}$ there is $\beta_{K,e}:(0,1]\to(0,1]$ with $\lim_{\alpha\to 0}\beta_{K,e}(\alpha)=0$ such that 
\[
\limsup_{t\to \infty} \sup_{x\in (c^*(e)e+K)t} u_{\om,e,\alpha}(t,x) \le\beta_{K,e}(\alpha)
\]
for each  $\alpha\in(0,1]$, then we say that $c^*(e)$ is an {\it exclusive front speed} in direction $e$ for \eqref{genhomeq} with the fixed $\om\in\Om$.
If there is $\Om_{0}\subseteq \Om$ with $\PP[\Om_{0}]=1$ such that \eqref{genhomeq} with each $\om\in\Om_0$ has the same exclusive front speed $c^*(e)$  in direction $e$ (with the same $\beta_{K,e}$), 
then we say that $c^*(e)$ is a {\it deterministic exclusive front speed} in direction $e$ for \eqref{genhomeq}.  
\end{definition}

{\it Remark.}  For \eqref{genhomeq} with ignition reactions, one can actually choose $\beta_{K,e}(\alpha)=\alpha$ for all sufficiently small $\alpha$ (see the proofs of Theorems \ref{T.1.6}(i) and \ref{T.1.6a}(ii) below).  
We will also see in Theorem \ref{T.4.4} that under very mild hypotheses, $\beta_{K,e}$ can be chosen to be independent of $e$.
\smallskip

In particular, it follows from this definition that if $c^*(e)$ is an exclusive front speed for \eqref{genhomeq}, then solutions that satisfy $\lim_{x\cdot e\to\infty} u(0,x)=0$ and $\liminf_{x\cdot e\to-\infty} u(0,x)>\theta_0$ all propagate with exact speed $c^*(e)$ in direction $e$ (in the sense of the above definitions).  Note that this is possible for ignition reactions but generally not  for monostable reactions.  In fact, if for some $\om\in\Om$ we have $\inf_{x\in\bbR^d} f(x,u,\om)>0$ for each $u\in(0,1)$, then the solutions from Definition \ref{D.1.5} satisfy $\lim_{t\to\infty} \inf_{x\in \bbR^d} u(t,x)=1$ whenever $\alpha>0$.


Part (ii) of the following result shows that existence of deterministic exclusive front speeds in all directions yields homogenization for all initial sets $A$.  In particular, both its parts can be combined with Theorem \ref{T.1.4}(i,ii) to obtain a stronger version of Theorem \ref{T.1.4}(iii) for ignition reactions.

\begin{thm} \label{T.1.6}
Assume \eqref{h:stat0}--\eqref{h:erg} and that $A\subseteq\RR^d$ is open.

(i) If   $d\le 3$, an ignition reaction $f\in \mathcal{F}(f_0,M,\zeta,\xi)$ for some $\zeta<c_0^2/4$ and $\xi>0$, and \eqref{genhomeq} has a deterministic front speed in direction $e\in\mathbb S^{d-1}$, then \eqref{genhomeq} has a deterministic exclusive front speed in direction $e$.

(ii)  If \eqref{genhomeq} has a deterministic exclusive front speed   $c^{*}(e)$ in each direction $e\in\mathbb S^{d-1}$, then for almost all $\om\in\Om$ the following holds.  If $\alpha>0$ and $u^\ve(\cdot,\cdot,\om)$ solves \eqref{eq:homeq} and
\begin{equation*} 
(\theta_0+\alpha)\chi_{A_{\psi(\ve)}^0+y_\ve} \le u^\ve(0,\cdot,\om)\le \chi_{B_{\psi(\ve)}(A)+y_\ve} + \psi(\ve) \chi_{\bbR^d\setminus(B_{\psi(\ve)}(A)+y_\ve)}
\end{equation*}
for each  $\ve>0$, with some $y_\ve\in B_{1/\alpha}(0)$ and $\lim_{\ve\to 0}\psi(\ve)=0$, then \eqref{1.30} holds
locally uniformly on $ ([0,\infty)\times\RR^d) \setminus \partial\Theta^{A, c^{*}}$ (and $\partial \Theta^{A, c^{*}}$ has zero measure).
\end{thm}

Our results naturally apply to periodic reactions, as these are a special class of stationary ergodic ones (we can then drop $\om$ from the notation and ``deterministic'' from the terminology).  The following result spells out this application, and also shows that the front speeds defined here coincide with the unique/minimal  pulsating front speeds for these reactions.

\begin{thm} \label{T.1.6a}
Let $f:\bbR^d\times[0,1]\to[0,\infty)$ be Lipschitz, periodic in $x\in\bbR^d$, and satisfying $f(\cdot,0)\equiv f(\cdot,1)\equiv 0$.  Assume also that there is $\theta>0$ such that $f$ is non-increasing in $u\in(1-\theta,1)$ for each $x\in\bbR^d$ as well as $f(x,u)>0$ for these $(x,u)$.  Finally, let $\theta'\ge 0$ be the largest number such that $f(\cdot,u)\equiv 0$ for each $u\in[0,\theta']$, and assume that $\sup_{x\in\bbR^d} f(x,u)>0$ for each $u\in(\theta',1)$. 

(i)  The PDE
\begin{equation} \label{1.1}
u_t=\Delta u + f(x,u)
\end{equation}
has a front speed in each direction $e\in\mathbb S^{d-1}$, and that speed equals the minimal pulsating front speed in direction $e$ from \cite{berhamel}.  Therefore the homogenization result Theorem \ref{T.1.4}(iii) (as well as (iv)) holds for \eqref{1.1}.

(ii)  If $\theta'>0$ (i.e., $f$ is an ignition reaction and the minimal speeds from (i) are also unique), then the front speeds from (i) are exclusive.  Therefore the homogenization result Theorem \ref{T.1.6}(ii) holds for \eqref{1.1}.
\end{thm}

{\it Remarks.}  1.  The hypotheses on $f$ are those from \cite{berhamel}, and while \cite{berhamel} requires monostable reactions (i.e., those with $\theta'=0$) to be $C^{1,\delta}$ in $u$, this is only used in the study of pulsating fronts with speeds strictly greater than the minimal speed (see Section 6 in \cite{berhamel}).  
\smallskip

2.  The results in \cite{berhamel} apply to general spatially periodic second order operators in place of $\Delta$ (satisfying standard ellipticity hypotheses, as well as the first-order term being divergence-free and mean-zero),  which turn \eqref{1.1} into a form captured by \eqref{1.10}.  Theorem \ref{T.1.6a} and its proof immediately extend to this case (with the versions of Theorems \ref{T.1.4} and \ref{T.1.6} from Theorem \ref{T.1.8} below).  
\smallskip

3.  The homogenization claim in  (i) is a stronger version of the result of Alfaro and Giletti \cite{alfarogiletti}, who require smooth $\partial A$ and $\theta'=0$, as well as slightly stronger hypotheses on $f$.  
\smallskip

Seeing the usefulness of Wulff shapes without corners in the study of front speeds and homogenization for \eqref{genhomeq}, it is natural to ask when a Wulff shape for \eqref{genhomeq} has no corners.  As mentioned in the introduction, a result from \cite{cafleemel} shows that Wulff shapes can have corners.  However, one simple case when this does not happen is when $f$ is {\it isotropic}.  That is, its statistics are invariant under rotations --- and hence its Wulff shape must be a ball, if it exists.

\begin{definition} \label{D.1.7}
Let $SO(d)$ be the group of rotation matrices on $\RR^{d}$. 
We say that $f$ from \eqref{genhomeq}
is {\it isotropic} if there is a group of measure-preserving transformations $\left\{\sigma_{\mscr{R}} :\Om\rightarrow \Om\right\}_{\mscr{R}\in SO(d)}$ such that
\begin{equation*}
\sigma_{\mscr{R}}\circ \sigma_{\mscr{P}}=\sigma_{\mscr{R}\mscr{P}}
\end{equation*}
for all $\mscr{R},\mscr{P}\in SO(d)$, and 
\begin{equation*}
f(x, u, \sigma_{\mscr{R}}\om) = f(\mscr{R}x, u, \om)
\end{equation*}
for each $(\mscr{R},x,u,\om)\in SO(d)\times\bbR^d\times[0,1]\times\Om$.
\end{definition}

A natural example of isotropic reactions are perturbations of homogeneous reactions by radially symmetric impurities randomly distributed according to a Poisson point process.  This is a special case of the following example.
 
\begin{example}\label{ex:ppp}
Let $\left\{x_{n,\om}\right\}_{n\in \mathbb{N}}\subseteq \RR^{d}$ be a  Poisson point process on $\bbR^d$ (with $\om\in\Om$ as above) and let $\{\mscr R_{n,\om}\}_{n\in \mathbb{N}}\subseteq SO(d)$ be a sequence of rotations on $\bbR^d$, each with uniform distribution. If $\left\{x_{n,\om}\right\}_{n\in \mathbb{N}}$ and $\{\mscr R_{n,\om}\}_{n\in \mathbb{N}}$ are independent and
\[
f(x,u, \om):= g(\{\mscr R_{n,\om}(x-x_{n,\om})\}_{n\in\mathbb N},u)
\]
for some function $g$, then $f$ is isotropic.  For instance, we could take $f(x,u, \om):= g(\inf_{n} |x-x_{n,\om}|,u)$ for some Lipschitz $g:[0,\infty)\times[0,1]\to \bbR$ with $g(y,0)=g(y,1)=0$ and  $g(y,\cdot)\ge f_0$ for each $y\ge 0$.

\end{example}

The above results, applied to isotropic reactions, and Lemma \ref{L.3.4} below now yield the following corollary.

\begin{cor} \label{C.1.7}
Assume \eqref{h:stat0}--\eqref{h:erg}, $d\le 3$, $f\in \mathcal{F}(f_0,M,\zeta,\xi)$ for some $\zeta<c_0^2/4$ and $\xi>0$, and that $f$ is isotropic.  Let $A\subseteq\RR^d$ be open.

(i) \eqref{genhomeq} has a  deterministic Wulff shape $\mathcal S=B_w(0)$ for some  $w>0$, and a deterministic front speed $c^*(e)=w$ in each direction $e\in\mathbb S^{d-1}$.  In particular, $\Theta^{A, c^{*}}=\Theta^{A, B_{w}(0)}$.

(ii) The claim in Theorem \ref{T.1.4}(iii) holds.

(iii)  If $f$ is an ignition reaction, then the deterministic front speeds are all exclusive and the claim in Theorem \ref{T.1.6}(ii)
holds.
\end{cor}


\subsection{Generalization to \eqref{1.10}}\label{S.1.3}
The above results in fact apply to the more general model \eqref{1.10}, which includes, for instance, some (viscous) Hamilton-Jacobi equations with possibly non-convex or non-coercive Hamiltonians.
The
$\ve$-space-time-scaled version of \eqref{1.10} is
\begin{equation} \label{1.10e}
u^{\ve}_{t} = \ve^{-1} F\left(\ve^2 D^2 u,\ve\nabla u,u, \ve^{-1} x, \om \right)  \qquad\text{on  $(0,\infty)\times\bbR^d$.}
\end{equation}
We consider the case when \eqref{1.10} models phase transitions, with solutions transitioning between two equilibria $u^-<u^+$.  After an appropriate transformation these can be assumed to be $u^-\equiv 0$ and $u^+\equiv 1$.

Definitions \ref{D.1.3} and \ref{D.1.5} above extend naturally to \eqref{1.10}, and this form also allows \eqref{genhomeq} with general second order linear operators and general reactions (including of bistable and mixed types) as well as more general first- and second-order terms.  The basic hypothesis in this setting will be as follows.

\medskip

{\bf Hypothesis H.}  (i) Let \eqref{1.10} have a unique solution in some class of functions $\mathcal A\subseteq L^1_{\rm loc}((0,\infty)\times\bbR^d)$ for each $\om\in\Om$ and each  locally BV initial condition $0\le u(0,\cdot,\om)\le 1$, with this solution being constant 0 resp.~1 when $u(0,\cdot,\om)\equiv 0$ resp.~$1$.  Assume also that left time-shifts of solutions (restricted to $(0,\infty)\times\bbR^d$) are solutions from $\mathcal A$ and that \eqref{1.10e} satisfies the (parabolic) comparison principle within $\mathcal A$.

(ii) Lemma \ref{L.4.2} below holds for solutions to \eqref{1.10} from $\mathcal A$, and there are $\theta_0<1$ and $R_0<\infty$ such that solutions $u_\om$ to \eqref{1.10} with $u_\om(0,\cdot)=\theta_0\chi_{B_{R_0}(0)}$ 
satisfy locally uniformly in $x\in\bbR^d$,
\[
\lim_{t\to\infty}\inf_{\om\in\Om} u_\om(t,x)=1.
\] 

(iii) The analogs of \eqref{h:stat} and \eqref{h:erg} for $F$ hold. 
\medskip

Lemmas \ref{lem:slowspread} and \ref{L.4.2} below show that Hypothesis H holds for \eqref{genhomeq},  with $\theta_0$ in (ii) being any number {\it strictly greater} than $\theta_0$ from \eqref{h:stat0}
 (e.g.,  $\frac{1+\theta_0}2$ as in  \eqref{1.21}),  
 and with, for instance, $\mathcal A=\bigcap_{\tau>0} C^{1+\ga, 2+\ga}((\tau,\infty)\times\bbR^d)$.
(We choose here this notation for the sake of simplicity of presentation, as the above results then generalize verbatim.  Then we can also equivalently state Definitions \ref{D.1.2} and \ref{D.1.3} with any number from $[\theta_0,1)$ in place of $\frac{1+\theta_0}2$, via the argument in the proof of Theorem \ref{T.1.6}(i) and at the start of Section \ref{sec:sspeed} below, and we will let this be $\theta_0$ for the sake of simplicity.)   The analysis for \eqref{1.10} below therefore also applies to \eqref{genhomeq}. Nevertheless, in it we will consider initial data with value $\ge \theta_0+\alpha$ (rather than $\ge \theta_0$) on some sets, with $\alpha>0$ so that our arguments also directly apply to \eqref{genhomeq} with $\theta_0$ from \eqref{h:stat0}. 

We  note that considering only continuous initial conditions would suffice, and we include locally BV ones only for notational convenience, so that we can use initial conditions that are (multiples of) characteristic functions of sets.  Also, (ii) in fact shows that spreading with some positive minimal speed $c_0>0$ holds for \eqref{1.10}, in the sense of \eqref{2.1} below.

Finally for $u:[0,\infty)\times\bbR^d\to [0,1]$ and $\eta\in \left(0, \frac{1}{2}\right)$, we define the \emph{width of the transition zone} of $u$ from $\eta$ to $1-\eta$ at time $t\ge 0$ to be 
\begin{equation}\label{def:tzwidth}
L_{u, \eta}(t):=\inf\left\{L>0 \mid \left\{x\,\mid\, u(t, x)\geq \eta\right\}\subseteq B_{L}\left(\left\{x\,\mid\, u(t, x)\geq 1-\eta \right\}^0_{1/\eta}\right)\right\}.
\end{equation}
This is also the Hausdorff distance of $\left\{u(t, \cdot)\geq 1-\eta\right\}^0_{1/\eta}$ and $\left\{u(t, \cdot)\leq \eta\right\}$, and it is the smallest $L$ such that if $u(t,x)\ge\eta$, then $B_{L+\eta^{-1}}(x)$ contains a ball of radius $\eta^{-1}$ on which $u(t,\cdot)$ is no less than $1-\eta$.  It is clear that for \eqref{1.10} to have a Wulff shape, it is necessary that $\lim_{t\to\infty} \frac 1t {L_{u, \eta}(t)}=0$ for any $\eta\in \left(0, \frac{1}{2}\right)$, with $u$ the solution from H(ii) above (for almost all $\om\in\Om$).  It turns out that this hypothesis is also sufficient, as part (i) of the following extension of some of the above results to solutions of \eqref{1.10} from the class $\mathcal A$ shows.

\begin{thm} \label{T.1.8}
Assume Hypothesis H.

(i)  If  for each $\eta>0$ and almost all $\om\in\Om$ we have  $\lim_{t\to\infty} \frac 1t {L_{u_\om, \eta}(t)}=0$ for the solution from H(ii), 
then \eqref{1.10} has a  deterministic Wulff shape.  If $F$ is also isotropic (i.e., an analog of Definition \ref{D.1.7} holds, with $F(X,p,u,x, \sigma_{\mscr R} \om)= F(\mscr R X \mscr R^T, \mscr R p, u,\mscr R x,\om)$), then Corollary \ref{C.1.7}(i) holds for \eqref{1.10}  in place of \eqref{genhomeq}.

(ii)  Theorem \ref{T.1.4}(ii--iv) and Theorem \ref{T.1.6}(ii) hold for \eqref{1.10} and \eqref{1.10e} in place of \eqref{genhomeq} and \eqref{eq:homeq}.
\end{thm}

The limitation to $d\le 3$ in Theorems \ref{T.1.4}(i) and \ref{T.1.6}(i) is due to the need in their proofs for Theorem \ref{thm:bw} below, which guarantees that they satisfy the hypothesis in Theorem \ref{T.1.8}(i).   In fact, Theorem \ref{thm:bw} guarantees more: uniform-in-$\om$ boundedness of $\sup_{t\ge T_\eta} L_{u_\om, \eta}(t)$ for each $\eta>0$ and some $T_\eta>0$.  (This also holds for some other reactions $f$, see Theorem 2.7 in \cite{andrejbd} and  Remark 2 after it.)  While Theorem \ref{thm:bw} does not hold for $d\ge 4$ \cite{andrejbd}, the construction of counterexamples in \cite{andrejbd} for which $\limsup_{t\to\infty} \frac 1t {L_{u_\om, \eta}(t)}>0$ is based on properties of reactions that occur with probability zero in the stationary ergodic setting.  It is therefore quite plausible that Theorems  \ref{T.1.4}(i) and \ref{T.1.6}(i) extend to all dimensions.  The former would immediately follow from  Theorem \ref{T.1.8}(i) and the proof of the following conjecture.

\begin{conjecture} \label{C.1.12}
Assume the hypotheses of Theorem \ref{T.1.4}(i), except for the limitation on $d$.  Then the solutions from Definition \ref{D.1.2} satisfy $\lim_{t\to\infty} \frac 1t {L_{u_\om, \eta}(t)}=0$ for each $\eta>0$ and almost all $\om\in\Om$.
 \end{conjecture}

Finally, we highlight one more result, which is of independent interest.  

\begin{thm}  \label{T.1.9}
Assume Hypothesis H
and let $\mathcal T$ (resp. $\tilde{\mathcal T}$) be the set of all directions $e\in\mathbb S^{d-1}$ for which \eqref{1.10}  has a deterministic front speed (resp. deterministic exclusive front speed) $c^*(e)$.  

(i) $\mathcal T$ and $\tilde{\mathcal T}$ are closed and $c^*$ is positive, bounded by $c'$ from Lemma~\ref{L.4.2}, and Lipschitz on $\mathcal T$.  Moreover, there is $\Om_0\subseteq\Om$ with $\mathbb P[\Om_0]=1$ such that \eqref{1.10} with any fixed $\om\in\Om_0$ has front speed $c^*(e)$ in any direction $e\in\mathcal T$ and exclusive front speed $c^*(e)$ in any direction  $e\in\tilde{\mathcal T}$.  

(ii)  If \eqref{1.10} has a deterministic Wulff shape given by \eqref{eq:wulffdef}, then for each $e\in\mathbb S^{d-1}$,
\begin{equation}\label{1.40}
w(e)=\inf_{e'\in\mathcal T \,\&\, e'\cdot e>0} \frac{c^{*}(e')}{e'\cdot e}.
\end{equation}
\end{thm}

{\it Remark.}  The proof of Theorem \ref{T.1.9}(i) in fact shows that a stronger result holds: we can choose $\Om_0$ so that in Definition \ref{D.1.5} we have $\beta_{R_eK,e}=\beta_{K,e_1}$ for any $e\in\mathcal {\tilde T}$,  any rotation $R_{e}$ on $\mathbb R^d$ with $R_{e} e_1=e$, and any compact $K\subseteq \{x_1>0\}\subseteq \RR^d$, and there is also uniformity of the limits in Definitions \ref{D.1.3} and \ref{D.1.5} over rotations as well as over certain translations in $\Om$.  Namely, for each $\om\in \Om_0$ and $\La>0$ we have
\begin{align}
\lim_{t\to \infty}  \inf_{e\in\mathcal T\,\&\,|y|\le\La t} \inf_{x\in (c^*(e)e-R_eK)t} u_{\mscr{T}_y\om,e}(t,x) &= 1,  \label{1.97}
\\ \lim_{t\to \infty} \sup_{e\in\mathcal T\,\&\,|y|\le\La t} \sup_{x\in (c^*(e)e+R_eK)t} u_{\mscr{T}_y\om,e}(t,x) &= 0, \label{1.98}
\\  \limsup_{t\to \infty}  \sup_{e\in\mathcal {\tilde T} \,\&\,|y|\le\La t} \sup_{x\in (c^*(e)e+R_eK)t} u_{\mscr{T}_y\om,e,\alpha}(t,x) & \le\beta_{K,e_1}(\alpha).  \label{1.99}
\end{align}
Such uniformity over translations in $\Om$ also holds for the Wulff shape limits in Definition \ref{D.1.2} (see Proposition \ref{P.3.3} below).

\subsection{Organization of the paper}  
In Section \ref{ss.background} we collect some preliminary results, and also prove Theorems \ref{T.1.6}(i) and \ref{T.1.6a}.
In Section \ref{sec:sspeed} we study Wulff shapes and prove  Theorems \ref{T.1.4}(i) and \ref{T.1.8}(i).  In Section \ref{s:fs} we relate Wulff shapes and front speeds, and prove Theorem \ref{T.1.4}(ii,iv) as well as the corresponding parts of Theorem \ref{T.1.8}(ii), and also  Theorem \ref{T.1.9}.
Finally, in Section \ref{s:fstohom} we 
prove Theorems \ref{T.1.4}(iii) and \ref{T.1.6}(ii), as well as the corresponding parts of Theorem \ref{T.1.8}(ii).  

\subsection{Acknowledgements.} 
We thank Mark Freidlin, Antoine Mellet, Panagiotis Souganidis, and Jack Xin for valuable discussions.  AZ was supported in part by NSF grants DMS-1652284 and DMS-1656269. JL was supported in part by NSF grants DMS-1147523 and  DMS-1700028.

\section{Preliminaries}\label{ss.background}

In this section we collect some useful results concerning solutions to \eqref{genhomeq} that we will need below.  The first one, which was already mentioned in the introduction, shows that the asymptotic spreading speed of solutions to \eqref{genhomeq} with large enough initial data is no less than $c_0$, the unique speed for the homogeneous reaction $f_0$.  

%

\begin{lem} \label{lem:slowspread}
For any $f_0,M$ as in \eqref{h:stat0}, there is $\eta_{0}=\eta_{0}(f_{0}, M)\in(0,\frac 12)$ such that for each $c<c_{0}$ and $\eta>0$, there is $\la(f_{0}, M, c, \eta)\geq 0$ such that the following holds.  If $f$ satisfies \eqref{h:stat0}, $0\le u\le 1$ solves \eqref{genhomeq} for some $\om\in\Om$, and
 $u(t_{1}, x)\geq 1-\eta_{0}$ for some $(t_{1}, x)\in [1,\infty)\times \RR^{d}$, then for each $t\geq t_{1}+\la(f_{0}, M, c, \eta)$ we have
\begin{equation} \label{2.1}
\inf_{|y-x|\leq c(t-t_{1})} u(t, y)\geq 1-\eta.
\end{equation}
The same result holds if the hypothesis $u(t_{1}, x)\geq 1-\eta_{0}$ 
is replaced by 
\begin{equation*}
u(t_{1}, \cdot)\geq (\theta_{0}+\alpha)\chi_{B_{R}(x)}(\cdot)
\end{equation*}
for some $(t_{1}, x)\in [0, \infty)\times \RR^{d}$ and $\alpha>0$, with $R=R(f_{0},\alpha)$ large enough. 
\end{lem}

The second claim in Lemma \ref{lem:slowspread} is a result of Aronson and Weinberger \cite{aronwein}, combined with the comparison principle, while the first claim follows from this and parabolic regularity (see Lemma 3.1 in \cite{andrejbd}). 

We also have an upper bound on the spreading speed for compactly supported initial data, which follows from the next lemma.

\begin{lem} \label{L.4.2}
There are $m',c'>0$ such that if $\om\in\Om$, $r>0$, $y\in\RR^d$, and $u,u'$ are two solutions of \eqref{genhomeq} taking values in $[0,1]$ and satisfying $u(0,x)\le u'(0,x)$ for all $x\in B_r(y)$, 
 then for all $t\ge 0$ we have
\[
u(t,y)\le u'(t,y) + c'e^{-m'(r-c't)}.
\]
\end{lem}
\begin{proof}
Without loss assume that $y=0$. Then $w:=u-u'$ satisfies $w(0,\cdot)\le \chi_{\RR^d\setminus B_r}$ and $w_t\le \Delta w + Mw$ (with $M\ge 1$ the Lipschitz constant for $f$).  It follows that, with $\{e_j\}_{j=1}^d$ the standard basis in $\RR^d$, we have
\[
w(t,x)\le \sum_{j=1}^d \left( e^{-\sqrt M(x\cdot e_j+rd^{-1/2}-2\sqrt M t)} + e^{-\sqrt M(-x\cdot e_j+rd^{-1/2}-2\sqrt M t)}\right)
\]
for any $(t,x)\in[0,\infty)\times\RR^d$  because the sum solves $w'_t= \Delta w' + Mw'$.  The claim now follows with $m':=\sqrt {M/d}$ and $c':=2 \sqrt M\, d$.
\end{proof}

We note that 
a more careful proof would allow for any $c'>2\sqrt M$ in the exponent, but we will not need to optimize $c'$ here. Also note that the lemma immediately generalizes to the conclusion 
\begin{equation*}
u(t,x)\leq u'(t,x)+Ce^{-m'(-|x-y|+r-c't)}\quad\text{for all } x\in \RR^{d}. 
\end{equation*}

A key ingredient in the proof of Theorems \ref{T.1.4}(i) and \ref{T.1.6}(i)
is the bounded width property, first defined by one of us in \cite{andrejbd}. 
 We  say  that $u$ has \emph{bounded width} if $\limsup_{t\to\infty} L_{u, \eta}(t)<\infty$ for each $\eta\in (0, \frac 12)$ (see \eqref{def:tzwidth}, and also compare this to the hypothesis in Theorem \ref{T.1.8}(i)).  
 That is,
\begin{equation}\label{def:bw}\tag{BW}
\sup_{t\geq T_{\eta}} L_{u, \eta}(t)\leq \ell_{\eta} \qquad\text{for each $\eta\in \left(0, \frac 12 \right)$ and some $\ell_{\eta}, T_{\eta}\ge 0$.} 
\end{equation}
(We note that \cite{andrejbd} defines $L_{u,\eta}(t)$ with $\left\{x \,\mid\, u(t, x)\geq 1-\eta \right\}$ in place of $\left\{x \,\mid\, u(t, x)\geq 1-\eta \right\}^0_{1/\eta}$ but as discussed before Theorem 2.9 in \cite{andrejbd}, the two resulting definitions of bounded width are equivalent by parabolic regularity.)  
The following theorem, proved by one of us in \cite{andrejbd}, guarantees \eqref{def:bw} and more for certain solutions to \eqref{genhomeq}.  In it, for $\eta,\eta'\in(0,1)$ we let
\[
L_{u, \eta,\eta'}(t):=\inf\left\{L>0 \mid \left\{x \,\mid\, u(t, x)\geq \eta\right\}\subseteq B_{L}\left(\left\{x \,\mid\, u(t, x)\geq \eta' \right\} \right)\right\}.
\]

\begin{thm}\label{thm:bw}
Let $f,f_{0}, M$ be from \eqref{h:stat0} and such that $f\in \mathcal{F}(f_0,M,\zeta,\xi)$ for some $\zeta<c_0^2/4$ and $\xi>0$, and let $0\le u\le 1$ solve \eqref{genhomeq} for some $\om\in\Om$.
If $d\le 3$, $\eta_0$ is from Lemma \ref{lem:slowspread}, $\sigma_u:=\sup_{\eta\in(0,1)}\eta e^{\sqrt\zeta\,L_{u,\eta,1-\eta_0}(0)}<\infty$, and
\begin{equation}\label{eq:icsub}
\De u(0,x)+f(x, u(0,x), \om)\geq 0 
\end{equation}
in the sense of distributions on $\bbR^d$,
 then we have \eqref{def:bw} and
\[
\inf_{\substack{(t,x)\in (T_{\eta}, \infty)\times \RR^{d}\\u(t,x)\in [\eta, 1-\eta]}}u_{t}(t,x)\geq m_{\eta} \qquad\text{for each $\eta>0$ and some $m_\eta>0$,}
\]
with $\ell_\eta,m_\eta$  depending only on $\eta,f_0,M,\zeta,\xi$ and $T_\eta$ also on $\sigma_u$.
\end{thm}

%
%

{\it Remarks.}  1.  Specifically, this is Remark 2 after Theorem~2.5 in \cite{andrejbd}, while Theorem 2.5 itself is a stronger result for ignition reactions, only requiring $L_{u, \eta, 1-\eta_0}(0)<\infty$ for each $\eta>0$.  The claim about $\ell_\eta,m_\eta,T_\eta$ follows from the proof of the remark in Section 5 of \cite{andrejbd}.  (In the case of Theorem 2.5, dependence of $T_\eta$ on $\sigma_u$ is replaced by its dependence on  the ignition temperature $\theta$ and $L_{u,h,1-\eta_0}(0)$ for some $h=h(\eta,f_0,M,\zeta,\xi,\theta)$.)  
We  note that these quantities do not explicitly depend on $f$ and $\om$, and the dependence on $u(0,\cdot)$ (of $T_\eta$ only) is only via $\sigma_u$.
\smallskip

2. The hypothesis \eqref{eq:icsub} guarantees that $u_{t}\geq 0$ because $w:=u_t$ satisfies
\begin{equation*}
w_t=\De w+f_{u}(x, 0, \om)w.
\end{equation*}
It will suffice for us to apply Theorem \ref{thm:bw} to such solutions here.
\smallskip

3.  The limitation to $d\le 3$ is not just technical, as it is proved in Theorem~2.4 in \cite{andrejbd} that in dimensions $d\ge 4$ there are $f$ as above for which typical solutions to \eqref{genhomeq} do not have bounded width.  This also is the sole reason for Theorems \ref{T.1.4}(i) and \ref{T.1.6}(i) being restricted to $d\le 3$ (but see Conjecture \ref{C.1.12} above).
\smallskip


%

Theorem \ref{thm:bw} shows that when studying the Wulff shape for \eqref{genhomeq}, it may be advantageous to (equivalently, see the start of Section \ref{sec:sspeed} below) define it with compactly supported initial data satisfying \eqref{eq:icsub} rather than those from Definition \ref{D.1.3}.  Existence of such functions is guaranteed by the following lemma.

\begin{lem}\label{lem:bump}
There exists a compactly supported (radially symmetric) $v: \RR^{d}\rightarrow [0,\infty)$ 
that satisfies $\De v + f_0(v)\ge 0$ in the sense of distributions on $\bbR^d$, as well as $\|v\|_\infty=v(0)=1-\eta_0$ (with $f_0,\eta_0$ from Lemma \ref{lem:slowspread}). 
\end{lem}

\begin{proof}
Since we must have $1-\eta_0>\theta_0$ for Lemma \ref{lem:slowspread} to hold, $f_0$ is bounded away from 0  near $1-\eta_0$.  Then there is small $r>0$ such that for any $R\ge 1$, with $R':=R+r$ and $R'':=R+r+r^{-2}$, the function
\[
v(x):=
\begin{cases}
1-\eta_0 & |x|\le R \\
1-\eta_0 -(|x|-R)^3 & |x|\in(R,R'] \\
\max \left\{1-\eta_0 -r^3 -3r^2 \left[ (|x|-R')- \frac{r^2}2(|x|-R')^2 \right],0 \right\} & |x|\in (R', R''] \\
0 & |x|> R''
\end{cases}
\]
satisfies $\De v + f_0(v)\ge 0$ on $B_{R'}(0)$.  If now $R\ge\frac{d-1}{r^2}$, then the inequality is satisfied (in the sense of distributions) on $B_{R''}(0)$, and therefore also on $\bbR^d$ because $v(x)=0$ when $|x|=R''$.
\end{proof}



We also recall the ergodic theorems that we will need here.

\begin{thm}[Wiener's ergodic theorem \cite{becker}]\label{thm:et}
If \eqref{h:erg} holds, then for each $g\in L^{1}(\Om)$, there is $\Om_{0}\subseteq \Om$ with $\PP[\Om_{0}]=1$ such that for all $\om\in \Om_{0}$, 
\begin{equation}\label{eq:et}
\lim_{r\rightarrow \infty} \aint_{B_{r}}g(\mscr{T}_{y}\om)\, dy= \int_\Om g(\om)\,d\PP.
\end{equation}
\end{thm}

\begin{thm}[Kingman's subadditive ergodic theorem \cite{kingman}]\label{thm:set}
Assume that $g_{n}\in L^1(\Om)$ is a sequence of measurable functions on $(\Om, \mathcal{F}, \PP)$ such that  
\[
\inf_{n\in\mathbb N} \frac 1n\int_\Om g_n(\om)\,d\PP>-\infty,
\] 
and that  $\tilde{\mscr{T}}:\Om\rightarrow \Om$ is  measure-preserving bijection satisfying 
\begin{equation*}
g_{m+n}(\cdot)\leq g_{m}(\cdot)+g_{n} \left(\tilde{\mscr{T}}^m(\cdot) \right)
\end{equation*}
for any $m,n\in\mathbb N$.
Then $\frac 1n g_{n}$ converges almost everywhere on $\Om$ to some $g\in L^1(\Om)$ as $n\to\infty$, and 
\begin{equation*}
\int_\Om g(\om)\, d\PP=\lim_{n\rightarrow\infty}\frac{1}{n}\int_\Om g_{n}(\om)\, d\PP=\inf_{n\ge 1}\frac{1}{n}\int_\Om g_{n}(\om)\, d\PP.
\end{equation*}
If $\left\{\tilde{\mscr{T}}^m\right\}_{m\in\mathbb Z}$ is ergodic, then $g$ is almost everywhere constant.
\end{thm}

We end this section with the proofs of Theorems \ref{T.1.6a} and \ref{T.1.6}(i).  We only need to prove the first claim in each part of Theorem \ref{T.1.6a}, as the second claims will follow once we prove the results they refer to.

\begin{proof}[Proof of Theorem \ref{T.1.6a}]
Extend $f$ to $\bbR^d\times(\bbR\setminus[0,1])$ by letting it be 0 there, let $M$ be its Lipschitz constant, and let it be $(l_1,\dots,l_d)$-periodic on $\mathbb R^d$.  Fix any $e\in\mathbb S^{d-1}$.   It is proved in \cite{berhamel} that  \eqref{1.1} has a pulsating front solution with speed $c^*(e)>0$ and of the form $u(t,x)=U(x\cdot e -c^*(e)t,x)$.  Here $U\in C^{1}(\bbR^{d+1})$ is  $(l_1,\dots,l_d)$-periodic in the second argument and satisfies  $\lim_{s\to -\infty} U(s,x)=1$ and $\lim_{s\to \infty} U(s,x)=0$ uniformly in $x\in\mathbb R^d$, as well as $U_s<0$ (and hence $u_t>0$).  We note that $u$ itself is a classical solution and hence in $C^{1,2}(\bbR\times\bbR^d)$, although we do not need this here.  It follows that for each $\eta>0$ we have
\begin{equation} \label{2.10}
m_\eta:=\inf_{\substack{(t,x)\in \RR^{d+1}\\u(t,x)\in [\eta, 1-\eta]}}u_{t}(t,x) >0.
\end{equation}
We also note that for $f$ from (ii) (i.e., ignition), this speed is unique, while for non-ignition $f$ the speed is not unique but there is a minimal speed (which is denoted $c^*(e)$).  In the latter case, there is a sequence $c_n\nearrow c^*(e)$ and Lipschitz ignition reactions $f_n(x,u):=g_n(u)f(x,u)$ with $0\le g_n\le 1$ (so $f_n\le f$) and $g_n=1$ on $[1-\theta,1]$ such that $c_n$ is the unique speed of a pulsating front in direction $e$ for $f_n$ (for each $n$).  In fact, $c^*(e)$ and the corresponding pulsating front in direction $e$ for $f$ are in this case  obtained in \cite{berhamel} as limits of these objects for the reactions $f_n$.

(ii) Let $\eta:=\frac12\min\{\theta,\theta'\}>0$, and consider any $\alpha\in(0,\eta]$, with
\[
u^{\pm\alpha}(t,x):= u((1\pm M\alpha m_{\eta}^{-1})t,x) \pm \alpha.
\]
Then \eqref{2.10} and the fact that $u$ solves \eqref{1.1} and $f$ is non-increasing in $u\in(-\infty,2\eta]$ as well as in $u\in[1-2\eta,\infty)$ show that $u^{-\alpha}$ is a subsolution and $u^{+\alpha}$ is a supersolution to \eqref{1.1}.  Moreover, the properties of $U$ show that 
\[
u^{-\alpha}(-t',\cdot)\le (1-\alpha)\chi_{\{x\cdot e<0\}} \le u(t',\cdot)
\]
 for some $t'\ge 0$.  Using also 
 Lemma \ref{lem:slowspread}, and $\lim_{\alpha\to 0} (1- M\alpha m_{\eta}^{-1})c^*(e)=c^*(e)$, it now follows that \eqref{1.1} has front speed $c^*(e)$ in direction $e$.  This speed is also exclusive (with $\beta_{K,e}(\alpha)=\alpha$ for all sufficiently small $\alpha>0$ once $K$ is arbitrary but fixed) because
\[
\chi_{\{x\cdot e<0\}} + \alpha \chi_{\{x\cdot e\ge 0\}} \le u^{+\alpha}(t'',\cdot)
\]
 for some $t''\ge 0$ and $\lim_{\alpha\to 0} (1+ M\alpha m_{\eta}^{-1})c^*(e)=c^*(e)$.

(i) For each $n$ as above, let $u^n$ be the pulsating front  in direction $e$ for $f_n$ (with speed $c_n$) and let $M_n$ be the Lipschitz constant for $f_n$.  Also define $\eta_n,m_{\eta_n}$ as above but for $f_n$ (so they also depend on $n$).  If now
\[
u^{n,-\alpha}(t,x):= u^n((1- M_n\alpha m_{\eta_n}^{-1})t,x) - \alpha
\]
for $\alpha\in(0,\eta_n]$, then $u^{n,-\alpha}$ is again a subsolution of \eqref{1.1} and  we  have
\[
u^{n,-\alpha}(-t'_n,\cdot)\le (1-\alpha)\chi_{\{x\cdot e<0\}} \le u(t'_n,\cdot)
\]
 for some $t'_n\ge 0$. Since  $\lim_{n\to\infty} \lim_{\alpha\to 0} (1- M\alpha m_{\eta_n}^{-1})c_n=c^*(e)$, it again follows  that \eqref{1.1} has front speed $c^*(e)$ in direction $e$.
\end{proof}

\begin{proof}[Proof of Theorem \ref{T.1.6}(i)]
Let $c^*(e)$ be the deterministic front speed in direction $e$, and let $u_{\om,e}$ be from Definition \ref{D.1.3}. Let $v$ be from Lemma \ref{lem:bump} for $d=1$ and let $\tilde u_{\om,e}$ solve \eqref{genhomeq} with initial data
\[
\tilde u_{\om,e}(0,x) = v(\max\{x\cdot e,0\}). 
\]
Lemma \ref{lem:slowspread}, applied to \eqref{1.3}, and the comparison principle show that there is $t'=t'(f_0)\ge 0$ such that 
\[
u_{\om,e}(0,\cdot) \le \tilde u_{\om,e}(t',\cdot) \qquad\text{and}\qquad  \tilde u_{\om,e}(0,\cdot) \le u_{\om,e}(t',\cdot).
\]
Hence the comparison principle implies that the definition of the front speed in direction $e$ is unchanged if we use $\tilde u$ in place of $u$.  Let us do so.  

Since \eqref{eq:icsub} holds for $\tilde u_{\om,e}$ due to   Lemma \ref{lem:bump}, Theorem \ref{thm:bw} applies to $\tilde u_{\om,e}$ and yields $m_\eta,T_\eta$ independent of $\om,e$ because $\sigma_{\tilde u_{\om,e}}$ also does not depend on them.  We now let
\[
\tilde u_{\om,e}^{+\alpha}(t,x):= \tilde u_{\om,e}((1+ M\alpha m_{\theta/2}^{-1})t,x) + \alpha,
\]
with $\theta$ from \eqref{h:stat0} and $\alpha\in(0,\frac\theta 2]$.
Similarly to the previous proof, this function is now a supersolution to \eqref{genhomeq} on $(T_{\theta/2},\infty)\times\bbR^d$, and Lemma~\ref{lem:slowspread} shows that it satisfies
\[
\tilde u_{\om,e}^{+\alpha}(t'',\cdot)\ge u_{\om,e,\alpha}(0,\cdot),
\]
with the right-hand side from Definition \ref{D.1.5} and $t''=t''(f_0,\alpha)\ge T_{\theta/2}$.  It follows from this, the comparison principle,  the fact that $\tilde u_{\om,e}$ propagates with speed $c^*(e)$ in direction $e$ (in the sense of Definition~\ref{D.1.3}) for almost all $\om\in\Om$, and $\lim_{\alpha\to 0} (1+ M\alpha m_{\theta/2}^{-1})c^*(e)=c^*(e)$ that $c^*(e)$ is a deterministic exclusive  front speed in direction $e$ for \eqref{genhomeq} (again with $\beta_{K,e}(\alpha)=\alpha$ for all sufficiently small $\alpha>0$  once $K$ is arbitrary but fixed).
\end{proof}

\section{Existence of Spreading Speeds and the Wulff Shape}
\label{sec:sspeed}

In this section we will prove Theorems \ref{T.1.4}(i) and \ref{T.1.8}(i).
We will first establish existence of a deterministic spreading speed in each direction, and then upgrade this to existence of a Wulff shape.
The key  will be to define an appropriate ``first passage time'' 
for spreading in any fixed direction, an approach that has been used extensively in the discrete setting of first passage percolation (see \cite{kestennotes, kestenpaper, fppbook} and references therein).

We only need to prove Theorem \ref{T.1.8}(i).  Then Theorem \ref{T.1.4}(i) will follow from the relationship 
\begin{equation} \label{3.000}
u_\om(0,\cdot) \le \tilde u_\om(t',\cdot) \qquad\text{and}\qquad  \tilde u_\om(0,\cdot) \le u_\om(t',\cdot)
\end{equation}
for the solutions $u_\om$ (from Definition \ref{D.1.2}) and  $\tilde u_\om$, where the latter solves \eqref{genhomeq} with initial data
$\tilde u_\om(0,\cdot) = v$
(with $v$ from Lemma \ref{lem:bump}) and $t'=t'(f_0)$.  (As in the proof of Theorem~\ref{T.1.6}(i) above, this follows from Lemma \ref{lem:slowspread}, applied to \eqref{1.3}, and the comparison principle.)  This yields  $\lim_{t\to\infty} \frac 1t {L_{u_\om , \eta}(t)}=0$ for each $\eta>0$ (so we can apply Theorem \ref{T.1.8}(i)) because \eqref{def:bw} for $\tilde u$ (see Theorem \ref{thm:bw}) then implies  \eqref{def:bw} for  $u$.  We note that this last claim also needs the fact that any super-level set of $\tilde u$ expands with a uniformly-bounded-above speed, proved in \cite{andrejbd}.  Alternatively, one can perform the proof below for the solutions $\tilde u$ instead of $u$, since \eqref{def:bw} for $\tilde u$ yields $\lim_{t\to\infty} \frac 1t {L_{\tilde u_\om, \eta}(t)}=0$ for each $\eta>0$, and then notice that \eqref{3.000} shows that Definition \ref{D.1.2} is equivalent to itself with $\tilde u$ in place of $u$.

So let us now  prove Theorem \ref{T.1.8}(i), assuming hypothesis $H$ and $\lim_{t\to\infty} \frac 1t {L_{u_\om , \eta}(t)}=0$ for the solutions from H(ii) and each $\eta>0$ (Definition \ref{D.1.2} is again equivalent to itself with $\theta_0$ in place of $\frac{1+\theta_0}2$ when H(ii) is assumed).  First note that the second claim in Lemma \ref{lem:slowspread} holds:

\begin{lem} \label{L.3.11}
Assume Hypothesis H(i,ii).  There is $c=c(F, \theta_0,R_0)>0$ such that for each $\eta>0$ there is $\la_\eta=\la_\eta(F, \theta_0,R_0)\geq 0$ such that the following holds.  If $0\le u\le 1$ solves \eqref{1.10} for some $\om\in\Om$ and
\begin{equation*}
u(t_{1}, \cdot)\geq \theta_{0}\chi_{B_{R_0}(x)}
\end{equation*}
for some $(t_{1}, x)\in [0,\infty)\times \RR^{d}$, then for each $t\geq t_{1}+\la_\eta$ we have \eqref{2.1}.  In particular, for $\eta\in(0,\min\{1-\theta_0,\frac 1{R_0}\}]$ and  $t_1\ge 0$ we have
\begin{equation*}
B_{c(t-t_1)-L_{u_\om , \eta}(t_1)} \left( \left\{x\in\RR^d \,|\, u_\om(t_1, x)\geq \eta\right\} \right)
  \subseteq  \left\{x\in\RR^d \,|\, u_\om(t, x)\geq 1-\eta\right\}
\end{equation*}
whenever $t\geq t_{1}+ \max\{\la_\eta, \frac 1c L_{u_\om , \eta}(t_1)\}$.
\end{lem}

\begin{proof}
If $\tau<\infty$ is such that $\inf_{\om\in\Om\,\&\, |x|<R_0+1} u_\om(\tau,x)\ge\theta_0$, then the first claim holds with any $c<\tau^{-1}$.   The second claim follows.
\end{proof}

For any $(y,z,\om)\in \RR^{d}\times \RR^{d}\times \Om$, we now let
\begin{equation}\label{def:tau}
\tau(y, z, \om):=\inf\left\{ t\geq 0\,\mid\, u(t,\cdot, \om; y)\geq \theta_{0}\chi_{B_{R_{0}}(z)} \right\},
\end{equation}
where $u(\cdot,\cdot, \om; y)$ solves \eqref{1.10} with $u(0,\cdot,\om;y)=\theta_0\chi_{B_{R_0}(y)}$.  (By  H(iii), $u(t,x, \om; y)=u_{\mscr{T}_y\om}(t,x-y)$.)  So $\tau(y,z,\om)$ can be thought of as the time of spreading from $y$ to $z$.
Notice that our PDE being of second order forces us to use $\theta_{0}\chi_{B_{R_{0}}(y)}$ and $\theta_{0}\chi_{B_{R_{0}}(z)}$ to define $\tau$, rather than just pointwise information. We next establish some useful properties of  $\tau$.

\begin{lem}\label{lem:tauprop}
Assume Hypothesis H.  Then
the function $\tau$ from \eqref{def:tau} satisfies the following.
\begin{enumerate}[(i)]
\item Subadditivity: For any $y_{1}, y_{2}, y_{3}\in \RR^{d}$ we have
\begin{equation}\label{eq:tausubadd}
\tau(y_{1}, y_{3}, \om)\leq \tau(y_{1}, y_{2}, \om)+\tau(y_{2}, y_{3}, \om).
\end{equation}
\item Stationarity: For any $y_{1}, y_{2},z\in \RR^{d}$ we have 
\begin{equation}\label{eq:taustat}
\tau(y_{1}, y_{2}, \mscr{T}_{z}\om)=\tau(z+y_{1}, z+y_{2}, \om)
\end{equation}
\item Linear upper bound: There exists  $C=C(F, \theta_0,R_0)$ such that for any $x_{1}, x_{2}, y_{1}, y_{2}\in \RR^{d}$ we have
\begin{equation}\label{eq:taugenlip}
\left|\tau(x_{1}, y_{1}, \om)-\tau(x_{2}, y_{2}, \om)\right|\leq C(|x_{1}-x_{2}|+|y_{1}-y_{2}|+1).
\end{equation}
In particular, for any $x, y\in \RR^{d}$ we have
\begin{equation}\label{3.3}
\left|\tau(x, y, \om)\right|\leq C(|x-y|+1).
\end{equation}
\end{enumerate}
\end{lem}

%

\begin{proof}
(i) follows from the comparison principle and (ii) from H(iii) (specifically, \eqref{h:stat}).  To prove (iii), by symmetry we only need to show
\begin{equation}\label{eq:1sidelipic}
\tau(x_{1}, y_{1}, \om)\leq \tau(x_{2}, y_{2}, \om)+C\left(|x_{1}-x_{2}|+|y_{1}-y_{2}|+1\right).
\end{equation}
By Lemma \ref{L.3.11}, there is $C=C(F, \theta_0,R_0)$ such that 
\begin{equation*}
u(C(|x_{1}-x_{2}|+1), \cdot, \om; x_{1})\geq \theta_{0}\chi_{B_{R_{0}}(x_{2})}.
\end{equation*}
Then the comparison principle and the definition of $\tau$ yield
\begin{equation*}
u(C(|x_{1}-x_{2}|+1)+\tau(x_{2},y_{2}, \om),\cdot, \om; x_{1})\geq \theta_{0}\chi_{B_{R_{0}}(y_{2})},
\end{equation*}
and after another application of \eqref{2.1} we obtain
\begin{equation*}
u(C(|x_{1}-x_{2}|+|y_{1}-y_{2}|+2)+\tau(x_{2},y_{2},\om), \cdot, \om; x_{1})\geq \theta_{0}\chi_{B_{R_{0}}(y_{1})}. 
\end{equation*}
Inequality \eqref{eq:1sidelipic} now follows after doubling $C$, and \eqref{3.3} is its special case with $x_1=x$, $y_1=y$, and $x_2=y_2=x$ because $\tau(x, x, \om)=0$.
\end{proof}

Using these properties and \eqref{h:erg}, we next show that for any $e\in \mathbb{S}^{d-1}$,
\begin{equation*}
\lim_{n\rightarrow\infty} \frac{\tau(0, ne, \om)}{n}
\end{equation*}
exists and is constant on a full measure subset of $\Om$. We also establish boundedness and Lipschitz continuity of this limit as a function of $e$.

\begin{lem}\label{lem:setapp}
Assume Hypothesis H.  Then for each $e\in \mathbb{S}^{d-1}$, there exists a constant $\overline{\tau}(e)$ and a set $\Om(e)\subseteq \Om$ with $\PP[\Om(e)]=1$ such that for each $\om\in \Om(e),$ we have
\begin{equation} \label{3.2}
\lim_{n\rightarrow\infty} \frac{\tau(0, ne, \om)}{n}=\overline{\tau}(e).
\end{equation}
Moreover, there is  $C=C(F, \theta_0,R_0)$ such that for any $e,e'\in \mathbb{S}^{d-1}$ we have
(with $c$ from Lemma \ref{L.3.11} and $c'$ from Lemma \ref{L.4.2}) 
\begin{equation}\label{eq:taubarbnds}
c\leq \frac 1{\overline{\tau}(e)}\leq  c'
\end{equation}
and
\begin{equation}\label{eq:taubarlip}
\max\left\{ \left| \overline{\tau}(e) - \overline{\tau}(e') \right|, \left| \frac 1{\overline{\tau}(e)} - \frac 1{\overline{\tau}(e')} \right| \right\} \leq C|e-e'|.
\end{equation}
\end{lem}

\begin{proof}
Fix $e\in \mathbb{S}^{d-1}$.  Lemma \ref{lem:tauprop}(i) shows that for  $m,n\geq 0$ we have
\begin{equation*}
\tau(0, (m+n)e, \om)\leq \tau(0, me, \om)+\tau(me, (m+n)e, \om).
\end{equation*}
Lemma \ref{lem:tauprop}(ii) shows that for $\mscr{T}'_{s}:=\mscr{T}_{se}$ for $s\in \RR$, 
\begin{equation*}
\tau(me, (m+n)e, \om)=\tau(0, ne, \mscr{T}_{me}\om)=\tau(0, ne, \mscr{T}'_{m}\om).
\end{equation*}
Also \eqref{3.3} shows that 
\begin{equation}\label{eq:taubdd}
\tau(me, ne, \om)\leq C\left(|n-m|+1\right)
\end{equation}
for some $C=C(F, \theta_0,R_0)$.
These statements and Theorem \ref{thm:set} with $g_{n}(\om):=\tau(0, ne, \om)$ and $\tilde{\mscr{T}}:=\mscr{T}_{e}$ now yield existence of
\begin{equation*}
\overline{\tau}(e, \om):=\lim_{n\rightarrow\infty}\frac{\tau(0, ne, \om)}{n} \, \in L^1(\Om)
\end{equation*}
for each $\om$ in some full measure set $\Om'(e)\subseteq\Om$.

In order to show that $\overline{\tau}(e, \om)$ is constant on a full measure subset of $\Om'(e)$, we need to use the ergodicity hypothesis \eqref{h:erg}. Note that the last statement in Theorem \ref{thm:set} does not apply directly as $\{\tilde{\mscr{T}}^{m}\}_{m\in \mathbb Z}$ need not be ergodic. By Lemma \ref{lem:tauprop}(ii,iii), for any $(y,\om)\in \RR^{d}\times\Om'(e)$ we have
\[
\left|\tau(0, ne, \mscr{T}_{y}\om)-\tau(0, ne, \om)\right| =\left|\tau(y, y+ne, \om)-\tau(0, ne, \om)\right| 
\leq 2C(|y|+1),
\]
so that
\begin{equation*}
\lim_{n\rightarrow\infty}\frac{\left|\tau(0, ne, \mscr{T}_{y}\om)-\tau(0, ne, \om)\right|}{n}= 0.
\end{equation*}
It then follows (after enlarging $\Om'(e)$ to the translation invariant set $\bigcup_{y\in \RR^d} \mscr{T}_{y}\Om'(e)$  and extending $\overline\tau(e,\cdot)$ to it)  that
\begin{equation*}
\overline{\tau}(e, \mscr{T}_{y}\om)=\overline{\tau}(e, \om)
\end{equation*}
for any $(y,\om)\in \RR^{d}\times\Om'(e)$.
Now  \eqref{h:erg} implies that $\overline{\tau}(e, \om)$ is a constant $\overline{\tau}(e)$ on some full measure set  $\Om(e)\subseteq\Om'(e)$.

The lower bound in \eqref{eq:taubarbnds} follows immediately Lemma \ref{L.3.11} and the upper bound from Lemma \ref{L.4.2} (since $\theta_{0}\chi_{B_{R_{0}}}$ is compactly supported), because these yield for each $\om\in\Om$ and with $o(1)=o(n^0)$,
\[
\frac 1{c'}+o(1)\le \frac{\tau(0, ne, \om)}n \le \frac 1{c}+o(1).
\]
Finally,  \eqref{eq:taugenlip} shows that
\begin{equation*}
\left|\tau(0,ne, \om)-\tau(0, ne', \om)\right|\leq C\left(n|e-e'|+1\right),
\end{equation*}
and by taking $\om\in \Om(e)\cap \Om(e')$ and then $n\to\infty$ we obtain
\[
\left|\overline{\tau}(e)-\overline{\tau}(e')\right|\leq C|e-e'|.
\]
Then \eqref{eq:taubarlip} follows from this and \eqref{eq:taubarbnds}.
\end{proof}

This result, together with Lemma \ref{L.3.11}, suggests that  $u_\om=u(\cdot,\cdot,\om;0)$ should have spreading speed 
\begin{equation}\label{eq:ssrep}
w(e):=\frac{1}{\overline{\tau}(e)} \in [c,c']
\end{equation}
in the direction $e\in \mathbb{S}^{d-1}$ (see Remark 1 after Definition \ref{D.1.2}) whenever $\om\in\Om(e)$  and $\lim_{t\to\infty} \frac 1t {L_{u_\om , \eta}(t)}=0$ for each $\eta>0$.
Moreover, if the convergence in \eqref{3.2} is uniform in $e$, then the super-level sets of $u$ should (after a scaling by $t$) acquire the Wulff shape 
\eqref{eq:wulffdef} as $t\to\infty$.  We will next show that this indeed happens in the case at hand.

\begin{proof}[Proof of Theorem \ref{T.1.8}(i)]
Let $\Om(e)$ be the set from Lemma \ref{lem:setapp} for any $e\in\mathbb S^{d-1}$, and assume without loss that $\lim_{t\to\infty} \frac 1t {L_{u_\om , \eta}(t)}=0$  for each $\eta>0$ and $\om\in \Om(e)$ (otherwise restrict $\Om(e)$ to such $\om$).  

Let  $Q$ be a countable dense subset of $\mathbb{S}^{d-1}$ and define
\begin{equation}\label{eq:Om0def}
\Om':=\bigcap_{e\in Q} \Om(e),
\end{equation}
so that $\PP[\Om']=1$.
We will prove that for each $\om\in\Om'$, $\delta\in(0,\frac 13)$, and $\eta'\in(0,1)$ we have
\begin{equation} \label{3.4}
(1-2\delta)\mathcal{S}t \subseteq \{ x\in\RR^d\,|\, u_\om(t,x)\ge \eta'  \} \subseteq (1+3\delta)\mathcal{S}t
\end{equation}
for all large enough $t$, which yields the claim with $\Omega_0:=\Om'$.  

Fix such $\om,\delta$ and any $\eta\in(0,\min\{1-\theta_0,\frac 1{R_0}\}]$. Let $t_{\delta,\eta}$ be such that 
\begin{equation} \label{3.333}
\max \left\{\lambda_\eta, \frac {L_{u_\om , \eta}(t)}c \right\}\le \delta' t \qquad \text{for all $t\ge (1-\delta')t_{\delta,\eta}$},
\end{equation}
where $c,\lambda_\eta$ are from Lemma \ref{L.3.11} and $\delta':=\min \{  \frac \delta{c'(C+2)} , \frac{\delta c}{4c'} \}$  (with $C$ from \eqref{eq:taubarlip}).
  Let $Q'\subseteq Q$ be finite and satisfying $\mathbb{S}^{d-1} \subseteq \bigcup_{e\in Q'} B_{\delta'}(e)$, and let $N\in\mathbb N$ be such that
\[
\sup_{e\in Q' \,\&\, n\ge N} \left|\frac{\tau(0, ne, \om)}{n}-\overline{\tau}(e)\right|\leq \delta'
\]
and 
\[
\sup_{e\in Q' \,\&\, n\le N} \tau(0, ne, \om) \le N (\overline{\tau}(e)+ \delta').
\]

Consider any  $t\ge \max\{N (c^{-1}+\delta'),t_{\delta,\eta}\}$.  From $(C+2)\delta'\le \frac \delta{c'}$, \eqref{eq:taubarbnds}, and \eqref{eq:taubarlip} we have 
\[
\bar\tau(e)-\bar\tau(e')+(1-\delta)\delta'\le \delta\bar\tau(e)
\]
 and thus $\frac{1-\delta}{\bar\tau(e')}\le \frac1{\bar\tau(e)+\delta'}$ whenever $|e-e'|\le\delta'$. This and $\frac{1-\delta}{\bar\tau(e')}t \le c't$ show that $(1-\delta)\mathcal{S}t$ is a subset of the union of the balls $B_{\delta' c't+1}(ne)$ with vectors $e\in Q'$ and integers $n\in[0, \frac t{\bar\tau(e)+\delta'}]$.  From \eqref{3.333} and the first claim in Lemma \ref{L.3.11} it follows that
 \[
 u_\om(s,\cdot)\ge (1-\eta)\chi_{B_{\delta' c't+1}(ne)}
 \]
 for each such ball and any $s\ge \tau(0,ne,\om)+c^{-1}(\delta' c't+1)$.  But  \eqref{eq:taubarbnds}, our choice of $N$, and $\delta\ge 2c^{-1}\delta' c'$ show that $s:=(1+\delta)t$ satisfies this for each such ball as long as $t$ is large enough, hence
 \[
  u_\om((1+\delta)t,\cdot)\ge (1-\eta)\chi_{(1-\delta)\mathcal{S}t}
 \]
 for such $t$.
 Since this holds for any small $\eta>0$ and $(1-2\delta)(1+\delta)\le 1-\delta$,
the first inclusion in \eqref{3.4} follows for any $\eta'\in(0,1)$ and all large $t$.

For the second inclusion, assume there is $s\ge \max\{N (c^{-1}+\delta'),t_{\delta,\eta}\}$ and $x\in\RR^d\setminus (1+3\delta)\mathcal{S}s$ such that $u_\om(s,x)\ge\eta$.  Then  there are $e'\in\mathbb S^{d-1}$ and   $t \ge \frac s{1-\delta}$ 
such that $x=\frac{1+\delta}{\bar\tau(e')}te'$ (because $\frac{1+\delta}{1-\delta}\le1+3\delta$ for $\delta\in(0,\frac 13)$). 
 From $(C+2)\delta'\le \frac \delta{c'}$, \eqref{eq:taubarbnds}, and \eqref{eq:taubarlip} we have 
\[
\bar\tau(e')-\bar\tau(e)+(1+\delta)\delta'\le \delta\bar\tau(e)
\]
 and thus $\frac{1+\delta}{\bar\tau(e')}\ge \frac1{\bar\tau(e)-\delta'}$ whenever $|e-e'|\le\delta'$.  This and $\frac{1+\delta}{\bar\tau(e')}t\le 2c't$ show that $\partial(1+\delta)\mathcal{S}t$ is a subset of the union of the balls $B_{2\delta' c't+1}(ne)$ with vectors $e\in Q'$ and integers $n> \frac t{\bar\tau(e)-\delta'}$.  Hence $x\in B_{2\delta' c't+1}(ne)$ for one such ball.  From $u_\om(s,x)\ge\eta$, \eqref{3.333} and the second claim in Lemma \ref{L.3.11} it follows that
 \[
 u_\om(s',\cdot)\ge (1-\eta)\chi_{B_{R_0}(ne)}
 \]
 for each  $s'\ge s+c^{-1}(3\delta' c't+1+R_0)$.
 From $\delta\ge 4c^{-1}\delta' c'$ and $t \ge \frac s{1-\delta}$ it follows that $s':=t$ satisfies this  as long as $s$ is large enough, and so $\tau(0,ne,\om)\le t$.  But \eqref{eq:taubarbnds} and our choice of $N$ show that this contradicts $n> \frac t{\bar\tau(e)-\delta'}$.  We must therefore have $u_\om(s,x)<\eta$ for all $x\in\RR^d\setminus (1+3\delta)\mathcal{S}s$ whenever $s$ is large enough.  Since this holds for any small enough $\eta>0$,  the second inclusion in \eqref{3.4} follows for any   $\eta'\in(0,1)$. 
 
 This proves the first claim in (i).  The second claim follows from Lemma \ref{L.3.4} below and from the part of Theorem \ref{T.1.8}(ii) corresponding to Theorem \ref{T.1.4}(ii) (which we will prove in Section \ref{s:fs}).
\end{proof}

We now show that, in fact, a stronger version of existence of the Wulff shape holds, including certain uniformity of the relevant limits with respect to shifts of the initial data.  This will be crucial in the proof of existence of deterministic front speeds in the next section.

\begin{prop}\label{P.3.3}
Assume Hypothesis H.  If  \eqref{1.10} has a deterministic Wulff shape  $\mathcal{S}$,
then there is $\Om_0\subseteq \Om$
with $\PP[\Om_0]=1$ such that for any $\om\in \Om_0$ and $\La,\delta>0$, 
the solutions $u(t,x,\om;y):=u_{\mscr{T}_y\om}(t,x-y)$ to  \eqref{1.10} 
(for which $u(0,\cdot,\om;y)=\theta_0\chi_{B_{R_0}(y)}$) satisfy 
\begin{align}
\lim_{t\rightarrow\infty} \inf_{|y|\leq \La t} \, \inf_{x\in (1-\delta)\mathcal{S}t}u(t,x+y, \om; y) & =1, \label{eq:lbsxi}
\\ \lim_{t\rightarrow\infty} \sup_{|y|\leq \La t} \, \sup_{x\notin (1+\delta)\mathcal{S}t}u(t,x+y, \om; y) & =0.  \label{eq:ubsxi}
\end{align}
\end{prop}

{\it Remark.}  Comparison principle shows that if \eqref{eq:lbsxi} and \eqref{eq:ubsxi} hold for some $\om$, then they also hold for $\mscr{T}_{z}\om$ for any $z\in\RR^d$.  Hence $\Om_0$ can be chosen to be translation invariant.

\begin{proof}
Let $\Omega'\subseteq\Om$ be such that $\mathbb P[\Om']=1$ and \eqref{eq:lbsxi} and \eqref{eq:ubsxi} hold for any $\om\in\Om'$ and $\La=0$ (i.e., only for $y=0$).  We will extend this to any $\La$ via a combination of Egorov's theorem and Wiener's ergodic theorem. 

For  $(\om,t,\eta)\in\Om\times[0,\infty)\times (0,1)$ let
\[
\Gamma_{\om,t, \eta}:=\{ x\in\RR^d\,|\, u(t,x,\om;0)\ge \eta \}.
\]
By  \eqref{eq:lbsxi} and \eqref{eq:ubsxi} for any $\om\in\Om'$ and $y=0$, and by Egorov's Theorem, for each $m\in\mathbb N$ there are $\tau_{m,\eta}\ge 1$ and $D_{m, \eta}\subseteq \Om$ with $\PP[D_{m, \eta}]\ge 1-2^{-md-1}$ such that for each $\om\in D_{m, \eta}$ and $t\geq \tau_{m, \eta}$, 
\begin{equation*}
(1-2^{-m})\calS t \subseteq \Ga_{\om,t, \eta} \subseteq  (1+2^{-m})\calS t.
\end{equation*}
Theorem \ref{thm:et} shows that there exists a set $\Om_{m, \eta}\subseteq \Om$ with $\PP[\Om_{m,\eta}]=1$ such that for each $\om\in \Om_{m, \eta}$, 
\begin{equation*}
\lim_{r\rightarrow\infty}\aint_{B_{r}}\chi_{D_{m, \eta}}(\mscr{T}_{y}\om)\, dy=\PP[D_{m, \eta}]\geq 1-2^{-md-1}.
\end{equation*}
Hence for each $\om\in \Om_{m,\eta}$ and $\La>0$, there is $r_{\om, \La,m, \eta}\ge 2\tau_{m, \eta}$ such that 
\begin{equation*}
\left| \left\{ y\in B_{2t\La}(0) \,\mid\, \mscr{T}_{y}\om\in D_{m, \eta} \right\} \right| \geq (1-2^{-md}) \left|B_{2t\La}(0) \right|
\end{equation*}
for all $t\geq r_{\om, \La,m, \eta}$, with $\left|\cdot \right|$ the Lebesgue measure. 

Fix any  $\La>0 $, $\om\in \Om_{m,\eta}$, and $t\geq r_{\om, \La,m, \eta}$, and let $y\in B_{\La t}(0)$ be arbitrary.  Then there is $z\in B_{2\La t}(0)$ such that $|z-y|\leq 2^{1-m} \La t$ and $\mscr{T}_{z}\om\in D_{m,\eta}$.  From the first claim in Lemma \ref{L.3.11} we have that 
\[
u \left( 2c^{-1} 2^{1-m}\La t, \cdot,  \mscr{T}_{y}\om;0 \right) \ge  \theta_0 \chi_{B_{R_{0}}(z-y)}
\]
and
\[
u \left( 2c^{-1} 2^{1-m}\La t, \cdot,  \mscr{T}_{z}\om;0 \right) \ge  \theta_0\chi_{B_{R_{0}}(y-z)}
\]
provided $c^{-1} 2^{1-m}\La t \ge \max\{ \la_{1-\theta_0}, c^{-1}{R_{0}} \}$ (which holds for all large $t$). 
But then from $\mscr{T}_{z}\om\in D_{m, \eta}$, $r_{\om, \La,m,\eta}\ge2 \tau_{m,\eta}$, and \eqref{eq:ssrep} we obtain
\begin{equation*}
\Ga_{\mscr{T}_{y}\om,t, \eta} \supseteq \Ga_{\mscr{T}_z\om,  (1-2c^{-1} 2^{1-m}\La)t, \eta} +(z-y) \supseteq [(1-2^{-m})(1-2c^{-1} 2^{1-m}\La)-c^{-1}2^{1-m} \La ] \calS t
\end{equation*}
and
\begin{equation*}
\Ga_{\mscr{T}_{y}\om,t, \eta} \subseteq \Ga_{\mscr{T}_z\om,  (1+2c^{-1} 2^{1-m}\La)t, \eta} +(z-y) \subseteq [(1+2^{-m})(1+2c^{-1} 2^{1-m}\La)+c^{-1}2^{1-m} \La ] \calS t
\end{equation*} 
for all $m\ge \log_2\frac{8\La}c$ (so that $(1-2c^{-1} 2^{1-m}\La)t\ge \tau_{m,\eta}$), any $\om\in\Om_{m, \eta}$, any large enough $t$ (depending on $F,\theta_0,R_0,\Lambda, m,\eta,\om$),
and any $y\in B_{\La t}(0)$.  Then for any $\delta,\La>0$, $\eta\in\mathbb Q\cap(0,1)$, and $\omega\in\Omega_0:=\bigcap_{\eta\in \mathbb{Q}\cap (0,1)}\bigcap_{m=1}^{\infty} \Om_{m, \eta}$ (so that $\PP[\Om_{0}]=1$) we obtain
\[
(1-\delta) \calS t \subseteq \Gamma_{\mscr{T}_{y}\om,t, \eta}=\{ x\in\RR^d\,|\, u(t,x+y,\om;y)\ge \eta \} \subseteq (1+\delta) \calS t 
\]
for all large enough $t$ and any $y\in B_{\La t}(0)$ (by first choosing $m$ above large enough, depending on $c,\La,\delta$). Since this holds for any $\eta\in\mathbb Q\cap (0,1)$,  \eqref{eq:lbsxi} and \eqref{eq:ubsxi} follow for any $\delta,\La>0$ and $\om\in\Om_0$.
%
\end{proof}

A similar argument for deterministic (exclusive) front speeds yields the following result.  Recall Definitions \ref{D.1.3} and \ref{D.1.5}, with the former having $u_{\om,e}(0,\cdot) = \theta_0\chi_{\{x\cdot e< 0\}}$ in the setting of Hypothesis H.

\begin{prop}\label{P.3.5}
Assume Hypothesis H. 

(i) If  \eqref{1.10} has a deterministic front speed $c^*(e)$ in direction $e\in\mathbb S^{d-1}$,
then there is $\Om_e\subseteq \Om$
with $\PP[\Om_e]=1$ such that for any $\om\in \Om_e$, $\La>0$, and compact $K\subseteq  \{x\cdot e>0\}\subseteq\bbR^d$,
the solutions $u(t,x,\om;y,e):=u_{\mscr{T}_y\om,e}(t,x-y)$ to  \eqref{1.10} 
(for which $u(0,\cdot,\om;y,e)=\theta_0\chi_{\{x\cdot e< y\cdot e\}}$) satisfy 
\begin{align*}
\lim_{t\to \infty}  \inf_{|y|\le \La t}\, \inf_{x\in (c^*(e)e-K)t} u(t,x+y,\om;y,e) &= 1,  
\\ \lim_{t\to \infty}  \sup_{|y|\le \La t}\, \sup_{x\in (c^*(e)e+K)t} u(t,x+y,\om;y,e) &= 0,
\end{align*}

(ii) If  \eqref{1.10} has a deterministic exclusive front speed $c^*(e)$ in direction $e\in\mathbb S^{d-1}$,
then there is $\Om_e'\subseteq \Om$
with $\PP[\Om_e']=1$ and for any compact $K\subseteq  \{x\cdot e>0\}\subseteq\bbR^d$ there is $\beta_{K,e}':(0,1]\to(0,1]$ with $\lim_{\alpha\to 0}\beta_{K,e}'(\alpha)=0$ such that 
the solutions $u(t,x,\om;y,e,\alpha):=u_{\mscr{T}_y\om,e,\alpha}(t,x-y)$ to  \eqref{1.10} 
(for which $u(0,\cdot,\om;y,e,\alpha)=\chi_{\{x\cdot e< y\cdot e\}} + \alpha \chi_{\{x\cdot e\ge y\cdot e\}}$) satisfy for each  $\om\in \Om_e'$ and $\La,\alpha>0$, 
\[
 \limsup_{t\to \infty}  \sup_{|y|\le \La t}\, \sup_{x\in (c^*(e)e+K)t} u(t,x+y,\om;y,e,\alpha) \le\beta_{K,e}'(\alpha).
\]
\end{prop}

{\it Remark.}  The sets $\Om_e,\Om_e'$ can again be chosen  translation invariant.

\begin{proof}
(i) For each $j\ge 1$, let $K_j\subseteq  \{x\cdot e>0\}$ be a  compact  such that $K_j\subseteq K_{j+1}$ and $\bigcup_{j\ge 1} K_j=\{x\cdot e>0\}$.  For  $(\om,t,\eta)\in\Om\times[0,\infty)\times (0,1)$ let
\[
\Gamma_{\om,t, \eta}:=\{ x\in\RR^d\,|\, u(t,x,\om;0,e)\ge \eta \}.
\]
As in the proof of Proposition \ref{P.3.3}, using that the claim holds for $\La=0$, we can find a full-measure set $\Om_e\subseteq\Om$ such that the following holds  for any $\om\in\Om_e$, $\La>0$, $\eta\in\mathbb Q\cap(0,1)$, $j\ge 1$, and  $m\ge \log_2\frac{8\La}c$: for any large enough $t$ and any $y\in B_{\La t}(0)$, there is $z\in B_{2\La t}(0)$ such that $|z-y|\leq 2^{1-m} \La t$ and 
\begin{equation*}
\Ga_{\mscr{T}_{y}\om,t, \eta} \supseteq \Ga_{\mscr{T}_z\om,  (1-2c^{-1} 2^{1-m}\La)t, \eta} +(z-y) \supseteq 
\left[(c^*(e)e-K_j)(1-2c^{-1} 2^{1-m}\La)t \right]^0_{2^{1-m}\La t}
\end{equation*}
and
\begin{equation*}
\Ga_{\mscr{T}_{y}\om,t, \eta} \subseteq \Ga_{\mscr{T}_z\om,  (1+2c^{-1} 2^{1-m}\La)t, \eta} +(z-y) \subseteq 
\RR^d\setminus \left[ (c^*(e)e+K_j)(1+2c^{-1} 2^{1-m}\La)t \right]^0_{2^{1-m}\La t}
\end{equation*}
Given any compact $K\subseteq  \{x\cdot e>0\}$, it now suffices to take $j$ such that $K\subseteq K_j^0$ and large enough $m$ (depending on $c,\La,K,j$) so that
\[
\left[(c^*(e)e-K_j)(1-2c^{-1} 2^{1-m}\La) \right]^0_{2^{1-m}\La } \supseteq c^*(e)e-K
\]
and
\[
\left[(c^*(e)e+K_j)(1+2c^{-1} 2^{1-m}\La) \right]^0_{2^{1-m}\La } \supseteq c^*(e)e+K.
\]
Indeed, since
\[
\Gamma_{\mscr{T}_{y}\om,t, \eta}=\{ x\in\RR^d\,|\, u(t,x+y,\om;y,e)\ge \eta \},
\]
taking $t\to\infty$ then yields
\begin{align*}
\liminf_{t\to \infty}  \inf_{|y|\le \La t}\, \inf_{x\in (c^*(e)e-K)t} u(t,x+y,\om;y,e) &\ge \eta,  
\\ \limsup_{t\to \infty}  \sup_{|y|\le \La t}\, \sup_{x\in (c^*(e)e+K)t} u(t,x+y,\om;y,e) &\le \eta
\end{align*}
for any $\om\in\Om_e$, $\La>0$, $\eta\in\mathbb Q\cap(0,1)$, and compact $K\subseteq  \{x\cdot e>0\}$.  The result follows.

(ii)  This is analogous, with $\beta_{K,e}':=\beta_{K_j,e}$ for the chosen $j$ (where $\beta_{K_j,e}$ is from Definition \ref{D.1.5}).
\end{proof}

Finally, we show that $w$ is continuous, and it is constant if $F$ in \eqref{1.10} (or $f$ in \eqref{genhomeq}) is isotropic.

\begin{lem} \label{L.3.4}
Assume Hypothesis H and 
let $\mathcal{S}$ be from \eqref{eq:wulffdef}, with $w(e)$ from \eqref{eq:ssrep} for each $e\in \mathbb{S}^{d-1}$.
Then $w$ is Lipschitz,
and if $F$ is isotropic, then $w(e)\equiv w$ is constant and $\calS=B_w(0)$.
\end{lem}

\begin{proof}
%
%

The first claim follows from \eqref{eq:taubarlip}.

If $F$ is isotropic, then for any $e, e'\in\mathbb{S}^{d-1}$, there is  $R\in SO(d)$ such that $e=Re'$. From \eqref{3.2} and isotropy we obtain 
\[
\overline{\tau}(e)=\lim_{n\rightarrow\infty} \frac{\tau(0, nRe', \om)}{n}=\lim_{n\rightarrow\infty} \frac{\tau(0, ne', \sigma_R\om)}{n}=\overline{\tau}(e')
\]
for almost all $\om\in\Om$.  Therefore $w(e)\equiv w$ is constant and $\calS=B_w(0)$.
\end{proof}

\section{From Spreading Speeds to Front Speeds and Homogenization for Convex Initial Sets}\label{s:fs}

It turns out that validity of Proposition \ref{P.3.3} for a fixed $\om$ is sufficient for our argument yielding existence of front speeds under relevant hypotheses.
Let us therefore consider the  PDE
\begin{equation} \label{4.0}
u_{t}= F(D^2 u,\nabla u,u,x)  \qquad\text{on  $(t,x)\in(0,\infty)\times\bbR^d$,}
\end{equation}
and its $\ve$-spacetime-scaled version
\begin{equation} \label{4.0e}
u^{\ve}_{t} = \ve^{-1} F\left(\ve^2 D^2 u,\ve\nabla u,u, \ve^{-1} x\right)  \qquad\text{on  $(t,x)\in(0,\infty)\times\bbR^d$.}
\end{equation}
We do not include $\om$ in \eqref{4.0}, which represents \eqref{1.10} for any fixed $\om\in\Om_0$, where $\Om_0$ is a full-measure set such that Proposition \ref{P.3.5} holds (with $\omega$-independent $\mathcal S$).
This is the starting point of this section, along with some other basic properties 
(cf. Hypothesis H in the introduction).


\medskip

{\bf Hypothesis H'.}  (i) Let \eqref{4.0} have a unique solution in some class of functions $\mathcal A\subseteq L^1_{\rm loc}((0,\infty)\times\bbR^d)$ for each locally BV initial condition $0\le u(0,\cdot)\le 1$, with this solution being constant 0 resp.~1 when $u(0,\cdot)\equiv 0$ resp.~$1$. (Below we only consider these solutions.)  Assume also that left time-shifts of solutions (restricted to $(0,\infty)\times\bbR^d$) are solutions from $\mathcal A$ and that \eqref{4.0} satisfies the (parabolic) comparison principle within $\mathcal A$.

(ii) Lemma \ref{L.4.2} holds for solutions to \eqref{4.0}, and there are $\theta_0<1$ and $R_0<\infty$ such that solutions $u(\cdot,\cdot;y)$ to \eqref{4.0} with $u(0,\cdot;y)=\theta_0\chi_{B_{R_0}(y)}$ for $y\in\bbR^d$ satisfy locally uniformly in $x\in\bbR^d$,
\[
\lim_{t\to\infty}\inf_{y\in\bbR^d} u(t,x+y;y)=1.
\] 

(iii) The PDE \eqref{4.0} has a {\it strong Wulff shape} $\calS$, satisfying \eqref{eq:wulffdef} with a continuous $w:\mathbb S^{d-1}\to(0,\infty)$, in the following strong sense:  for each $\Lambda,\delta>0$,  
the solutions from (ii) satisfy
\begin{align}
\lim_{t\rightarrow\infty} \inf_{|y|\leq \La t} \, \inf_{x\in (1-\delta)\mathcal{S}t+y}u(t,x; y) & =1 \label{4.14}
\\ \lim_{t\rightarrow\infty} \sup_{|y|\leq \La t} \, \sup_{x\notin (1+\delta)\mathcal{S}t+y}u(t,x; y) & =0. \label{4.15}
\end{align}
\medskip

{\it Remarks.}  
1.  Of course, as in Section \ref{sec:sspeed} above,
(ii) shows that nothing would change if in the case of \eqref{genhomeq} we instead considered solutions $\tilde u$ with initial data $\tilde u(0,\cdot;y) = v(\cdot-y)$ in (ii).
\smallskip

2. Results from the previous sections show that Hypothesis H' holds for \eqref{genhomeq} with any $\om\in\Om_0$, where $\Om_0$ is from Proposition \ref{P.3.5}, with $\theta_0$ being any number greater than $\theta_0$ from \eqref{h:stat0} (and with, for instance, $\mathcal A=\bigcap_{\tau>0} C^{1+\gamma,2+\gamma}((\tau,\infty)\times\bbR^d)$).  The analysis below therefore also applies to \eqref{genhomeq} with any $\om\in\Om_0$. (Nevertheless, below we will consider initial data with value $\theta_0+\alpha$ on some sets, with $\alpha>0$ so that our arguments also directly apply to \eqref{genhomeq} with $\theta_0$ from \eqref{h:stat0}.)  
\smallskip


3.  Similarly to H'(iii) above, we will consider here the analog of Proposition \ref{P.3.5} instead of just Definitions \ref{D.1.3} and \ref{D.1.5}. That is, we will say that $c^*(e)$ is a {\it strong front speed}  in direction $e\in\mathbb S^{d-1}$ for \eqref{4.0} if for each compact $K\subseteq  \{x\cdot e>0\}\subseteq\bbR^d$,
the solutions $u(\cdot,\cdot;y,e)$ to  \eqref{4.0} with initial data 
$u(0,\cdot;y,e)=\theta_0\chi_{\{x\cdot e< y\cdot e\}}$ satisfy for each $\La>0$,
\begin{align*}
\lim_{t\to \infty}  \inf_{|y|\le \La t}\, \inf_{x\in (c^*(e)e-K)t} u(t,x+y;y,e) &= 1,  
\\ \lim_{t\to \infty}  \sup_{|y|\le \La t}\, \sup_{x\in (c^*(e)e+K)t} u(t,x+y;y,e) &= 0.
\end{align*}
And if also for each compact $K\subseteq  \{x\cdot e>0\}$ there is $\beta_{K,e}:(0,1]\to(0,1]$ with $\lim_{\alpha\to 0}\beta_{K,e}(\alpha)=0$ such that the solutions $u(\cdot,\cdot;y,e,\alpha)$ with initial data $u(0,\cdot;y,e,\alpha)=\chi_{\{x\cdot e< y\cdot e\}} + \alpha \chi_{\{x\cdot e\ge y\cdot e\}}$ satisfy for each  $\La,\alpha>0$, 
 \[
 \limsup_{t\to \infty}  \sup_{|y|\le \La t}\, \sup_{x\in (c^*(e)e+K)t} u(t,x+y;y,e,\alpha) \le\beta_{K,e}(\alpha),
\]
then $c^*(e)$ will be a {\it strong exclusive front speed}  in direction $e$ for \eqref{4.0}.
\smallskip

We will next use these properties to obtain results about solutions with more general initial data, but first we will show that $\mathcal S$ is convex.  Below we will use the notation $B_r:=B_r(0)\subseteq\bbR^d$.

\begin{lem} \label{L.4.0}
If Hypothesis H' holds, then
$\mathcal{S}$ is convex.
\end{lem}

\begin{proof}
We need to show that
\begin{equation} \label{3.25}
w\left(\frac{sw(e)e+(1-s)w(e')e'}{|sw(e)e+(1-s)w(e')e'|}\right)\geq \left|sw(e)e+(1-s)w(e')e'\right|
\end{equation}
for any $e,e'\in\mathbb S^{d-1}$ and $s\in(0,1)$.  From \eqref{4.14} and $\mathcal{S}$ being open, we obtain for any $\delta>0$, 
\[
\lim_{t\to\infty} \inf_{|y|\le\delta^{-1}} u(st,(1-\delta)stw(e)e+y,\om;0) =1,
\]
as well as (with $y_t:=(1-\delta)stw(e) e$ and $\La$  large enough)
\[
\lim_{t\to\infty} u((1-s)t,(1-\delta)(1-s)tw(e')e'+y_t,\om;y_t)=1.
\]
This and the comparison principle imply
\[
\lim_{t\to\infty} u(t,(1-\delta)[(1-s)w(e')e'+sw(e)e]t,\om;0)=1,
\]
so that  \eqref{4.15} yields \eqref{3.25} after taking $\delta\to 0$.
\end{proof}

We next prove a ``lower bound'' on the region where $u^\ve\approx 1$ in the homogenization regime, allowing also for some dependence of initial data on $\ve>0$.  (The following results will be stated in terms of  \eqref{4.0e}.)

\begin{thm}  \label{T.3.6}
Assume Hypothesis H'.
Let  $A\subseteq\bbR^d$ be open and for $\ve>0$ and $y\in \bbR^d$ let $u^\ve(\cdot,\cdot;y)$ solve \eqref{4.0e}.  If   $\alpha,\lambda>0$ and
\[
u^\ve(0,\cdot;y)\ge (\theta_0+\alpha)\chi_{A+y},
\]
then 
\[
\lim_{\ve\to 0}  \inf_{|y|\le \lambda} u^\ve(t,x+y;y)=1 
\]
locally uniformly on $\Theta^{A,\mathcal{S}}:=\{(t,x)\in(0,\infty)\times\bbR^d \,|\, x\in A+t\mathcal{S} \}$.
\end{thm}

\noindent {\it Remarks.}
1.  The proof shows that the convergence is in fact uniform on $[t_0,\infty)\times Q$ for any $t_0>0$ and compact $Q\subseteq A+t_0\mathcal{S}$.
\smallskip

2.  If $u^\ve(0,\cdot;y)\ge \chi_{A+y}$, then the proof can easily be adapted to the case when $\{0\}\times A$ is added to $\Theta^{A,\mathcal{S}}$ (and $t_0\ge 0$ in Remark 1 also works).  This also applies to Theorems \ref{T.4.1}(i), \ref{T.4.4}(ii), and \ref{T.4.3} below.

\begin{proof}
Let $K\subseteq \Theta^{A,\mathcal{S}}$ be compact.  Since $\Theta^{A,\mathcal{S}}$ is open, there are $\Lambda',\de>0$ such that 
\[
K \subseteq\{(t,x)\in(0,\infty)\times\bbR^d \,|\, x\in (A\cap B_{\Lambda' t_K}) +(1-2\de)t \mathcal{S} \},
\]
where $t_K=\min_{(t,x)\in K} t>0$.
From Hypothesis H'(ii) and the comparison principle, 
applied to $u_\ve(t,x;y):=u^\ve(\ve t, \ve x;y)$ (which solves \eqref{4.0}), we know that there is ($\ve$-independent) $T_0>0$ such that for all small enough $\ve>0$  we have
\[
 \inf_{y\in\bbR^d} u_\ve(T_0,\cdot+\ve^{-1}y;y) \ge \theta_0 \chi_{\ve^{-1}A+B_{R_0}}.
\]
This, \eqref{4.14} with $\La:=\La'+\lambda$, and the comparison principle show that 
for each $\eta>0$ there is $\tau_\eta$ such that
\[
 \inf_{|y|\le \lambda} u_\ve(T_0+\ve^{-1}t,\cdot+\ve^{-1}y;y) \ge (1-\eta)\chi_{(\ve^{-1}A\cap B_{\Lambda'\ve^{-1}t}) +(1-\de)\ve^{-1}t \mathcal{S}}
\]
whenever $\ve^{-1}t\ge\tau_\eta$.  But this means that
\[
 \inf_{|y|\le \lambda} u_\ve(\ve^{-1}t,\cdot+\ve^{-1}y;y) \ge (1-\eta)\chi_{\ve^{-1}[(A\cap B_{\Lambda' t_K})+(1-2\de)t\mathcal{S}]}
\]
whenever $t\ge t_K$ and $\ve>0$ is small enough.  Since   $\eta>0$ was arbitrary, the result follows.
\end{proof}

Next we show how Wulff shapes with tangent hyperplanes (and thus with normal vectors) give rise to front speeds.  

\begin{thm}  \label{T.4.1}
Assume Hypothesis H',
 let $e\in \mathbb{S}^{d-1}$ and  $c^*(e)$ be from \eqref{eq:fsdef},
 and for $\ve>0$ and $y\in \bbR^d$ let  $u^\ve(\cdot,\cdot;y)$ solve \eqref{4.0e}.  

(i) If  $\alpha,\lambda>0$ and 
\[
u^\ve(0,\cdot;y)\ge (\theta_0+\alpha)\chi_{\{x\cdot e< y\cdot e\}}
\]
for all $\ve>0$ and $y\in \bbR^d$, then 
\[
\lim_{\ve\to 0}  \inf_{|y|\le \lambda} u^\ve(t,x+y;y)= 1
\]
locally uniformly on $\{(t,x)\in (0,\infty)\times\RR^d \,|\,  x\cdot e< c^*(e)t\}$.

(ii) If $e$ is a unit outer normal of  $\mathcal{S}$, $\alpha,\lambda>0$, and
\[
u^\ve(0,\cdot;y)\le (1-\alpha)\chi_{\{x\cdot e\le y\cdot e\}}
\]
for all $\ve>0$ and $y\in \bbR$, then 
\[
\lim_{\ve\to 0}  \sup_{|y|\le \lambda} u^\ve(t,x+y;y)= 0
\]
locally uniformly on $\{(t,x)\in (0,\infty)\times\RR^d \,|\,  x\cdot e> c^*(e)t\}$.  
 In particular, this and (i) imply that $c^*(e)$ is a strong front speed  in  direction $e$ for \eqref{4.0}
  in the sense of Remark 3 after Hypothesis H'.
\end{thm}

\noindent {\it Remark.}  (ii) and Lemma \ref{L.4.2} show that (ii) in fact holds locally uniformly on $\{(t,x)\in [0,\infty)\times\RR^d \,|\,  x\cdot e> c^*(e)t\}$ (the proof also shows this).  This also applies to Theorems  \ref{T.4.4}(iii,iv), and \ref{T.4.3} below.
\smallskip

This immediately yields Theorem \ref{T.1.4}(ii) as well as the corresponding part of Theorem \ref{T.1.8}.

\begin{proof}[Proof of Theorem \ref{T.1.4}(ii) and of the corresponding part of Theorem \ref{T.1.8}] 
Convexity of $\calS$ is established by Lemma \ref{L.4.0}. The second claim follows from the last claim in Theorem \ref{T.4.1}(ii), applied to \eqref{1.10} (or specifically \eqref{genhomeq}) with any fixed $\om\in\Om_0$, where $\Om_0$ is the full-measure  set from
Proposition \ref{P.3.3}.
\end{proof}

\begin{proof}[Proof of Theorem \ref{T.4.1}]
(i) This immediately follows from \eqref{eq:fsdef}
and Theorem \ref{T.3.6} with $A=\{x\cdot e< 0\}$,
because  then $A+t\mathcal S = \{x\cdot e< c^*(e)t\}$ for all $t>0$.  

(ii)
The second claim is immediate from the first and (i).  Indeed, if we take $u^\ve(0,\cdot;y)= \frac{1+\theta_0}2\chi_{\{x\cdot e< y\cdot e\}}$ and $(\alpha,\lambda):=(\frac{1-\theta_0}2,\La)$, then the functions $u^1(\cdot,\cdot;y)$, which solve \eqref{4.0}, satisfy $u^1(t,x;y)=u^{1/t}(1,\frac xt; \frac yt)$ and hence 
\begin{align*}
\lim_{t\to \infty}  \inf_{|y|\le \La t} \inf_{x\in (c^*(e)e-K)t} u^1(t,x+y;y) &= 1,  
\\ \lim_{t\to \infty}  \sup_{|y|\le \La t} \sup_{x\in (c^*(e)e+K)t} u^1(t,x+y;y) &= 0
\end{align*}
for any compact $K\subseteq  \{x\cdot e>0\}$.  In fact, running this argument in the opposite direction (and using $u^{\ve}(t,x; y)=u^1(\frac t\ve,\frac t\ve \frac xt; \frac t\ve \frac yt)$), together with an argument as in the proof of Theorem \ref{T.1.4}(i), also show that the claim in (i) and the first claim in (ii) (even without requiring $e$ to be a unit normal of $\mathcal S$) follow from the second claim in (ii).

Let us now prove the first claim. It suffices to consider $u^\ve(0,\cdot;y)= (1-\alpha)\chi_{\{x\cdot e\le y\cdot e\}}$, in which case the function $u_\ve(t,x;y):=u^\ve(\ve t, \ve x;y)$ is the solution of \eqref{4.0} satisfying $u_\ve(0,\cdot+\ve^{-1}y;y)= (1-\alpha)\chi_{\{x\cdot e\le 0\}}$.

Let $\Lambda>0$ be such that $\mathcal S\subseteq B_{\Lambda/4}$ and let $K\subseteq \{t\ge 0\,\,\&\,\,  x\cdot e> c^*(e)t\}$ be a compact set.  Then 
\[
K\subseteq \{t\ge 0\,\,\&\,\,  x\cdot e\ge  2\delta'+(c^*(e)+\delta')t\}\cap B_{1/\delta'}(0,0) 
\]
 for some $\delta'>0$ (with $B_r(t,x)\subseteq\bbR^{d+1}$).  Finally, let $e'\in \mathbb{S}^{d-1}$ be such that  $\mathcal S$ has unit outer normal $e$ at the point $w(e')e'\in\partial \mathcal S$ (then $w(e')e'\cdot e=c^*(e)$ by convexity of $\mathcal S$).

The proof is based on \eqref{4.14} and  \eqref{4.15} and  the observation that for all large $t_0$, the boundary $\partial ((t+t_0)\mathcal S-w(e')e't_0)$ is very close to the hyperplane $\{x\cdot e =  c^*(e)t\}$ at the point $w(e')e't$.  Hence the solution to \eqref{4.0} that starts from initial data $\theta_0 \chi_{B_{R_0}(-w(e')e't_0)}$ at time $t=-t_0$ will be close to the solution with initial data $\chi_{\{x\cdot e<0\}}$ at time $t=0$ on a space-time ball centered at $(0,0)$ and with radius $\ll t_0$.  By \eqref{4.14} and  \eqref{4.15}, on this ball the former solution (with compactly-supported initial data) looks like a front moving with speed $c^*(e)$ in direction $e$, so we will be able to conclude the same about the latter solution (with front-like initial data), which is essentially just a rescaling of $u^\ve$ (with $t_0\sim\ve^{-1}$).  To make this argument rigorous, we will need to choose the parameters involved very carefully, and we will also need to contend with the shifts $\ve^{-1}y$ at the same time.  (The reader may want to first consider the notationally simpler case $y=0$, when $u_\ve$ is also independent of $\ve$ and the function $u^{C,b,T}$ and number $t_{C,b,T}$ below do not depend on $b$.)

For any  $C,b,T>0$  with $C>b\lambda w(e')^{-1}$ let 
\[
u^{C,b,T}(t-t_{C,b,T}(y),x;y):=u(t,x;bT y-CTw(e')e'),
\]
 with $u(\cdot,\cdot;\cdot)$ from Hypothesis H'(ii)
 and $t_{C,b,T}(y)>0$ the smallest number such that
\begin{equation} \label{4.2}
u^{C,b,T}(0,\cdot+bTy;y)\ge (1-\alpha)\chi_{\{x\cdot e<0\}\cap B_T}.
\end{equation}
(We will eventually choose a small $b$ and a large $C$, and then take $T\to\infty$.)
It follows from \eqref{4.14} and  \eqref{4.15} (with the above $\Lambda$, and $\delta\to 0$; note that $bT\lambda+CTw(e')\le \frac {\Lambda} 2 CT$)  that
\[
k_C:= \lim_{T\to\infty} \frac{t_{C,b,T}(y)}{CT} 
\]
exists and (for each $y$) it is the smallest number such that 
\[
\{x\cdot e<c^*(e)\}\cap B_{1/C}(w(e')e')\subseteq k_C\mathcal S,
\] 
as well as
\[
\lim_{T\to\infty} \sup_{|y|\le \lambda} \left| \frac{t_{C,b,T}(y)}{CT}-k_C \right|=0.
\]
So $k_C\ge 1$ and we also have $k_C=1+o(C^{-1})$ (as $C\to\infty$) because $\mathcal S$ has a tangent hyperplane at $w(e')e'$ with outer normal $e$.  Moreover, 
\eqref{4.15} and 
$c^*(e)=w(e')e'\cdot e$
also show that for any fixed $(C,b)$ we have 
\[
\lim_{T\to\infty} \sup_{|y|\le \lambda} \sup_{s\ge 0} \sup_{x\cdot e\ge c^*(e)(k_C-1)CT +\delta'bT+(c^*(e)+\delta')bTs} u^{C,b,T}(bTs,x+bTy;y)=0
\]
because $\mathcal S \cap \{x\cdot e\ge c^*(e)\}=\emptyset$, $\delta'>0$, and (with $o(1)=o(T^0)$ uniform in $|y|\le \la$)
\[
c^*(e)(k_C-1)CT +\delta'bT= [c^*(e) +\delta'b(k_CC)^{-1}+o(1)]t_{C,b,T}(y) -CTw(e')e'\cdot e.
\]

If we choose $C$ large enough so that $c^*(e)(k_C-1)C\le \delta' b$ (which is possible for any $b>0$ because $k_C=1+o(C^{-1})$), we obtain
\[
\lim_{T\to\infty} \sup_{|y|\le \lambda} \sup_{s\ge 0} \sup_{x\cdot e\ge [2\delta'+(c^*(e)+\delta')s]bT} u^{C,b,T}(bTs,x+bTy;y)=0,
\]
and therefore 
\begin{equation} \label{4.3}
\lim_{T\to\infty} \sup_{|y|\le \lambda} \sup_{(t,x)\in bTK} u^{C,b,T}(t,x+bTy;y)=0.
\end{equation}
Lemma \ref{L.4.2} applied to $u_{(bT)^{-1}}$ and $u^{C,b,T}$ (with the couple $(y,r)$ being $(x+bTy,T-|x|)$ when $|x|< T$), together with \eqref{4.2}, now yields
\[
u_{(bT)^{-1}}(t,x+bTy;y)\le u^{C,b,T}(t,x+bTy;y) + c'e^{-m'(T-|x|-c't)}
\]
for each $(t,x)\in[0,\infty)\times\RR^d$.  This, \eqref{4.3}, and $K\subseteq B_{1/\delta'}(0,0)$ then yield
\[
0\le \lim_{T\to\infty} \sup_{|y|\le \lambda} \sup_{(t,x)\in bTK} u_{(bT)^{-1}}(t,x+bTy;y)\le \lim_{T\to\infty}  c'e^{-m'(T-(1+c')bT/\delta')}=0,
\]
so long as we choose any $b<\delta'(1+c')^{-1}$ (and then $C$ accordingly).  But then
\[
\lim_{\ve\to 0} \sup_{|y|\le \lambda} \sup_{(t,x)\in K} u^\ve(t,x+y;y) = \lim_{\ve\to 0}  \sup_{|y|\le \lambda} \sup_{(t,x)\in \ve^{-1} K} u_\ve( t,x+\ve^{-1}y;y)=0.
\]
\end{proof}

In fact, the above convergences are uniform in all directions $e$ for which \eqref{4.0} has a front speed. This is the content of the following result, which can then be used to prove Theorem \ref{T.1.9} and Theorem \ref{T.1.4}(iv).

\begin{thm}  \label{T.4.4}
Assume Hypothesis H'(i,ii),
 let $\mathcal T$ be the set of all directions $e\in\mathbb S^{d-1}$ for which \eqref{4.0}  has a strong front speed $c^*(e)$ in the sense of Remark 3 after Hypothesis H', and let $\tilde{\mathcal T}\subseteq\mathcal T$ be the set of all  $e\in \mathbb{S}^{d-1}$ for which \eqref{4.0} has a strong exclusive front speed. 
For each $e\in\mathbb S^{d-1}$, let
$R_{e}$ be any rotation on $\mathbb R^d$ with $R_{e} e_1=e$,   and for $\ve>0$ and $y\in \bbR^d$ let  $u^\ve(\cdot,\cdot;y,e)$ solve \eqref{4.0e}.  

(i)  $\mathcal T$ and $\tilde{\mathcal T}$ are closed and $c^*|_{\mathcal T}$ is positive, bounded by $c'$ from Lemma~\ref{L.4.2}, and Lipschitz continuous with Lipschitz constant only depending on $c'$.

(ii) If  $\alpha,\lambda>0$ and
\[
u^\ve(0,\cdot;y,e)\ge (\theta_0+\alpha)\chi_{\{x\cdot e< y\cdot e\}}
\]
for all $\ve>0$ and $(y,e)\in \bbR^d\times\mathcal T$, then 
\[
\lim_{\ve\to 0}  \inf_{|y|\le \lambda\,\&\, e\in\mathcal T} u^\ve(t, R_{e} x+c^*(e)t e +y;y,e)= 1
\]
locally uniformly on $\{(t,x)\in (0,\infty)\times\RR^d \,|\,  x_1<0\}$.

(iii) If  $\alpha,\lambda>0$ and
\[
u^\ve(0,\cdot;y,e)\le (1-\alpha)\chi_{\{x\cdot e\le y\cdot e\}}
\]
for all $\ve>0$ and $(y,e)\in \bbR^d\times\mathcal T$, then
\[
\lim_{\ve\to 0}  \sup_{|y|\le \lambda \,\&\, e\in\mathcal T} u^\ve(t, R_{e} x+c^*(e)t e +y;y,e)= 0
\]
locally uniformly on $\{(t,x)\in (0,\infty)\times\RR^d \,|\,  x_1>0\}$.

(iv) 
For each compact set $K\subseteq  \{(t,x)\in (0,\infty)\times\RR^d \,|\,  x_1>0\}$ there is $\beta_{K}:(0,1]\to(0,1]$ with $\lim_{\alpha\to 0}\beta_{K}(\alpha)=0$ such that if   $\alpha,\la>0$ and
 \begin{align*} 
u^{\ve}(0,\cdot;y,e,\alpha) & \le  \chi_{\{x\cdot e\le y\cdot e\}} + \alpha \chi_{\{x\cdot e> y\cdot e\}}
\end{align*}
for all $\ve>0$ and $(y,e)\in \bbR^d\times \tilde{\mathcal T}$, then 
\begin{align*} 
\limsup_{\ve\to 0} \sup_{|y|\le \lambda \,\&\, e\in\tilde{\mathcal T}} \sup_{(t,x)\in K} u^\ve(t, R_{e} x+c^*(e)t e +y;y,e,\alpha) \le \beta_{K}(\alpha).
\end{align*}
\end{thm}

\begin{proof}[Proof of Theorem \ref{T.1.9}]
(i) Let $\mathcal{T}'$ be a dense countable subset of $\mathcal T$.  There is obviously $\Om_0\subseteq \Om$ with $\mathbb P[\Om_0]=1$ such that  \eqref{1.10} with any fixed $\om\in\Om_0$ has the same front speed $c^*(e)$ in direction $e$ for each $e\in\mathcal{T}'$.  Proposition \ref{P.3.5} shows that these are strong front speeds, and applying Theorem \ref{T.4.4}(i) now yields that \eqref{1.10} with any fixed $\om\in\Om_0$ has  the same (strong) front speed $c^*(e)$ in direction $e$ for each $e\in\mathcal T$, as well as the claimed bounds on $c^*$.  Finally, Theorem \ref{T.4.4}(ii,iii) yield \eqref{1.97} and \eqref{1.98} because $u_{\mscr{T}_y\om,e}(t,x) = u(t,x+y,\om;y,e)$.

The same argument applies to $\tilde{\mathcal T}$.  Moreover, the proof of Theorem~\ref{T.4.4}(iv) below shows that the functions $\beta_K$ there for compacts $K\subseteq \{(t,x)\in (0,\infty)\times\RR^d \,|\,  x_1>0\}$  are determined from the functions $\beta_{R_eK',e}$ from Proposition \ref{P.3.5}(ii) (i.e., the same ones as in Definition \ref{D.1.5}), where $e$ are all vectors from any dense countable subset $\mathcal{\tilde T}'$ of $\mathcal {\tilde T}$, and $K'$ are all compacts contained in $\{x_1>0\}\subseteq\RR^d$.  Since we can take the same $\mathcal{\tilde T}'$ for all the $\om\in\Om_0$ (for the here-relevant full-measure set $\Om_0$), it follows that the same $\beta_K$ is shared by all $\om\in\Om_0$.  If we now take $K:=\{1\}\times K'$ for any given compact $K' \subseteq \{x_1>0\}$, then Theorem~\ref{T.4.4}(iv) and $\frac 1\ve$-scaling in $(t,x)$ show that for each $\om\in\Om_0$ and $e\in\mathcal{\tilde T'}$, we can take $\beta_{R_eK',e}$ in the definition of strong exclusive front speeds for \eqref{1.10} with this fixed $\om$ to be precisely this $\beta_K$.  Hence we get the same $\beta_{R_eK',e}$ for all $\om\in\Om_0$ and $e\in \mathcal{\tilde T'}$ (as the definition of deterministic exclusive front speeds requires), and it also equals $\beta_{K',e_1}$.  Finally, Theorem \ref{T.4.4}(iv) also yields \eqref{1.99}.

(ii)  Since $\mathcal S$ is convex due to Lemma \ref{L.4.0}, for each $e\in\mathbb S^{d-1}$ there is a sequence $e_n\in\mathbb S^{d-1}$ converging to $e$ such that $\mathcal S$ has some unit outer normal $e_n'$ at $w(e_n)e_n$. 
Then Theorem \ref{T.4.1}(ii) shows that $e_n'\in\mathcal T$ and $c^*(e_n')=\max_{e'} w(e')e'\cdot e_n'=w(e_n)e_n\cdot e_n'$, the latter because $\mathcal S$ is convex. But then
\[
w(e) = \lim_{n\to\infty} w(e_n) =  \lim_{n\to\infty} \frac {c^*(e_n')}{e_n\cdot e_n'},
\]  
which yields one inequality in \eqref{1.40}.  The other is immediate from the comparison principle and the definitions of the (deterministic) Wulff shape and front speeds.
\end{proof}

%

\begin{proof}[Proof of Theorem \ref{T.4.4}]
(i--iii) First note that considering the solution from Definition \ref{D.1.3} and $u'\equiv 0$ in Lemma \ref{L.4.2} immediately yields $c^*(e)\le c'$.  We also have $c^*(e)> 0$ by iterating the second assumption in Hypothesis H'(ii) (and using the comparison principle).

Next, 
in (ii) and (iii) we only need to consider the convergence on sets $K_M^-:=[\frac 1M, M]\times(\overline {B_M}\cap\{ x_1\le - \frac 1M\})$ and $K_M^+:=[\frac 1M, M]\times( \overline {B_M}\cap\{ x_1\ge \frac 1M\})$  for all large $M$.
The definition of strong front speeds and $\ve$-scaling show that (ii,iii) hold for each fixed $e\in\mathcal T$ (see the first paragraph  of the proof of Theorem~\ref{T.4.1}(ii)), and hence also when the inf and sup are over all $e\in\mathcal T'$, with $\mathcal T'$ any finite subset of $\mathcal T$.  The full result (including the rest of (i)) will now be obtained once we prove an appropriate bound on the difference of solutions $u^\ve$ above with $e$ and $e'$ such that $|e-e'|$ is small.

This will be achieved using Lemma \ref{L.4.2} and the comparison principle.  Fix $\alpha,\lambda$, let $M\ge 1$ be arbitrary and let us consider $K_M^+$.  Let $\mathcal T'\subseteq \mathcal T$ be any finite set such that $\overline{\mathcal T}\subseteq B_{(5(3c'+1)M^2)^{-1}}(\mathcal T')$, and for any $e\in\overline{\mathcal T}$, let $e'\in \mathcal T'$ be arbitrary such that
\begin{equation} \label{4.45}
 |e-e'| \le (5(3c'+1)M^2)^{-1}.
\end{equation}

Let now $u,u'$ solve \eqref{4.0e} for some $\ve>0$, with initial data satisfying
\begin{align*}
u(0,\cdot) &\le (1-\alpha)\chi_{\{x\cdot e\le y\cdot e\}},
\\ u'(0,\cdot) &= (1-\alpha)\chi_{\{x\cdot e'\le y\cdot e'+(5M)^{-1}\}}.
\end{align*}
From \eqref{4.45} it follows that 
\[
\{x\cdot e\le y\cdot e \}\cap B_{(3c'+1)M}(y)\subseteq \{x\cdot e'\le y\cdot e'+(5M)^{-1}\}.
\]
Combining this with $c^*(e')\le c'$, we have that for any $t\in[0,M]$ and $z\in B_M(y+c^*(e')te)$, $u(0,\cdot)\le u'(0,\cdot)$ on $B_{2c'M}(z)\subseteq B_{(3c'+1)M}(y)$. 
Applying Lemma \ref{L.4.2} (after an $\ve$-scaling in space and time) yields
\[
u(t,z)\le u'(t,z) + c'e^{-m'\ve^{-1}(2c'M-c't)} \le  u'(t,z) + c'e^{-m'c'M\ve^{-1}}.
\]
So if $R,R'$ are rotations in $\bbR^d$ such that $Re_1=e$ and $R'e_1=e'$, then
\[
\sup_{(t,x)\in K_M^+} u(t,R x+y+ c^*(e')t e) \le  \sup_{(t,x)\in K_M^+} u'(t,R x+y+ c^*(e')t e) + c'e^{-m'c'M\ve^{-1}}.
\]
Now $t\le M$, \eqref{4.45}, and $c^*(e')\le c'$ yield 
\[
|c^*(e')te-c^*(e')te'| \le \frac 1{5M},
\]
and then $|Rx-R'x|\le \frac{1}{5(3c'+1)M}$ for any $x\in \overline{B_M}$ gives
\begin{multline*}
R \left( \overline{B_M}\cap \left\{ x_1\ge \frac 1M \right\} \right) +y+ c^*(e')t e \\
\subseteq R' \left( \overline{B_{M+1}}\cap \left\{ x_1\ge \frac 1{5M} \right\} \right) + y + (5M)^{-1}e'+c^*(e')te'.
\end{multline*}
This means that 
\begin{multline*}
\sup_{(t,x)\in K_M^+} u(t,R x+y+ c^*(e')t e) \\\le  \sup_{(t,x)\in K_{5M}^+} u'(t,R' x+y+ (5M)^{-1}e'+ c^*(e')t e') + c'e^{-m'c'M\ve^{-1}}.
\end{multline*}

Therefore 
\begin{align*}
 &  \sup_{|y|\le \lambda \,\&\, e\in\overline{\mathcal T} \,\&\, (t,x)\in K_M^+} \sup_{e'\in\mathcal T'\,\&\,  |e-e'| \le (5(3c'+1)M^2)^{-1} } u^\ve(t, R_{e} x+y+c^*(e')t e;y,e)
\\ \le &  \sup_{|y|\le \lambda \,\&\, e'\in\mathcal T' \,\&\, (t,x)\in K_{5M}^+} u^\ve(t,R_{e'} x+(y+ (5M)^{-1}e')+ c^*(e')t e';y+(5M)^{-1}e',e')\\
& + c'e^{-m'c'M\ve^{-1}}
\end{align*}
if we assume $u^\ve(0,\cdot;y,e)= (1-\alpha)\chi_{\{x\cdot e\le y\cdot e\}}$ (which suffices in (iii)).  Since the right-hand side converges to 0 as $\ve\to 0$ 
(see the start of this proof), so does the left-hand side.
If we now fix any $e,e'\in\mathcal T$ such that $M:=(5(3c'+1)|e-e'|)^{-1/2}\ge 1$ and pick $\mathcal T'$ as above containing $e'$, then this for $(t,x,y,\lambda)=(M,\frac 1Me_1,0,1)$ yields
\[
\lim_{\ve\to 0} u^\ve (M,(M^{-1} +c^*(e')M)e;0,e)=0.
\]
It then follows from the definition of front speed in direction $e$  that 
\[
c^*(e)\le c^*(e')+M^{-2} \le c^*(e')+ 5(3c'+1)|e-e'|
\]
 when $e,e'\in\mathcal T$ and $|e-e'| \le \frac 1{5(3c'+1)}$.

A similar argument, using
initial data satisfying
\begin{align*}
u(0,\cdot) &\ge (\theta_0+\alpha)\chi_{\{x\cdot e< y\cdot e\}},
\\ u'(0,\cdot) &= (\theta_0+\alpha)\chi_{\{x\cdot e'< y\cdot e'-(5M)^{-1}\}},
\end{align*}
eventually gives
\begin{align*}
 &  \inf_{|y|\le \lambda \,\&\, e\in\overline{\mathcal T} \,\&\, (t,x)\in K_M^-} \inf_{e'\in\mathcal T'\,\&\,  |e-e'| \le (5(3c'+1)M^2)^{-1}} u^\ve(t, R_{e} x+y+c^*(e') t e;l,e)
\\ \ge &  \inf_{|y|\le \lambda \,\&\, e'\in\mathcal T' \,\&\, (t,x)\in K_{5M}^-} u^\ve(t,R_{e'} x+(y- (5M)^{-1}e')+ c^*(e')t e';y-(5M)^{-1}e',e')\\
&- c'e^{-m'c'M\ve^{-1}}
\end{align*}
if we assume $u^\ve(0,\cdot;y,e)= (\theta_0+\alpha)\chi_{\{x\cdot e\le y\cdot e\}}$ (which suffices in (ii)).  Since the right-hand side  converges to 1 as $\ve\to 0$, so does the left-hand side, and from this we also get $c^*(e)\ge c^*(e')- 5(3c'+1)|e-e'|$ when $e,e'\in\mathcal T$ and $|e-e'| \le \frac 1{5(3c'+1)}$.

These last two paragraphs now show that 
$c^*|_{\mathcal T}$ is Lipschitz, with $|c^*(e)-c^*(e')|\le 5(3c'+1)|e-e'|$ when $e,e'\in\mathcal T$ and $|e-e'| \le \frac 1{5(3c'+1)}$.  We can therefore continuously extend $c^*$ to $\overline{\mathcal T}$.  Then for any $M'\ge 1$, the $\ve\to 0$ limits of the displayed inequalities in those two paragraphs with  $M:=2M'$
and  $\mathcal T'\subseteq\mathcal T$ such that $\overline{\mathcal T}\subseteq B_{(5(3c'+1)M^2)^{-1}}(\mathcal T')$ yield
\begin{align*}
 \lim_{\ve\to 0} \sup_{|y|\le \lambda \,\&\, e\in\overline{\mathcal T} \,\&\, (t,x)\in K_{M'}^+}  u^\ve(t, R_{e} x+y+c^*(e)t e;y,e) & = 0,
\\   \lim_{\ve\to 0} \inf_{|y|\le \lambda \,\&\, e\in\overline{\mathcal T} \,\&\, (t,x)\in K_{M'}^-}  u^\ve(t, R_{e} x+y+c^*(e)t e;y,e) & = 1
\end{align*}
because $|c^*(e)te-c^*(e')te|\le \frac 1M\le \frac 1{M'}-\frac 1M$ when $|e-e'|\le (5(3c'+1)M^2)^{-1}$ and $|t|\le M$.  Since $M'$ was arbitrary, it now follows after $\ve$-scaling that $c^*(e)$ is the front speed in direction $e$ for \eqref{4.0} whenever $e\in \overline{\mathcal T}$.  So
$\mathcal T$ is closed, and the last two limits also prove (ii,iii).

(iv)  This (including the proof that $\tilde{\mathcal T}$ is closed) is analogous to the above argument, 
but with 
 initial data satisfying
\begin{align*}
u(0,\cdot) &\le  \chi_{\{x\cdot e\le y\cdot e\}} + \alpha \chi_{\{x\cdot e> y\cdot e\}},
\\ u'(0,\cdot) &=  \chi_{\{x\cdot e'\le y\cdot e'+(5M)^{-1}\}} + \alpha \chi_{\{x\cdot e'> y\cdot e'+(5M)^{-1}\}},
\end{align*}
and also using that $c^*|_{\mathcal T}$ is continuous.  
(We note that the formula $u^{\ve}(t,x; y)=u^1(\frac t\ve,\frac t\ve \frac xt; \frac t\ve \frac yt)$ from the start of the proof of Theorem~\ref{T.4.1}(ii) shows that for any $M'\ge 1$ and $M:=2M'$ we can pick 
$\beta_{K_{M'}^+}$  to be the maximum of $\beta_{R_{e'}(\overline{B_{(5M)^2}}\cap\{ x_1\ge (5M)^{-2}\}),e'}$ 
from Remark 3 after Hypothesis H' (i.e., definition of strong exclusive front speeds) over all directions $e'\in\mathcal T'$, with any finite $\mathcal T'\subseteq\tilde{\mathcal T}$ such that  $\tilde{\mathcal T}\subseteq B_{(5(3c'+1)M^2)^{-1}}(\mathcal T')$.)
\end{proof}

\begin{proof}[Proof of Theorem \ref{T.1.4}(iv) and of the corresponding part of Theorem \ref{T.1.8}] 
Our proof of the first claim is similar to that in \cite{wein} for periodic ignition reactions (when pulsating fronts exist in all directions), although the uniformity with respect to certain translations in the remark after Theorem \ref{T.1.9} simplifies it  (uniformity with respect to directions is not necessary here).  Define $\mathcal S$ via \eqref{eq:fg} and \eqref{eq:wulffdef} and observe that then
\[
\mathcal S = \bigcap_{e\in\mathbb S^{d-1}} \{x\in\RR^d\,:\,x\cdot e < c^*(e)\}.
\]
This, Definition \ref{D.1.5}, and the comparison principle immediately yield \eqref{eq:ubsxi'} for almost all $\om\in\Om$ because $B_{2c'}(0)\cap ( \RR^d\setminus(1+\delta)\mathcal S)$ is contained in some finite union of some compacts $K_e\subseteq\{x\in\RR^d\,:\,x\cdot e > c^*(e)\}$ whenever $\delta>0$.  (Note that if we replace $(1+\delta)\mathcal S$ by $B_{2c'}(0)$ in \eqref{eq:ubsxi'}, then the convergence is obvious from Lemma \ref{L.4.2}.)

To obtain \eqref{eq:lbsxi'}, consider any strictly convex compact $W\subseteq\mathcal S$ with a smooth boundary.  If $W$ is such a set, then for each $e\in\mathbb S^{d-1}$ there is a unique point $x_e\in\partial W$ at which $\partial W$ has unit outer normal $e$ (and this map is a bijection).  Also, continuity and positivity of $c^*$ show that there is $\delta>0$ such that
\[
W \subseteq  \bigcap_{e\in\mathbb S^{d-1}} \{x\in\RR^d\,:\,x\cdot e < (1-\delta)c^*(e) \},
\]
which in particular means $x_e\cdot e\le (1-\delta)c^*(e)$ for all $e$.
Smoothness of $\partial W$ shows that there is $\delta'\in(0,\delta)$ such that for each $e\in\mathbb S^{d-1}$,
\[
\{x\in\RR^d\,\mid\,x\cdot e \le x_e\cdot e -c^*(e)\delta' \text{ and } |x-x_e|\le 4c'\delta^{-1}\delta'\} \subseteq W,
\]
where $c'$ is from Lemma \ref{L.4.2}.

Let us now fix any $\om$ from the full-measure set $\Om_0$ from the Remark after Theorem \ref{T.1.9}.  Pick $\Lambda$ so that $\frac\delta{2\delta'}W\subseteq B_\La(0)$, and also any $e\in\mathbb S^{d-1}$ and any $t_0\gg 1$.  We can now use Lemma \ref{L.4.2} with $t=2\delta^{-1}\delta't_0$, any $y\in B_{2c'\delta^{-1}\delta't_0}(t_0x_e)$, and functions $u_{t_0,e}$ and $u_{t_0}$ in place of $u$ and $u'$, where  $u_{t_0,e}(0,\cdot)=\frac{1+\theta_0}2 \chi_{\{x\cdot e< (x_e\cdot e -c^*(e)\delta')t_0 \}}$ and $u_{t_0}(0,\cdot)=\frac{1+\theta_0}2 \chi_{t_0W}$, to find
\[
\inf_{x\in B_{2c'\delta^{-1}\delta't_0}(t_0x_e)}  [u_{t_0}(2\delta^{-1}\delta't_0,x) - u_{t_0,e}(2\delta^{-1}\delta't_0,x) ] \ge -c' e^{-2m'\delta^{-1}\delta't_0}.
\]
From the remark after Theorem \ref{T.1.9} we now obtain
\[
\lim_{t_0\to \infty}  \inf_{e\in\mathcal T} \inf_{x\in B_{2c'\delta^{-1}\delta't_0}(t_0x_e)\cap\{x\cdot e< (x_e\cdot e + 2c^*(e)\delta'(\delta^{-1}-1))t_0\}} u_{t_0,e}(2\delta^{-1}\delta't_0,x) = 1.
\]
(In fact, only taking finitely many directions $e$ would suffice here.)
From $x_e\cdot e\le (1-\delta)c^*(e)$ we then obtain
\[
\lim_{t_0\to \infty}  \inf_{e\in\mathcal T} \inf_{x\in B_{2c'\delta^{-1}\delta't_0}(t_0x_e)\cap\{x\cdot e< x_e\cdot e(1+ 2\delta^{-1}\delta')t_0\}} u_{t_0}(2\delta^{-1}\delta't_0,x) = 1.
\]
But since
\[
(1+ 2\delta^{-1}\delta')t_0 W \subseteq t_0W \cup \bigcup_{e\in\mathcal T} [ B_{2c'\delta^{-1}\delta't_0}(t_0x_e)\cap\{x\cdot e< x_e\cdot e(1+ 2\delta^{-1}\delta')t_0\} ],
\]
this and the first claim in Lemma \ref{L.3.11} yield 
\[
\lim_{t_0\to \infty} \inf_{x\in (1+ 2\delta^{-1}\delta')t_0 W} u_{t_0}(2\delta^{-1}\delta't_0,x) = 1.
\]
This and the comparison principle show that if $t_0$ is large enough and a solution $u$ to \eqref{1.10} satisfies
\[
\inf_{(t,x)\in[t_1,t_1+2t_0]\times t_0W} u(t,x)\ge \frac {1+\theta_0}2
\]
for some $t_1$ (which $u_\om$ from Definition \ref{D.1.3} does by Lemma \ref{L.3.11}), then
\[
\lim_{t\to\infty} \inf_{x\in (t-t_1-t_0)W} u(t,x)=1.
\] 
But this implies  \eqref{eq:lbsxi'} for any $\delta>0$ such that $(1-\frac\delta 2)\mathcal S\subseteq W$.  Since $W\subseteq \mathcal S$ was an arbitrary compact, the proof of the first claim is finished.

To prove the second claim, assume now that also \eqref{eq:fsdef} holds for each $e\in \mathbb{S}^{d-1}$.
The hypotheses show that the Hamiltonian 
\begin{equation*}
H(p)=-c^{*}\left(-\frac{p}{|p|}\right)|p|
\end{equation*}
in \eqref{eq:hjeq}  satisfies
\begin{equation*}
\tilde{H}(p):=-H(-p)=c^{*}\left(\frac{p}{|p|}\right)|p|=\sup_{e'\in \mathbb{S}^{d-1}} w(e')e'\cdot p.
\end{equation*}
So if $\overline v$ and $\overline{v}_{0}$ are as in \eqref{1.40'}, then $\tilde v:=-\overline v$ solves 
\begin{align*}
\tilde v_{t}+\tilde{H}(\nabla \tilde v)=0&\qquad \text{on } (0, \infty)\times \RR^{d},\\
\tilde v(0,x)=-\overline{v}_{0}(x)&\qquad \text{on } \RR^{d}
\end{align*}
with a convex Hamiltonian $\tilde H$.
Therefore the function
\begin{equation}\label{eq:lagrangian}
L(p):=(\tilde{H})^{*}(p):=\sup_{q\in \RR^{d}}\left[p\cdot q-c^{*}\left(\frac{q}{|q|}\right)|q|\right]=\begin{cases} 0&\text{if}\quad |p|\leq w\left(\frac{p}{|p|}\right),\\
\infty&\text{if}\quad |p|> w\left(\frac{p}{|p|}\right)
\end{cases}
\end{equation}
is the associated Lagrangian (with the last equality due to convexity of $\mathcal S$, by Lemma \ref{L.4.0}) and $\tilde{H}(p)=L^{*}(p)$. 
The Hopf-Lax formula now yields the identity
 \begin{equation}\label{5.7}
\tilde v(t,x)=\inf_{y\in \RR^{d}} \left[tL\left(\frac{x-y}{t}\right)-\overline{v}_{0}(y)\right]
 \end{equation}
 for the (unique) viscosity solution $\tilde v$.
Since $L\left(\frac{x-y}{t}\right)=0$ precisely when $y\in x-t\overline{\mathcal{S}}$, we obtain
 \begin{equation}\label{5.8}
  \left\{x \,\mid\, \overline{v}(t,x)>0\right\}=A+ t\overline{\mathcal{S}} =A+t\mathcal{S}
 \end{equation}
(because $A$ is open) for any $t>0$.   That is, $\Theta^{A, c^*} = \Theta^{A, \mathcal{S}}$ and the proof is finished. Note also that we did not assume convexity of $A$ here. \end{proof}
%

%

We will now prove our first homogenization result,  Theorem \ref{T.4.3}.   When also \eqref{eq:fsdef} holds, we immediately obtain from it the claim in Theorem \ref{T.1.4}(iii) (and the corresponding part of Theorem \ref{T.1.8})  because 
$\overline{\Theta^{A, \mathcal{S}}} \setminus \Theta^{A, \mathcal{S}}=\partial \Theta^{A, \mathcal{S}}$ for convex $A$ and $\Theta^{A, c^*} = \Theta^{A, \mathcal{S}}$ due to \eqref{eq:fsdef}.

We will prove  Theorem \ref{T.1.4}(iii) without assuming \eqref{eq:fsdef} in Section~\ref{s:fstohom}.  In fact, Theorem \ref{T.4.3} will not be needed in the rest of the paper and can be skipped.  We include it only to demonstrate how one can prove homogenization in some settings without the need for the theory of discontinuous viscosity solutions in Section \ref{s:fstohom}.

Recall that for  $A\subseteq\bbR^d$ and $r\ge 0$, we define $A_r^0:=A\setminus \overline{B_r(\partial A)}$.

\begin{thm}  \label{T.4.3}
Assume Hypothesis H'
and that \eqref{4.0} has a strong front speed $c^{*}(e)$ satisfying \eqref{eq:fsdef} in each direction $e\in \mathbb{S}^{d-1}$ (e.g., if $\mathcal{S}$ has no corners, due to Theorem \ref{T.4.1}).
 If  $A\subseteq\RR^d$ is open, $A'={\rm ch}(A)$ is its convex hull, $\alpha>0$, and $u^\ve$ solves \eqref{4.0e} and  
\begin{equation} \label{4.4}
(\theta_0+\alpha)\chi_{A_{\psi(\ve)}^0+y_\ve} \le u^\ve(0,\cdot)\le (1-\alpha)\chi_{B_{\psi(\ve)}(A')+y_\ve}
\end{equation}
for each  $\ve>0$, with some $y_\ve\in B_{1/\alpha}$ and $\lim_{\ve\to 0}\psi(\ve)=0$, then 
\[
\lim_{\ve\to 0}  u^\ve(t,x+y_\ve)= 
\begin{cases}
1  & (t,x)\in \Theta^{A, \mathcal{S}},\\
0  & (t,x)\in ([0,\infty)\times \RR^d) \setminus \overline{\Theta^{A', \mathcal{S}}}
\end{cases}
\]
locally uniformly on $ ([0,\infty)\times\RR^d) \setminus (\overline{\Theta^{A', \mathcal{S}}} \setminus \Theta^{A, \mathcal{S}})$.
\end{thm}

To prove Theorem \ref{T.4.3}, we  will need the following geometric lemma.

\begin{lem} \label{L.4.5}
Assume that $\mathcal S\subseteq\RR^d$ containing 0 is open, bounded, and convex.  For each $e\in \mathbb{S}^{d-1}$, let $w(e)$ and $c^*(e)$ be given by \eqref{eq:wulffdef} and \eqref{eq:fsdef}.
If    $A\subseteq\RR^d$ is open and convex,
then for each $t\ge 0$ we have
\begin{align*}
A+t{\mathcal S} & =\bigcap_{e\in \mathbb{S}^{d-1}}  \{ x\cdot e < \sup_{y\in A} y\cdot e+c^*(e)t \},\\
\overline{A+t{\mathcal S}} & =\bigcap_{e\in \mathbb{S}^{d-1}}  \{ x\cdot e \le \sup_{y\in A} y\cdot e+c^*(e)t \}.
\end{align*}
\end{lem}

\begin{proof}
Let $B$ be the above intersection.
If $x\in A+t{\mathcal S}$, then there are $y\in A$ and $z\in {\mathcal S}$ such that $x=y+zt$.  If $z=0$ then obviously $x\in B$; otherwise let $e'=\frac z{|z|}$.  
Then for any $e\in \mathbb{S}^{d-1}$ 
we have
\[
x\cdot e = y\cdot e + |z|te'\cdot e < \sup_{y'\in A} y'\cdot e + c^*(e)t
\]
by $|z|<w(e')$, \eqref{eq:fsdef}, and $c^*(e)>0$. Hence $A+t{\mathcal S}\subseteq B$.

If now $x\notin A+t{\mathcal S}$, let $x'\in A+t{\mathcal S}$ be the closest point from (the convex set) $\overline{A+t{\mathcal S}}$ to $x$.  Let $e\in \mathbb{S}^{d-1}$ be such that $x''\cdot e\le x'\cdot e  \le x\cdot e$ for any $x''\in \overline{A+t{\mathcal S}}$ (if $x\neq x'$, then $e=\frac{x-x'}{|x-x'|}$ works). Let $e'\in \mathbb{S}^{d-1}$ be such that $c^*(e)=w(e')e'\cdot e$.  Then for any $y\in A$ we have $y+w(e')e't\in \overline{A+t{\mathcal S}}$, so 
\[
y\cdot e + c^*(e)t  = (y+ w(e')e't)\cdot e\le  x\cdot e.
\]
Therefore $x\notin B$, so $A+t{\mathcal S}\supseteq B$.  
This proves the first claim, and the proof of the second is analogous (noticing also that $\overline{A+{t\mathcal S}}=\overline A + t\overline{\mathcal S}$).
\end{proof}

\begin{proof}[Proof of Theorem \ref{T.4.3}]
Obviously, these are two separate results, with the first inequality in \eqref{4.4} yielding the convergence to 1 and the second inequality yielding convergence to 0. 
Hence, due to both convergences being locally uniform and due to the inclusion of the shifts $y_\ve$, it suffices to consider in both proofs $\psi(\ve)=0$ for all $\ve>0$.

The  convergence to 1 then follows directly from Theorem \ref{T.3.6}.  For the convergence to 0, recall that $A'=\text{ch}(A)$, and let $K\subseteq ([0,\infty)\times\RR^d) \setminus \overline{ \Theta^{A', \mathcal{S}}}$ be a compact set.  We can assume without loss that 
$K=[0,t_0]\times Q$ for some $t_0>0$ and some  compact $Q\subseteq \RR^d\setminus\overline {A'+t_0 \mathcal S }$ because $K$ is contained in a finite union of such sets (with distinct $t_0$).  

With $c'$ from Lemma \ref{L.4.2}, let $\Lambda>0$ be such that $Q\subseteq B_{\Lambda-2c't_0}$ and let $B:=A'\cap B_\Lambda$ (which is open, bounded, and convex).  Then $Q\,\cap\, \overline {B+t_0 \mathcal S }=\emptyset$, so by Lemma \ref{L.4.5}, there is $\delta'>0$ such that 
\[
Q \cap \bigcap_{e\in \mathbb{S}^{d-1}}  \{ x\cdot e \le  l_B(e)+c^*(e)t_0+\delta' \}
 = \emptyset,
\]
with $l_B(e):=\sup_{y\in B} y\cdot e$.
Since $l_B$ and $c^*$ are continuous by their definitions, there is $M$ such that  $Q_M:=B_M\cap\{x_1\ge \frac 1M\}\subseteq\bbR^d$ satisfies 
\[
Q\subseteq \bigcup_{e\in \mathbb{S}^{d-1}}  (R_eQ_M + l_B(e)e+c^*(e)t_0e )
\]
But then an application of Theorem \ref{T.4.4}(iii)
with $\lambda:=\Lambda+\frac 1\alpha$ and
$y=l_{B}(e)e+y_\ve$ for any $\ve>0$ and $e\in\mathbb S^{d-1}$
shows that $w^\ve$ solving \eqref{4.0e} and
\[
 w^\ve(0,\cdot)= (1-\alpha)\chi_{B+y_\ve}
\]
satisfies
\[
\lim_{\ve\to 0} \sup_{(t,x)\in[0,t_0]\times Q} w^\ve(t,x+y_\ve)= 0.
\]
Since $Q\subseteq B_{\Lambda-2c't_0}$ and $B:=A'\cap B_\Lambda$, we have from this and Lemma \ref{L.4.2},
\[
0\le \lim_{\ve\to 0} \sup_{(t,x)\in K} u^\ve(t,x+y_\ve) \le \lim_{\ve\to 0}  c'e^{-m'(2c'\ve^{-1}t_0-c'\ve^{-1}t_0)}=0.
\]
This establishes the convergence to 0 claim, and the proof is finished.
\end{proof}

\section{Homogenization For General Initial Sets}\label{s:fstohom}

In this section we prove Theorems \ref{T.1.4}(iii) and \ref{T.1.6}(ii), as well as the same results for \eqref{1.10}.  We will again consider a fixed $\om\in\Om$ (which is again dropped from the notation), and hence show that if strong exclusive front speeds $c^*(e)$ in all directions $e$ exist for \eqref{4.0}, then solutions to \eqref{4.0e} with appropriate families of initial conditions converge to (discontinuous) viscosity solutions of the Hamilton-Jacobi equation \eqref{eq:hjeq} as $\ve\to 0$.
Specifically, if the initial conditions satisfy \eqref{4.4b} below, then solutions converge to $\chi_{\Theta^{A,c^*}}$, with the set $\Theta^{A,c^*}$ from \eqref{1.40'} obtained from solving \eqref{eq:hjeq} with initial condition \eqref{eq:hjeq5'}.



\subsection{Hamilton-Jacobi equations and viscosity solutions}
We begin by recalling some basic properties of Hamilton-Jacobi equations and their viscosity solutions. We want to consider the PDE \eqref{eq:hjeq}
with initial condition 
\begin{equation}\label{5.0}
\overline{u}(0,\cdot)=\chi_{A},
\end{equation}
where $A\subseteq\RR^d$ is an open set.
As this results in us having to consider discontinuous functions, we will have to employ the notion of viscosity solutions (the ones we use here are also called {\it discontinuous viscosity solutions}).  With $u^*,u_*:[0, \infty)\times\RR^{d}\rightarrow \RR$ being the upper and lower semicontinuous envelopes
of a function $u:(0, \infty)\times\RR^{d}\rightarrow \RR$ that is bounded on bounded sets, we have the following definition.

\begin{definition}[\cite{capdolbook}]  \label{D.5.1}
A function $u:(0, \infty)\times\RR^{d}\rightarrow \RR$ is a {\it viscosity subsolution} to \eqref{eq:hjeq} if for any $\phi\in C^{1}((0, \infty)\times \RR^{d})$ we have  
\begin{equation}\label{6.2}
\phi_{t}(t_{0}, x_{0})\leq c^{*}\left(-\frac{\nabla \phi(t_{0}, x_{0})}{|\nabla \phi(t_{0}, x_{0})|}\right)|\nabla \phi(t_{0}, x_{0})|
\end{equation}
whenever $u^{*}-\phi$ has a local maximum at $(t_{0}, x_{0})$.
Similarly,  $u$ is a {\it viscosity supersolution} to \eqref{eq:hjeq} if  for any $\phi\in C^{1}((0, \infty)\times \RR^{d})$ we have  
\begin{equation}\label{6.3}
\phi_{t}(t_{0}, x_{0})\geq c^{*}\left(-\frac{\nabla \phi(t_{0}, x_{0})}{|\nabla \phi(t_{0}, x_{0})|}\right)|\nabla \phi(t_{0}, x_{0})|
\end{equation}
whenever $u_{*}-\phi$ has a local minimum at $(t_{0}, x_{0})$.

We say that $u$ is a {\it viscosity solution} to \eqref{eq:hjeq} when it is both a viscosity subsolution and a viscosity supersolution.  We also say that $u$ satisfies initial condition \eqref{5.0} whenever (with $A^0={\rm int}(A)$)
\[
\chi_{A^0}\le u_*(0,\cdot)\le u^*(0,\cdot)\le \chi_{\bar A}.
\]
\end{definition}

\medskip

Since we can always add a constant to $\phi$, a function  $u$ is a viscosity subsolution provided \eqref{6.2} holds whenever $\phi\in C^1$ satisfies $\phi(t_0,x_0)=u^*(t_0,x_0)$ as well as  $\phi\ge u^*$ on a neighborhood of $(t_0,x_0)$, and $u$ is a viscosity supersolution provided \eqref{6.3} holds whenever $\phi\in C^1$ satisfies $\phi(t_0,x_0)=u_*(t_0,x_0)$ as well as  $\phi\le u_*$ on a neighborhood of $(t_0,x_0)$.

While
the PDE \eqref{eq:hjeq} satisfies a comparison principle for upper semicontinuous subsolutions and lower semicontinuous supersolutions  (see Theorem \ref{thm:soravia}(iii) below),
the question of uniqueness of discontinuous viscosity solutions
is more subtle
(see, e.g., \cite{bss} for some counterexamples).  However,  the initial value problem \eqref{eq:hjeq}+\eqref{5.0} does admit a unique (discontinuous) viscosity solution (up to sets of measure 0) for certain initial conditions and under appropriate hypotheses on $c^{*}$.  This turns out to be closely related to 
the following definition.



\begin{definition}\label{def:noint}
Let $\overline{v}$ be the unique continuous viscosity solution to \eqref{eq:hjeq} with initial condition $\overline{v}(0,\cdot)=\overline{v}_{0}$, where $\overline{v}_{0}:\RR^{d}\rightarrow \RR$ is any uniformly continuous function satisfying
\begin{equation}\label{eq:hjeq5'}
\overline{v}_{0}(x)=\begin{cases}>0 & x\in A^0,\\
0& x\in \partial A,\\
<0& x\in \RR^{d}\setminus \overline{A}.
\end{cases}
\end{equation}
We say that the initial value problem \eqref{eq:hjeq}+\eqref{5.0} {\it does not develop an interior} if for any such $\overline{v}_0$, we have
\begin{equation}\label{eq:noint}
\left|\left\{(t,x)\,\mid\, \overline{v}(t,x)=0\right\}\right|=0, 
\end{equation}
with $|\cdot|$ the Lebesgue-measure on $\RR^{d+1}$.
\end{definition}

This definition differs slightly from those in \cite{bss, barlestakis, takisfronts, soravia} and other references; see Remark 2 after Theorem \ref{thm:soravia} for a discussion of this.
Also observe that since \eqref{eq:hjeq} satisfies both uniqueness and a comparison principle for continuous viscosity solutions, it follows that the sets $\{\overline v>0\}$, $\{\overline v<0\}$, $\{\overline v=0\}$ for continuous viscosity solutions to \eqref{eq:hjeq}+\eqref{eq:hjeq5'} are independent of the choice of $\overline{v}_0$ \cite[Theorem 1.4]{takisfronts}.  Therefore, as mentioned in the introduction, 
we define $\Theta^{A, c^{*}}$ via \eqref{1.40'},
with $\overline v$ any such solution.

We now have the following theorem, which is a collection (and combination) of results by Barles, Soner, and Souganidis \cite[Theorems 2.1 and 4.1]{bss}, Crandall, Ishii, and Lions \cite[Theorem 8.2]{users}, Souganidis \cite[Theorem 1.10]{takisfronts}, and Soravia \cite{soravia}.  In particular, it implies that in our setting, \eqref{eq:hjeq} admits a unique discontinuous viscosity solutions for any relevant initial condition.

\begin{thm}
\label{thm:soravia}
\leavevmode
Let $c^{*}:\mathbb S^{d-1}\to \mathbb R$ be Lipschitz  and let $A\subseteq\RR^d$ be open. 
\begin{enumerate}[(i)]
\item The problem 
\eqref{eq:hjeq}{\rm+}\eqref{5.0} has a unique (up to space-time sets of measure 0) viscosity solution $\overline u$  if and only if it does not develop an interior. This solution is then given by $\overline u=\chi_{\left\{\overline{v}>0\right\}}$,
where $\overline{v}$ is any solution to \eqref{eq:hjeq}{\rm+}\eqref{eq:hjeq5'}.
\item If $c^{*}(e)>0$ for all $e\in \mathbb{S}^{d-1}$, then \eqref{eq:hjeq}{\rm +}\eqref{5.0} does not develop an interior. 
\item If $u, v:[0, \infty)\times \RR^{d}\to \RR$ are a bounded upper semicontinuous subsolution and a bounded lower semicontinuous supersolution to \eqref{eq:hjeq}, respectively, and $u(0,x)\leq v(0,x)$ holds for each $x\in\RR^d$, then $u(t,x)\leq v(t,x)$ for all $(t,x)\in[0,\infty)\times\RR^d$. 
\end{enumerate}
\end{thm}

{\it Remarks.}  
1.  In \cite{bss, takisfronts}, the statement of Theorem \ref{thm:soravia}(ii) is formulated in terms of the PDE \eqref{eq:hjeq} with initial condition $\overline{u}(0,\cdot)=\chi_{A^0}-\chi_{\bar A^{c}}$, and the unique viscosity solution is $\overline{u}=\chi_{\left\{\overline{v}>0\right\}}-\chi_{\left\{\overline{v}<0\right\}}$ . The two statements are clearly equivalent.
\smallskip

2. In some earlier papers, the no interior condition was stated in the topological sense, namely
\begin{equation} \label{5.6}
\partial \left\{\overline{v}>0\right\}=\left\{\overline{v}=0\right\}=\partial \left\{ \overline{v}< 0\right\}
\end{equation}
for any solution to \eqref{eq:hjeq}{\rm+}\eqref{eq:hjeq5'}.
The two definitions coincide in most situations of interest but differ in some pathological cases.  For instance, when $c^*\equiv 0$ (so \eqref{eq:hjeq} is $\overline{u}_t=0$) and $A=(0,1)\cup(1,2)\subseteq\RR$, then $v(t,\cdot)=\overline{v}_0$ for each $t>0$, so \eqref{eq:noint} holds but \eqref{5.6} fails (and $\overline{u}(t,x)=\chi_A(x)$ is the only viscosity solution to \eqref{eq:hjeq}{\rm+}\eqref{5.0}).  On the other hand, if $c^*\equiv 0$ and $A,B\subseteq\RR$ are two open sets with $A\cup B$ dense, $\partial A=\partial B$, and $|\partial A|>0$, then again $v(t,\cdot)=\overline{v}_0$ for each $t>0$, so \eqref{eq:noint} fails but \eqref{5.6} holds (and $\overline{u}(t,x)=\chi_{A'}(x)$ is a viscosity solution to \eqref{eq:hjeq}{\rm+}\eqref{5.0} whenever $A\subseteq A'\subseteq \overline{A}$).  The proof of Theorem \ref{thm:soravia}(i,ii) in fact applies with our Definition \ref{def:noint} (see, e.g., \cite{bss}), which justifies its use here.
\smallskip

3.  Note that (ii) implies $|\partial\Theta^{A, c^{*}}|=0$.
\smallskip

\subsection{Homogenization for \eqref{genhomeq} and \eqref{1.10}}

We will again first consider \eqref{4.0}, with the following hypotheses.

\medskip
{\bf Hypothesis H''.}  (i) Assume H'(i).

(ii) Lemma \ref{L.4.2} holds for solutions to \eqref{4.0}.

(iii) The PDE \eqref{4.0} has a strong exclusive front speed $c^*(e)>0$ in each direction $e\in\mathbb S^{d-1}$, in the sense of Remark 3 after Hypothesis H'.
\medskip

{\it Remark.}  Note that we do not  assume the second claim in H'(ii) here.  Its role will instead be played by the assumption $c^*(e)>0$.\smallskip

We now reformulate the last claim in Theorem \ref{T.1.8}(ii) (corresponding to Theorem \ref{T.1.6}(ii)) in terms of \eqref{4.0} and Hypothesis H''.  Note that with $\theta_0$ chosen to satisfy H''(iii), we can  replace $\thet_0+\alpha$ in the statement of the theorem by just $\theta_0$ (see Remark 2 after Hypothesis H').

\begin{thm}\label{T.5.3}
Assume Hypothesis $H''$.  If $A\subseteq\bbR^d$ is open, $R>0$, and $u^\ve$ solves \eqref{4.0e} and
\begin{equation} \label{4.4b}
\theta_0\chi_{A_{\psi(\ve)}^0+y_\ve} \le u^\ve(0,\cdot)\le \chi_{B_{\psi(\ve)}(A)+y_\ve} + \psi(\ve) \chi_{\bbR^d\setminus(B_{\psi(\ve)}(A)+y_\ve)}
\end{equation}
for each $\ve>0$, with some $y_\ve\in B_{R}(0)$ and $\lim_{\ve\to 0}\psi(\ve)=0$,
then
\[
\lim_{\ve\to 0}  u^\ve(t,x+y_{\ve})= \chi_{\Theta^{A,c^*}}(t,x)
\]
locally uniformly on $\left([0, \infty)\times \RR^{d}\right)\setminus \partial \Theta^{A,c^*}$.
\end{thm}

\begin{proof}[Proof of Theorem \ref{T.1.6}(ii) and of the corresponding part of Theorem \ref{T.1.8}]
This is immediate from  Theorem \ref{T.5.3} applied to those $\om\in\Om$ for which \eqref{1.10} (or specifically \eqref{genhomeq}) has strong exclusive front speed $c^*(e)$ in each direction $e\in\mathbb S^{d-1}$.   The set of such $\om$ has full measure by the remark after Theorem \ref{T.1.9}.
\end{proof}


\begin{proof}[Proof of Theorem \ref{T.5.3}]
Since $\Lambda$ in H''(iii) is arbitrary (and hence can be replaced by $\Lambda+R$), we can assume $y_{\ve}=0$
without loss of generality.  
Also, let 
$c'> \|c^*\|_\infty$ be as in Lemma~\ref{L.4.2} (note that Hypothesis H''(ii) implies that $\|c^*\|_\infty\le c'$, and increasing $c'$ keeps Lemma~\ref{L.4.2} valid).

Given a family of functions $\left\{u^{\ve}\right\}$, we let their half-relaxed limits (introduced in \cite{barlesperthame1,barlesperthame2}) be
\begin{align*}
\limsups u^{\ve}(t, x)&=\sup\left\{\limsup_{\ve\rightarrow 0} u^{\ve}(t_\ve,x_{\ve}) \,\mid\, (t_{\ve}, x_{\ve})\rightarrow (t,x)\right\}
\end{align*}
and
\begin{align*}
\liminfs u^{\ve}(t, x)&=\inf\left\{\liminf_{\ve\rightarrow 0} u^{\ve}(t_\ve,x_{\ve}) \,\mid\, (t_{\ve}, x_{\ve})\rightarrow (t,x)\right\}.
\end{align*}
Then let 
\begin{equation} \label{5.21}
\Theta^{1}:=\left\{(t,x)\in (0,\infty)\times\bbR^d \,\mid\, \liminfs u^{\ve}(t,x)=1\right\}^0
\end{equation}
and 
\begin{equation*} 
\Theta^{2}:= 
\left\{(t,x)\in (0,\infty)\times\bbR^d \,\mid\, \limsups u^{\ve}(t,x)=0\right\}^0.
\end{equation*}
It follows then that for any compact $K\subseteq \Theta^{1}$ we have
\begin{equation}\label{eq:locunif1}
\lim_{\ve\rightarrow 0} \inf_{(t,x)\in K}  u^{\ve}(t, x)=1, 
\end{equation}
while for any compact $ K\subseteq  \Theta^{2}$ we have
\begin{equation}\label{eq:locunif0}
\lim_{\ve\rightarrow 0} \sup_{(t,x) \in K}  u^{\ve}(t, x)=0.
\end{equation}
We now claim that the functions $\chi_{\Theta^{1}}$ and $\chi_{[(0, \infty)\times\RR^{d}]\setminus\Theta^{2}}$ are, respectively, a viscosity supersolution and a viscosity subsolution to \eqref{eq:hjeq}.

Let us start with $\overline{u}:=\chi_{\Theta^{1}}$, which is obviously lower semicontinuous, so $\overline{u}_{*}=\overline{u}$.  Let $\phi\in C^{1}((0, \infty)\times \RR^{d})$ be such that $\overline{u}-\phi$ has a strict local minimum at $(t_{0}, x_{0})$. Without loss of generality, we may assume $\phi\le \overline{u}$ and $\phi(t_{0}, x_{0})=\overline{u}(t_{0},x_{0})$. If $(t_0,x_{0})\in \Theta^{1}$ or $(t_0,x_{0})\in [(0,\infty)\times\RR^{d}]\setminus \overline{\Theta^{1}}$, then the claim in the supersolution part of Definition \ref{D.5.1} is  satisfied for $\phi$ and $(t_0,x_0)$ because $\overline{u}$ is locally constant at $(t_{0}, x_{0})$, which implies 
\begin{equation*}
\nabla \phi(t_{0}, x_{0})=0\qquad\text{and}\qquad \phi_{t}(t_{0}, x_{0})=0. 
\end{equation*}

It remains to consider the case $(t_0,x_0)\in \partial \Theta^{1}$, when $\overline{u}(t_0,x_0)=0$. 
Let us first assume that $\nabla \phi(t_{0}, x_{0})=0$, so that we need to show $\phi_{t}(t_{0}, x_{0})\ge 0$.  Hence let us assume, towards contradiction, that $\phi_{t}(t_{0}, x_{0})<0$. 
Then for each small enough $h\in\left(0,\frac{t_0}{2}\right)$ we have $\inf_{|x-x_0|<5c'h} \phi(t_{0}-2h,x)>0$ because $\phi\in C^1$ (fix one such $h$), so that $\{t_{0}-2h\}\times B_{5c'h}(x_0)\subseteq \Theta^1$.  
Thus we can apply Lemma \ref{L.4.2} to  $u_\ve(t,x):=u(t,x;(x_0\cdot e + 5c'h)\ve^{-1}, e)$ for any $e\in\mathbb S^{d-1}$ (with $u$ from Hypothesis H''(iii)) and $u_\ve'(t,x):=u^\ve(\ve t+t_0-2h, \ve x)$, which solve \eqref{4.0} and satisfy $u_\ve(0,\cdot)\le u_\ve'(0,\cdot)$ on $B_{5c'h/\ve}(\frac {x_0}\ve)$ for all small enough $\ve>0$.  The lemma and Hypothesis H''(iii) then show 
\[
\lim_{\ve\to 0} \, \inf_{(t,x)\in (h/\ve, 3h/\ve)\times B_{c'h/\ve}(x_0/\ve)} u_\ve'(t,x)=1.
\]
Therefore $(t_0-h,t_0+h)\times B_{c'h}(x_0)\subseteq\Theta^1$, contradicting $(t_0,x_0)\in \partial \Theta^{1}$.  

Let us now assume that $p:=-\nabla \phi(t_{0}, x_{0})\neq 0$ (and let $\hat p:=\frac p{|p|}$), so that we need to show $\phi_{t}(t_{0}, x_{0})\ge c^*(\hat p)|p|$.  Hence let us assume that $\phi_{t}(t_{0}, x_{0})<s|p|<c^*(\hat{p})|p|$ for some $s\in\RR$.  Since $\phi\in C^1$, it follows that for each small enough $h\in(0,t_0)$ we have $\phi(t_{0}-h,x)>0$ for all $x\in B_{2c'h}(x_0)$ such that $(x-x_0)\cdot \hat p \le -sh$ (fix one such $h$).  This means 
\[
\{t_{0}-h\}\times \left(B_{2c'h}(x_0)\cap \left\{x\,\mid\, (x-x_0)\cdot \hat p \le -sh \right\} \right)\subseteq \Theta^1.
\]
Thus we can apply Lemma \ref{L.4.2} to  $u_\ve(t,x):=u(t,x;(x_0\cdot \hat p-sh)\ve^{-1},\hat p)$ (with $u$ from Hypothesis H''(iii)) and $u_\ve'(t,x):=u^\ve(\ve t+t_0-h, \ve x)$, which solve \eqref{4.0} and satisfy $u_\ve(0,\cdot)\le u_\ve'(0,\cdot)$ on $B_{2c'h/\ve}(\frac {x_0}\ve)$ for all small enough $\ve>0$.  The lemma and Hypothesis H''(iii) then show 
\[
\lim_{\ve\to 0} \, \inf_{(t,x)\in ((h-\delta)/\ve, (h+\delta)/\ve)\times B_{\delta/\ve}(x_0/\ve)} u_\ve'(t,x)=1
\]
for some small $\delta>0$ because $c^*(\hat p)>s$.  But this now means that $(t_0-\delta,t_0+\delta)\times B_\delta(x_0)\subseteq\Theta^1$, contradicting $(t_0,x_0)\in \partial \Theta^{1}$.  Therefore $\chi_{\Theta^{1}}$ is indeed a viscosity supersolution to \eqref{eq:hjeq}.

The argument for $\overline{u}:=\chi_{[(0, \infty)\times\RR^{d}]\setminus\Theta^{2}}$ is similar, but using exclusivity of the front speeds $c^*(e)$.  Again notice that $\overline{u}$ is upper semicontinuous, so $\overline{u}^{*}=\overline{u}$.  Let $\phi\in C^{1}((0, \infty)\times \RR^{d})$ be such that $\overline{u}-\phi$ has a strict local maximum at $(t_{0}, x_{0})$. Without loss of generality, we may assume $\phi\ge \overline{u}$ and $\phi(t_{0}, x_{0})=\overline{u}(t_{0},x_{0})$. If $(t_0,x_{0})\in \Theta^{2}$ or $(t_0,x_{0})\in [(0,\infty)\times\RR^{d}]\setminus \overline{\Theta^{2}}$, then the claim in the subsolution part of  Definition \ref{D.5.1} is  satisfied for $\phi$ and $(t_0,x_0)$ because $\overline{u}$ is locally constant at $(t_{0}, x_{0})$, which implies 
\begin{equation*}
\nabla \phi(t_{0}, x_{0})=0\qquad\text{and}\qquad \phi_{t}(t_{0}, x_{0})=0. 
\end{equation*}

It remains to consider the case $(t_0,x_0)\in \partial \Theta^{2}$, when $\overline{u}(t_0,x_0)=1$. 
Let us first assume that $\nabla \phi(t_{0}, x_{0})=0$, so that we need to show $\phi_{t}(t_{0}, x_{0})\le 0$.  Hence let us assume that $\phi_{t}(t_{0}, x_{0})>0$. Then for each small enough $h\in\left(0,\frac{t_0}{2}\right)$ we have $\sup_{|x-x_0|<5c'h} \phi(t_{0}-2h,x)<1$ (fix one such $h$), which means that $\{t_{0}-2h\}\times B_{5c'h}(x_0)\subseteq  \Theta^2$.  
Thus we can apply Lemma \ref{L.4.2} to  $u_\ve'(t,x):=u(t,x;(x_0 - 5c'he)\ve^{-1}, e,\alpha)$ for any $e\in\mathbb S^{d-1}$ and any small $\alpha>0$ (with $u$ from Hypothesis H''(iii)) and $u_\ve(t,x):=u^\ve(\ve t+t_0-2h, \ve x)$, which solve \eqref{4.0} and satisfy $u_\ve(0,\cdot)\le u_\ve'(0,\cdot)$ on $B_{5c'h/\ve}(\frac {x_0}\ve)$ for all small enough $\ve>0$.  The lemma and Hypothesis H''(iii) with 
the compact $K:=\{x\in \overline{B_{6c'}(0)}\,:\, x\cdot e\ge \frac{c'}3\}$   (and any $\Lambda\ge h^{-1}|x_0| + 5c'$)  then show
\[
\limsup_{\ve\to 0} \, \sup_{(t,x)\in (h/\ve, 3h/\ve)\times B_{c'h/\ve}(x_0/\ve)} u_\ve(t,x) \le \beta_{K,e}(\alpha) \qquad(\to 0 \text{ as $\alpha\to 0$}).
\]
Therefore $(t_0-h,t_0+h)\times B_{c'h}(x_0)\subseteq \Theta^2$,
contradicting $(t_0,x_0)\in \partial \Theta^{2}$.

Let us now assume that $p:=-\nabla \phi(t_{0}, x_{0})\neq 0$ (and let $\hat p:=\frac p{|p|}$), so that we need to show $\phi_{t}(t_{0}, x_{0})\le c^*(\hat p)|p|$.  Hence let us assume that $\phi_{t}(t_{0}, x_{0})>s|p|>c^*(\hat{p})|p|$ for some $s\in\RR$.  Since $\phi\in C^1$, it follows that for each small enough $h\in(0,t_0)$ we have $\phi(t_{0}-h,x)<1$ for all $x\in B_{2c'h}(x_0)$ such that $(x-x_0)\cdot \hat p \ge -sh$ (fix one such $h$).  This means 
\[
\{t_{0}-h\}\times \left(B_{2c'h}(x_0)\cap \left\{x\,\mid\, (x-x_0)\cdot \hat p \ge -sh \right\} \right)\subseteq \Theta^2.
\]
Thus we can apply Lemma \ref{L.4.2} to  $u_\ve'(t,x):=u(t,x;(x_0 -sh\hat p)\ve^{-1},\hat p)$ (with $u$ from Hypothesis H''(iii)) and $u_\ve(t,x):=u^\ve(\ve t+t_0-h, \ve x)$, which solve \eqref{4.0} and satisfy $u_\ve(0,\cdot)\le u_\ve'(0,\cdot)$ on $B_{2c'h/\ve}(\frac {x_0}\ve)$ for all small enough $\ve>0$.  The lemma and Hypothesis H''(iii) then show 
\[
\limsup_{\ve\to 0} \, \inf_{t\in ((h-\delta)/\ve, (h+\delta)/\ve)\,\&\, x\in B_{\delta/\ve}(x_0/\ve)} u_\ve(t,x) \le \beta_{K,\hat p}(\alpha) \qquad(\to 0 \text{ as $\alpha\to 0$})
\]
for some small $\delta>0$ and some compact $K\subseteq \{x\cdot \hat p>0\}$
because $c^*(\hat p)<s$.  But this  means $(t_0-\delta,t_0+\delta)\times B_\delta(x_0)\subseteq\Theta^2$, contradicting $(t_0,x_0)\in \partial \Theta^{2}$.  Therefore $\chi_{[(0, \infty)\times\RR^{d}]\setminus\Theta^{2}}$ is indeed a viscosity subsolution to \eqref{eq:hjeq}.

We next need to show that both $\chi_{\Theta^1}$ and $\chi_{[(0, \infty)\times\RR^{d}]\setminus\Theta^{2}}$ satisfy initial condition \eqref{5.0}.  Let $\Theta^j_t:=\{x\,\mid\,(t,x)\in\Theta^j\}$ for $t>0$ and $j=1,2$, and pick any $s\in(0,\inf_{e\in\mathbb S^{d-1}} c^*(e))$. (Note that Theorem \ref{T.4.4} holds here because the second claim in Hypothesis H'(ii) was only used in its proof to show that $c^*(e)>0$, which we assume in H''(iii).  Therefore $c^*$ is continuous and hence $\inf_{e\in\mathbb S^{d-1}} c^*(e)>0$.)  An argument as in the last part of the proof that $\chi_{\Theta^1}$ is a supersolution, using Lemma \ref{L.4.2} and existence of strong front speeds in all directions, then shows that if $B_r(x)\subseteq A$ for some $x\in A$ and $r>0$, then $B_{r+sh}(x)\subseteq \Theta^1_h$ for all small enough $h>0$.  Repeating this recursively and using that $A$ is open we obtain $B_{st}(A)\subseteq \Theta^1_t$ for all $t>0$.  This shows that $\bigcap_{t>0} \Theta^1_t \supseteq \overline{A}$.  

Next pick any $s>\|c^*\|_\infty$.  An argument as in the last part of the proof that $\chi_{[(0, \infty)\times\RR^{d}]\setminus\Theta^{2}}$ is a subsolution, using Lemma \ref{L.4.2} and existence of strong exclusive front speeds in all directions, then shows that if now $B_r(x)\subseteq \bbR^d\setminus\overline{A}$, then $B_{r-sh}(x)\subseteq \Theta^2_h$ for all small enough $h>0$.  Repeating this recursively now yields $(\bbR^d\setminus\overline{A})^0_{st}\subseteq \Theta^2_t$ for all $t>0$.  
This, together with $\bigcap_{t>0} \Theta^1_t \supseteq \overline{A}$ and $\Theta^1\cap\Theta^2=\emptyset$, shows that in fact $\bigcap_{t\in(0,T)} \Theta^1_t = \overline{A}$ for each $T>0$ as well as $(\bbR^d\setminus\overline{A})^0_{st}\subseteq \Theta^2_t\subseteq \bbR^d\setminus\overline{A}$ for each $t>0$.

The former claim means that 
$(\chi_{\Theta^1})_*=\chi_{\Theta_*}$ where
\[
\Theta_*:=\Theta^1\cup [\{0\}\times(\overline{A})^0], 
\]
while the latter means that $(\chi_{[(0, \infty)\times\RR^{d}]\setminus\Theta^{2}})^*=\chi_{\Theta^*}$ where
\[
\Theta^*:= ([(0, \infty)\times\RR^{d}]\setminus\Theta^{2})\cup [\{0\}\times\overline{A}].
\]
Hence both $\chi_{\Theta^1}$ and $\chi_{[(0, \infty)\times\RR^{d}]\setminus\Theta^{2}}$ satisfy \eqref{5.0}.

Next we would like to apply Theorem \ref{thm:soravia}(iii) to the subsolution $\chi_{\Theta^*}$ and supersolution $\chi_{\Theta_*}$ but we cannot do that because they are not appropriately ordered at $t=0$. (In fact, $\Theta^1\cap\Theta^2=\emptyset$ even yields $\chi_{\Theta_*}\le \chi_{\Theta^*}$ pointwise.)  Nevertheless, since we proved that $(\overline{A}\subseteq)\, B_{sh}(A)\subseteq \Theta^1_h$ for some $s>0$ and all $h>0$, we can use the supersolution $\chi_{\Theta_*}^h(t,x):=\chi_{\Theta_*}(t+h,x)$ instead of $\chi_{\Theta_*}$.  Then Theorem \ref{thm:soravia}(iii) yields for each $h>0$ the second of the pointwise inequalities
\begin{equation} \label{5.55}
\chi_{\Theta_*}\le \chi_{\Theta^*}\le \chi_{\Theta_*}^h.
\end{equation}
But now for any $(T,y)\in(0,\infty)\times\RR^d$ we have 
\[
\int_{[0,T]\times B_1(y)} (\chi_{\Theta^*} - \chi_{\Theta_*})\,dxdt \le \int_{[T-h,T]\times B_1(y)} \chi_{\Theta_*}^h\,dxdt \to 0 \qquad\text{as $h\to 0$,}
\]
so $\chi_{\Theta_*}= \chi_{\Theta^*}$ up to a set of measure 0.  This of course
means that  $[(0, \infty)\times\RR^{d}]\setminus[\Theta^1\cup \Theta^2]$ has zero measure and $\chi_{\Theta^1}$ is actually a solution to \eqref{eq:hjeq}{\rm+}\eqref{5.0}.  Theorem \ref{thm:soravia}(i,ii) now shows that it is the unique solution, and also that $\Theta^1=\Theta^{A,c^*}$.
(We note that one can similarly show that $u\le \chi_{\Theta_*}^h$ for $h>0$ and any upper semicontinuous subsolution $u$ to \eqref{eq:hjeq}{\rm+}\eqref{5.0}, as well as $\chi_{\Theta^*}\le u^h$ for $h>0$ and any lower semicontinuous supersolution $u$.  This then yields an alternative proof of uniqueness of the viscosity solution $\chi_{\Theta^1}$.)

Since $[(0, \infty)\times\RR^{d}]\setminus[\Theta^1\cup \Theta^2]$ contains no open balls, it must be contained in $\partial\Theta^2\cup \partial\Theta^1$.  And since definitions of $\Theta^1$ and $\Theta^2$ show that
\[
\lim_{\ve\to 0}  u^\ve(t,x)=
\begin{cases}
1 & (t,x)\in\Theta^1, \\
0 & (t,x)\in\Theta^2\cup[\{0\}\times(\RR^d\setminus\overline{A})], \\
\end{cases}
\]
it remains to prove $\partial\Theta^2\cap [(0,\infty)\times\RR^d]\subseteq \partial\Theta^1$ (note that $\{0\}\times\overline{A}\subseteq \partial\Theta^1$).  But if $(t,x)\in \partial\Theta^2$ for some $t>0$, then $(t,x)\notin\Theta^1\cup\Theta^2$ because the two sets are open and disjoint.  This means that $\chi_{\Theta^*}(t,x)=1$, so \eqref{5.55} shows $(t+h,x)\in \Theta^1$ for all $h>0$. It follows that $(t,x)\in\partial\Theta^1$, and the proof is finished.  (In fact, \eqref{5.55} implies $\partial\Theta^2\cap [(0,\infty)\times\RR^d]= \partial\Theta^1\cap [(0,\infty)\times\RR^d]$.)  
\end{proof}

We can now also prove Theorem \ref{T.1.4}(iii). 

\begin{proof}[Proof of Theorem \ref{T.1.4}(iii) and of the corresponding part of Theorem \ref{T.1.8}]
The argument in the proof of Theorem \ref{T.1.6}(ii) above shows that there is a full measure set of $\om$ such that \eqref{genhomeq} (resp. \eqref{1.10}) with this $\om$ has strong front speed $c^*(e)$ in each direction $e\in \mathbb{S}^{d-1}$.
Fix any such $\om$ (then Hypothesis H'' will be satisfied with the word ``exclusive'' removed) and as in the proof of Theorem \ref{T.5.3}, we can assume without loss of generality that $y_{\ve}=0$. 

As in  the proof of Theorem \ref{T.5.3}, consider $\Theta^1$  from \eqref{5.21}.
Note that in that proof  we only used H''(i,ii), and the strong front speeds claim in H''(iii) to show that $\chi_{\Theta^{1}}$ is a (lower semicontinuous) viscosity supersolution to \eqref{eq:hjeq} and $\bigcap_{t>0} \Theta^1_t \supseteq \overline{A}$.  Therefore this is still the case now, even though existence of exclusive front speeds is not assumed.

Instead of exclusive front speeds, we will use convexity of  $A$.  Let
\begin{equation} \label{5.23}
\Theta^3 := \bigcap_{e\in \mathbb{S}^{d-1}} \left \{(t,x)\in (0,\infty)\times\bbR^d\,\mid\, x\cdot e <  \sup_{y\in \partial A} y\cdot e+c^*(e)t \right \}.
\end{equation}
We now claim that it suffices to show that $\Theta^3 \subseteq \Theta^{A,c^{*}}$.

Let us assume this is the case.  From the hypotheses, the definition of strong front speeds, and the comparison principle it follows that $\limsups u^{\ve}(t,x)=0$ for each $(t,x)$ in the set
\[
 ([0, \infty)\times\RR^{d})\setminus \overline{\Theta^{3}}  = \bigcup_{e\in \mathbb{S}^{d-1}}  \left\{
 (t,x)\in [0,\infty)\times\bbR^d\,\mid\, 
 x\cdot e >  \sup_{y\in \partial A} y\cdot e+c^*(e)t \right \}.
\]
(Note that $\overline{\Theta^3}\cap[\{0\}\times\RR^d]=\overline A$.)
It follows that $\Theta^1 \subseteq \overline{\Theta^{3}}\subseteq \overline{\Theta^{A,c^{*}}}$.  This and
\[
\overline{\Theta^{A,c^{*}}}\cap [\{0\}\times\RR^d]= \overline {A} \subseteq \bigcap_{t>0} \Theta^1_t
\]
(the equality being from the definition of $\Theta^{A,c^{*}}$ and $\|c^*\|_\infty<\infty$) allow us to apply the argument from the last proof (using time shifts by $h>0$ and Theorem \ref{thm:soravia}(iii)) to the subsolution $\chi_{\Theta^*}$ with $\Theta_*:=\Theta^1\cup [\{0\}\times(\overline{A})^0]$ and the (super)solution $\chi_{\Theta_*}$ with $\Theta^*:=\overline{\Theta^{A,c^{*}}}$.  This yields $\Theta^1=\Theta^3=\Theta^{A,c^*}$, finishing the proof.



It remains to prove $\Theta^3 \subseteq \Theta^{A,c^{*}}$.
Consider $\overline{v}_{0}$ defined by 
\begin{equation*}
\overline{v}_{0}(x)=\begin{cases}
{\rm dist}(x,\partial A)& x\in A,\\
-{\rm dist}(x,\partial A)& x\in \RR^d\setminus A.
\end{cases}
\end{equation*}
%
Then $\overline{v}_{0}$ is uniformly continuous, concave, and satisfies \eqref{eq:hjeq5'}.  It follows that ${\Theta^{A, c^{*}}}={\left\{\overline{v}>0\right\}}$ for $\overline{v}$ the unique viscosity solution of \eqref{eq:hjeq} with $\overline{v}(0,\cdot)=\overline{v}_{0}$.
Since $\overline{v}_{0}$ is concave and grows at most linearly, we can apply the Hopf-Lax formula for convex initial  data \cite{bardievans} to obtain that
\begin{equation*}
-\overline{v}(t,x)=\sup_{y\in \RR^{d}} \inf_{z\in \RR^{d}} \left\{-\overline{v}_{0}(z)+(x-z)\cdot y-c^{*}\left(\frac{y}{|y|}\right)|y|t\right\}. 
\end{equation*}
Since $\overline{v}_0|_{\partial A}=0$, 
 it follows that 
\[
\overline{v}(t,x)\geq \inf_{y\in \RR^{d}} \sup_{z\in \partial A}\left\{(z-x)\cdot y+c^{*}\left(\frac{y}{|y|}\right)|y|t\right\}. 
\]
Then obviously $\overline v(t,x)>0$ for each $(t,x)\in (0,\infty)\times\RR^d$ such that
\begin{equation}\label{5.56}
\inf_{e\in \mathbb S^{d-1}} \sup_{z\in \partial A}\left\{(z-x)\cdot e+c^{*}\left(e\right)t\right\}>0.
\end{equation}
But compactness of $\mathbb S^{d-1}$ and continuity of $c^*$ show that these are precisely the points from $\Theta^3$, so indeed $\Theta^3 \subseteq \Theta^{A,c^{*}}$.
\end{proof}

It follows from the above that the set $\Theta^{A,c^{*}}$ is precisely the set of $(t,x)$ satisfying \eqref{5.56} when $A$ is convex.  Convexity shows that these are obviously those points satisfying 
\[
\inf_{z\in \partial A} \inf_{e\in n_z} \left\{(z-x)\cdot e+c^{*}\left(e\right)t\right\}>0,
\]
where $n_z$ is the set of all outer unit vectors for $A$ at $z\in\partial A$.  When $A$ is convex, bounded, and $C^1$ (so $n_z$ contains a single vector for each $z$), the characteristic function of this set has also been proved to be the homogenization limit in the case of periodic monostable reactions \cite{alfarogiletti}.

\bibliographystyle{amsplain}
\bibliography{ign_homog_bib}

\def\cprime{$'$} \def\polhk#1{\setbox0=\hbox{#1}{\ooalign{\hidewidth
  \lower1.5ex\hbox{`}\hidewidth\crcr\unhbox0}}}
\providecommand{\bysame}{\leavevmode\hbox to3em{\hrulefill}\thinspace}
\providecommand{\MR}{\relax\ifhmode\unskip\space\fi MR }
\providecommand{\MRhref}[2]{%
  \href{http://www.ams.org/mathscinet-getitem?mr=#1}{#2}
}
\providecommand{\href}[2]{#2}
\begin{thebibliography}{10}

\bibitem{alfarogiletti}
M.~Alfaro and T.~Giletti, \emph{Asymptotic analysis of a monostable equation in
  periodic media}, Tamkang J. Math. \textbf{47} (2016), no.~1, 1--26.
  \MR{3474533}

\bibitem{scottpierre}
S.~N. Armstrong and P.~Cardaliaguet, \emph{Stochastic homogenization of
  quasilinear hamilton-jacobi equations and geometric motions}, 2015,
  arXiv:1504.02045 [math.AP].

\bibitem{arms-soug}
S.~N. Armstrong and P.~E. Souganidis, \emph{Stochastic homogenization of
  level-set convex {H}amilton-{J}acobi equations}, Int. Math. Res. Not. IMRN
  (2013), no.~15, 3420--3449. \MR{3089731}

\bibitem{armstrongtran}
S.~N. Armstrong and H.~V. Tran, \emph{Stochastic homogenization of viscous
  {H}amilton-{J}acobi equations and applications}, Anal. PDE \textbf{7} (2014),
  no.~8, 1969--2007. \MR{3318745}

\bibitem{armstrongtranyu1}
S.~N. Armstrong, H.~V. Tran, and Y.~Yu, \emph{Stochastic homogenization of a
  nonconvex {H}amilton-{J}acobi equation}, Calc. Var. Partial Differential
  Equations \textbf{54} (2015), no.~2, 1507--1524. \MR{3396421}

\bibitem{armstrongtranyu2}
\bysame, \emph{Stochastic homogenization of nonconvex {H}amilton-{J}acobi
  equations in one space dimension}, J. Differential Equations \textbf{261}
  (2016), no.~5, 2702--2737. \MR{3507985}

\bibitem{aronwein}
D.~G. Aronson and H.~F. Weinberger, \emph{Multidimensional nonlinear diffusion
  arising in population genetics}, Adv. in Math. \textbf{30} (1978), no.~1,
  33--76. \MR{511740}

\bibitem{fppbook}
A.~Auffinger, M.~Damron, and J.~Hanson, \emph{50 years of first passage
  percolation}, arXiv:1511.03262.

\bibitem{capdolbook}
M.~Bardi and I.~Capuzzo-Dolcetta, \emph{Optimal control and viscosity solutions
  of {H}amilton-{J}acobi-{B}ellman equations}, Systems \& Control: Foundations
  \& Applications, Birkh\"auser Boston, Inc., Boston, MA, 1997, With appendices
  by Maurizio Falcone and Pierpaolo Soravia. \MR{1484411}

\bibitem{bardievans}
M.~Bardi and L.~C. Evans, \emph{On {H}opf's formulas for solutions of
  {H}amilton-{J}acobi equations}, Nonlinear Anal. \textbf{8} (1984), no.~11,
  1373--1381. \MR{764917}

\bibitem{barlesperthame1}
G.~Barles and B.~Perthame, \emph{Discontinuous solutions of deterministic
  optimal stopping time problems}, RAIRO Mod\'el. Math. Anal. Num\'er.
  \textbf{21} (1987), no.~4, 557--579. \MR{921827}

\bibitem{barlesperthame2}
\bysame, \emph{Exit time problems in optimal control and vanishing viscosity
  method}, SIAM J. Control Optim. \textbf{26} (1988), no.~5, 1133--1148.
  \MR{957658}

\bibitem{bss}
G.~Barles, H.~M. Soner, and P.~E. Souganidis, \emph{Front propagation and phase
  field theory}, SIAM J. Control Optim. \textbf{31} (1993), no.~2, 439--469.
  \MR{1205984 (94c:35005)}

\bibitem{barlestakis}
G.~Barles and P.~E. Souganidis, \emph{A new approach to front propagation
  problems: theory and applications}, Arch. Rational Mech. Anal. \textbf{141}
  (1998), no.~3, 237--296. \MR{1617291 (99c:35106)}

\bibitem{becker}
M.~E. Becker, \emph{Multiparameter groups of measure-preserving
  transformations: a simple proof of {W}iener's ergodic theorem}, Ann. Probab.
  \textbf{9} (1981), no.~3, 504--509. \MR{614635}

\bibitem{berhamel}
H.~Berestycki and F.~Hamel, \emph{Front propagation in periodic excitable
  media}, Comm. Pure Appl. Math. \textbf{55} (2002), no.~8, 949--1032.
  \MR{1900178 (2003d:35139)}

\bibitem{bernadin1D}
H.~Berestycki and G.~Nadin, \emph{Spreading speeds for one-dimensional
  monostable reaction-diffusion equations}, J. Math. Phys. \textbf{53} (2012),
  no.~11, 115619, 23. \MR{3026564}

\bibitem{cafleemel}
L.~A. Caffarelli, K.-A. Lee, and A.~Mellet, \emph{Homogenization and flame
  propagation in periodic excitable media: the asymptotic speed of
  propagation}, Comm. Pure Appl. Math. \textbf{59} (2006), no.~4, 501--525.
  \MR{2199784 (2007d:35011)}

\bibitem{card-soug}
P.~Cardaliaguet and P.~E. Souganidis, \emph{On the existence of correctors for
  the stochastic homogenization of viscous {H}amilton-{J}acobi equations}, C.
  R. Math. Acad. Sci. Paris \textbf{355} (2017), no.~7, 786--794. \MR{3673054}

\bibitem{users}
M.G. Crandall, H.~Ishii, and P.-L. Lions, \emph{User's guide to viscosity
  solutions of second order partial differential equations}, Bull. Amer. Math.
  Soc. (N.S.) \textbf{27} (1992), no.~1, 1--67. \MR{1118699 (92j:35050)}

\bibitem{kosdavini}
A.~Davini and E.~Kosygina, \emph{Homogenization of viscous and non-viscous {HJ}
  equations: a remark and an application}, Calc. Var. Partial Differential
  Equations \textbf{56} (2017), no.~4, Art. 95, 21. \MR{3661019}

\bibitem{davini}
A.~Davini and A.~Siconolfi, \emph{Metric techniques for convex stationary
  ergodic {H}amiltonians}, Calc. Var. Partial Differential Equations
  \textbf{40} (2011), no.~3-4, 391--421. \MR{2764912}

\bibitem{evanshom}
L.~C. Evans, \emph{Periodic homogenisation of certain fully nonlinear partial
  differential equations}, Proc. Roy. Soc. Edinburgh Sect. A \textbf{120}
  (1992), no.~3-4, 245--265. \MR{1159184}

\bibitem{willtakis}
W.~M. Feldman and P.~E. Souganidis, \emph{Homogenization and non-homogenization
  of certain non-convex {H}amilton--{J}acobi equations}, J. Math. Pures Appl.
  (9) \textbf{108} (2017), no.~5, 751--782. \MR{3711473}

\bibitem{Fisher}
R.A. Fisher, \emph{The advance of advantageous genes}, Ann. Eugenics \textbf{7}
  (1937), 335--361.

\bibitem{freidbook}
M.~Freidlin, \emph{Functional integration and partial differential equations},
  Annals of Mathematics Studies, vol. 109, Princeton University Press,
  Princeton, NJ, 1985. \MR{833742}

\bibitem{morefreid}
M.~I. Freidlin, \emph{On wavefront propagation in periodic media}, Stochastic
  analysis and applications, Adv. Probab. Related Topics, vol.~7, Dekker, New
  York, 1984, pp.~147--166. \MR{776979}

\bibitem{gao2}
H.~Gao, \emph{Stochastic homogenization of certain nonconvex
  {H}amilton-{J}acobi equations}, arXiv:1803.08633.

\bibitem{gao1}
\bysame, \emph{Random homogenization of coercive {H}amilton-{J}acobi equations
  in 1d}, Calc. Var. Partial Differential Equations \textbf{55} (2016), no.~2,
  Art. 30, 39. \MR{3466903}

\bibitem{freidgart}
J.~Gartner and M.~I. Fre{\u\i}dlin, \emph{The propagation of concentration
  waves in periodic and random media}, Dokl. Akad. Nauk SSSR \textbf{249}
  (1979), no.~3, 521--525. \MR{553200 (81d:80005)}

\bibitem{kestenpaper}
H.~Kesten, \emph{Percolation theory and first-passage percolation}, Ann.
  Probab. \textbf{15} (1987), no.~4, 1231--1271. \MR{905330}

\bibitem{kestennotes}
\bysame, \emph{First-passage percolation}, From classical to modern
  probability, Progr. Probab., vol.~54, Birkh\"auser, Basel, 2003, pp.~93--143.
  \MR{2045986}

\bibitem{kingman}
J.~F.~C. Kingman, \emph{The ergodic theory of subadditive stochastic
  processes}, J. Roy. Statist. Soc. Ser. B \textbf{30} (1968), 499--510.
  \MR{0254907}

\bibitem{KPP}
A.~N. Kolmogorov, I.G. Petrovsky, and N.S. Piskunov, \emph{Etude de l'equation
  de la diffusion avec croissance de la quantit\' e de mati\` ere et son
  application a un probl\` eme biologique}, Moskow Univ. Math. Bull. \textbf{1}
  (1937), 1--25.

\bibitem{LinZla2}
J.~Lin and A.~Zlato\v{s}, \emph{Stochastic homogenization for {F}isher-{KPP}
  reaction-diffusion equations.}, in preparation.

\bibitem{plltakisvisc}
P.-L. Lions and P.E. Souganidis, \emph{Homogenization of ``viscous''
  {H}amilton-{J}acobi equations in stationary ergodic media}, Comm. Partial
  Differential Equations \textbf{30} (2005), no.~1-3, 335--375. \MR{2131058
  (2005k:35019)}

\bibitem{majtakis}
A.~J. Majda and P.~E. Souganidis, \emph{Large-scale front dynamics for
  turbulent reaction-diffusion equations with separated velocity scales},
  Nonlinearity \textbf{7} (1994), no.~1, 1--30. \MR{1260130}

\bibitem{nadinrandom}
G.~Nadin, \emph{Critical travelling waves for general heterogeneous
  one-dimensional reaction-diffusion equations}, Ann. Inst. H. Poincar\'e Anal.
  Non Lin\'eaire \textbf{32} (2015), no.~4, 841--873. \MR{3390087}

\bibitem{nolryz}
J.~Nolen and L.~Ryzhik, \emph{Traveling waves in a one-dimensional
  heterogeneous medium}, Ann. Inst. H. Poincar\'e Anal. Non Lin\'eaire
  \textbf{26} (2009), no.~3, 1021--1047. \MR{2526414 (2010d:35201)}

\bibitem{nolenxin}
J.~Nolen and J.~Xin, \emph{Asymptotic spreading of {KPP} reactive fronts in
  incompressible space-time random flows}, Ann. Inst. H. Poincar\'e Anal. Non
  Lin\'eaire \textbf{26} (2009), no.~3, 815--839. \MR{2526403}

\bibitem{oshermerriman}
S.~Osher and B.~Merriman, \emph{The {W}ulff shape as the asymptotic limit of a
  growing crystalline interface}, Asian J. Math. \textbf{1} (1997), no.~3,
  560--571. \MR{1604922 (98m:73013)}

\bibitem{richards}
G.~D. Richards, \emph{An elliptical growth model of forest fire fronts and its
  numerical solution}, International Journal for Numerical Methods in
  Engineering \textbf{30} (1990), no.~6, 1163--1179.

\bibitem{rossi}
L.~Rossi, \emph{The {F}reidlin--{G}\"artner formula for general reaction
  terms}, Adv. Math. \textbf{317} (2017), 267--298. \MR{3682669}

\bibitem{soravia}
P.~Soravia, \emph{Generalized motion of a front propagating along its normal
  direction: a differential games approach}, Nonlinear Anal. \textbf{22}
  (1994), no.~10, 1247--1262. \MR{1279982}

\bibitem{takisfronts}
P.~E. Souganidis, \emph{Front propagation: theory and applications}, Viscosity
  solutions and applications ({M}ontecatini {T}erme, 1995), Lecture Notes in
  Math., vol. 1660, Springer, Berlin, 1997, pp.~186--242. \MR{1462703
  (98g:35010)}

\bibitem{wein}
H.~F. Weinberger, \emph{On spreading speeds and traveling waves for growth and
  migration models in a periodic habitat}, J. Math. Biol. \textbf{45} (2002),
  no.~6, 511--548. \MR{1943224 (2004b:92043a)}

\bibitem{xinflame}
J.~X. Xin, \emph{Existence of planar flame fronts in convective-diffusive
  periodic media}, Arch. Rational Mech. Anal. \textbf{121} (1992), no.~3,
  205--233. \MR{1188981 (93m:35110)}

\bibitem{bruno}
B.~Ziliotto, \emph{Stochastic homogenization of nonconvex {H}amilton-{J}acobi
  equations: a counterexample}, Comm. Pure Appl. Math. \textbf{70} (2017),
  no.~9, 1798--1809. \MR{3684310}

\bibitem{ZlaHJ}
A.~Zlato\v{s}, \emph{Stochastic homogenization for isotropic
  {H}amilton-{J}acobi equations.}, in preparation.

\bibitem{andrejhom}
\bysame, \emph{Generalized traveling waves in disordered media: existence,
  uniqueness, and stability}, Arch. Ration. Mech. Anal. \textbf{208} (2013),
  no.~2, 447--480. \MR{3035984}

\bibitem{ZlaBist}
\bysame, \emph{Existence and non-existence of transition fronts for bistable
  and ignition reactions}, Ann. Inst. H. Poincar\'e Anal. Non Lin\'eaire
  \textbf{34} (2017), no.~7, 1687--1705. \MR{3724753}

\bibitem{andrejbd}
\bysame, \emph{Propagation of reactions in inhomogeneous media}, Comm. Pure
  Appl. Math. \textbf{70} (2017), no.~5, 884--949. \MR{3628878}

\end{thebibliography}

%
\end{document}